\documentclass[paper=a4, fontsize=11pt]{scrartcl}

\usepackage[sc]{mathpazo}
\usepackage[scaled]{helvet}
\usepackage[T1]{fontenc}

\usepackage[english]{babel}															
\usepackage[protrusion=true,expansion=true]{microtype}	
\usepackage{amsmath,amsfonts,amsthm} 
\usepackage{bm} 
\usepackage{dsfont}
\usepackage[pdftex]{graphicx}	
\usepackage{enumitem}

\usepackage{enumitem}
\usepackage{algorithm}
\usepackage{algpseudocode}
\usepackage{adjustbox}
\usepackage{xr-hyper} 

\usepackage{url}
\usepackage{colortbl}
\usepackage{booktabs}
\usepackage{tabularx} 
\usepackage{float} 
\usepackage{multirow}%
\usepackage[dvipsnames]{xcolor}
\usepackage{fullpage}
\usepackage{natbib}
\usepackage{subcaption}  
\usepackage{graphicx}
\usepackage{tabularx} 
\usepackage{float} 
\usepackage{caption}

\usepackage{xr-hyper} 

\usepackage[colorlinks=true,citecolor=blue,urlcolor=blue]{hyperref}
\usepackage[capitalise]{cleveref}
\usepackage{natbib}

\usepackage{graphicx}
\usepackage[dvipsnames]{xcolor}
\usepackage[english]{babel}
\usepackage{amsmath,amsfonts,amsthm} 
\usepackage{enumitem}
\usepackage{dsfont}
\usepackage{amsmath,amssymb}

\newcommand{\diam}{\operatorname{diam}}

\newcommand{\pr}{\mathbb{P}}

\newcommand{\argmin}{\operatornamewithlimits{\arg\min}}
\newcommand{\ind}{\mathds{1}}

\newcommand{\PP}{\mathbb{P}}
\newcommand{\RR}{\mathbb{R}}

\newcommand{\EE}{\mathbb{E}}

\newtheorem{theorem}{Theorem}
\newtheorem{definition}[theorem]{Definition}

\newtheorem{corollary}[theorem]{Corollary}
\newtheorem{proposition}[theorem]{Proposition}
\newtheorem{lemma}[theorem]{Lemma}
\newtheorem{example}{Example}

\usepackage[sc]{mathpazo}
\usepackage[scaled]{helvet}
\usepackage[T1]{fontenc}

\usepackage[english]{babel}															
\usepackage[protrusion=true,expansion=true]{microtype}	
\usepackage{amsmath,amsfonts,amsthm} 
\usepackage{bm} 
\usepackage{dsfont}
\usepackage[pdftex]{graphicx}	
\usepackage{enumitem}

\usepackage{url}
\usepackage{colortbl}
\usepackage{booktabs}
\usepackage{tabularx} 
\usepackage{float} 
\usepackage{multirow}%
\usepackage[dvipsnames]{xcolor}
\usepackage{fullpage}
\usepackage{natbib}
\usepackage{graphicx}
\usepackage{tabularx} 
\usepackage{float} 
\usepackage{caption}

\usepackage{xr-hyper} 

\usepackage[colorlinks=true,citecolor=blue,urlcolor=blue]{hyperref}
\usepackage[capitalise]{cleveref}
\usepackage{natbib}

\definecolor{offwhite}{RGB}{255,250,240}
\definecolor{gray}{RGB}{155,155,155}
\definecolor{foreground}{RGB}{80,80,80}
\definecolor{background}{RGB}{255,255,255}
\definecolor{title}{RGB}{89,132,212}
\definecolor{subtitle}{RGB}{255,199,0}
\definecolor{hilit}{RGB}{248,117,79}
\definecolor{vhilit}{RGB}{255,111,207}
\definecolor{lolit}{RGB}{200,200,200}
\definecolor{lit}{RGB}{255,199,0}
\definecolor{mdlit}{RGB}{71,106,170}
\definecolor{link}{RGB}{248,117,79}
\definecolor{subcol}{RGB}{207,216,233}

\definecolor{darkraspberry}{rgb}{0.53, 0.15, 0.34}

\newcommand{\horrule}[1]{\rule{\linewidth}{#1}} 	

\title{
		\vspace{-1in} 	
		\usefont{OT1}{bch}{b}{n}
		\horrule{0.5pt} \\[0.4cm]
		\Large On the pointwise and sup-norm errors for local regression estimators
		\\
		\horrule{2pt} \\[0.1cm]
}

\date{\large \today}

\author{ \textbf{Jérémy Bettinger, François Portier, Adrien Saumard} \\
\large \href{mailto:jeremy.bettinger@ensai.fr}{jeremy.bettinger@ensai.fr} ; \href{mailto:francois.portier@ensai.fr}{francois.portier@ensai.fr} ; \href{mailto:adrien.saumard@ensai.fr}{adrien.saumard@ensai.fr} \\
\large Department of Statistics, \\
\large University of Rennes, ENSAI, CNRS, CREST-UMR 9194, F-35000 Rennes, France }

\begin{document}

\maketitle

\begin{abstract}

 In this paper, we analyze the behavior of various non-parametric local regression estimators, i.e. estimators that are based on local averaging, for estimating a Lipschitz regression function at a fixed point, or in sup-norm. 
 
 We first prove some deviation bounds for local estimators that can be indexed by a VC class of sets in the covariates space. We then introduce the general concept of shape-regular local maps, corresponding to the situation where the local averaging is done on sets which, in some sense, have ``almost isotropic'' shapes. On the one hand, we prove that, in general, shape-regularity is necessary to achieve the minimax rates of convergence. On the other hand, we prove that it is sufficient to ensure the optimal rates, up to some logarithmic factors. 
 
 Next, we prove some deviation bounds for specific estimators, that are based on data-dependent local maps, such as nearest neighbors, their recent prototype variants, as well as a new algorithm, which is a modified and generalized version of CART, and that is minimax rate optimal in sup-norm. In particular, the latter algorithm is based on a random tree construction that depends on both the covariates and the response data. For each of the estimators, we provide insights on the shape-regularity of their respective local maps. Finally, we conclude the paper by establishing some probability bounds for local estimators based on purely random trees, such as centered, uniform or Mondrian trees. Again, we discuss the relations between the rates of the estimators and the shape-regularity of their local maps.

 
\end{abstract}


%

%

\begin{abstract}
 We introduce the concept of shape-regular regression maps as a framework to derive optimal rates of convergence for various non-parametric local regression estimators. Using Vapnik-Chervonenkis theory, we establish upper and lower bounds on the pointwise and the sup-norm estimation error,  even when the localization procedure depends on the full data sample, and under mild conditions on the regression model. Our results demonstrate that the shape regularity of regression maps is not only sufficient but also necessary to achieve an optimal rate of convergence for Lipschitz regression functions. To illustrate the theory, we establish new concentration bounds for many popular local regression methods such as nearest neighbors algorithm, CART-like regression trees and several purely random trees including  Mondrian trees.
\end{abstract}

\section{Introduction}\label{s1}
%
%
Consider the standard regression problem where the goal is to estimate the regression function of a random variable $Y\in \mathbb R$ given the covariates vector $X\in \mathbb R^d$, defined as $ g(x) := \mathbb E [ Y| X= x]$, $x\in \mathbb R ^d$. One leading approach, called \textit{local regression} or \textit{local averaging}, consists in averaging the observed response variables, restricted to covariates that lie in a small region of the domain $\mathbb R ^d$. Local regression methods include kernel smoothing regression \cite{nadaraya1964estimating}, nearest neighbors algorithm \cite{fix1989discriminatory,cover1968estimation} and regression trees or, more generally, partitioning regression estimators \cite{breiman1984classification,nobel}. We refer to the books
\cite{devroye96probabilistic,gyorfi2006distribution} for an overview of local regression methods and to \cite{biau2015lectures} for a precise theoretical account on the nearest neighbors algorithm.

%
%
Concerning the estimation problem, when the error is measured in terms of the mean squared error ($L_2$ error), the optimal convergence rates are known \cite{stone1982optimal} and depend on the smoothness of the regression function $g$. Whether or not these convergence rates are achieved often serves as a theoretical baseline to evaluate the accuracy of local regression methods. For example, a Lipschitz function $g$ can only be approximated at the rate $ n ^{- 1/ (d+2)  }$ in general, when $n $ independent observations are given. Many of the above estimators are known to achieve optimal convergence rates. The nearest neighbors, the Nadaraya-Watson and the fixed partitioning (histogram) regression estimators are all optimal for Lipschitz functions (as well as for twice differentiable functions for the first two listed methods), as explained in \cite{biau2015lectures},  \cite{tsy_08} and Chapter 4 of \cite{gyorfi2006distribution}, respectively. Furthermore, the Nadaraya-Watson \cite{einmahl2000,gine+g:02} and the nearest neighbors \cite{jiang2019non,portier2021nearest} estimators are both known to achieve a rate of sup-norm convergence that is of the same order as the $L_2$-rate (up to a logarithmic term). 

Other local regression estimators based on purely random trees are of interest \cite{biau2016random}, despite the independence of the leaves with respect to the original data, notably because of their ability to explain certain patterns in the success or failure of different tree constructions and to illustrate the success of random forest regression. Among purely random trees, Mondrian trees, as introduced in \cite{lakshminarayanan2014mondrian}, indeed achieve the optimal convergence rate \cite{mondrian} for Lipschitz functions, while, in contrast, centered trees do not reach it \cite{biau2012analysis,klusowski2021sharp}. 

It is also worth highlighting that partitioning the space with Voronoi cells based on another sample (independent of the original data)  can also lead to the optimal rate of convergence \citep{gyorfi2021universal,kerem2023error}. This type of approach is known as nearest neighbor-based prototype learning rules (as described in \cite{devroye96probabilistic}, Chapter 19). They have recently receive much attention \cite{NIPS2014_8c19f571,NIPS2017_934815ad,hanneke2021universal,gyorfi2021universal,xue2018achieving,kerem2023error} due to (a) their ability to compress data - offering interesting ways to reduce time complexity and memory use - and (b) their surprisingly good statistical performances. A prototype learning algorithm follows from the construction of some ``prototype data'', based on the initial sample, and training a learning rule based on these prototypes. Hence, good prototype learning algorithms would require a small number of prototypes while maintaining good statistical performance. One first remarkable fact is that prototype learning acts as regularization: the overfitting $1$-NN algorithm may be consistent when applied to a good prototype sample \cite{NIPS2017_934815ad,hanneke2021universal,gyorfi2021universal,xue2018achieving,kerem2023error}. Another notable fact is that some prototype rules are universally consistent in general metric spaces, while $k$ -NN is not \cite{NIPS2017_934815ad,hanneke2021universal,gyorfi2021universal}.

%
%
Despite the many existing results available for the Nadaraya-Watson and nearest neighbors regression estimators, and also fixed or purely random partitioning regression rules, only a few are known about local regression based on data-dependent partitions, such as the well-known CART regression tree \cite{breiman1984classification}. Such an algorithm is indeed much harder to analyze mathematically.
First results on data dependent partitions can be found in \cite{stone1977consistent}, but they are restricted to cases where the partition depends only on the covariates, as in nearest neighbors regression or for statistically equivalent blocks \cite{anderson}. 
More advanced results, that are valid for general data dependent partitioning estimators, are obtained in \cite{gordon1980consistent,breiman1984classification,nobel}, where conditions are given to ensure almost sure $L_2$-consistency. The typical assumptions that are required in the previous works include (i) large enough points in each partition element and (ii) small diameter, while having (iii) a reduced complexity on the partition elements. Note also that Theorem 1 in \cite{scornet2015consistency} can be applied to CART regression algorithm and gives sufficient conditions for the $L_2$-consistency.

Beyond consistency, few is known about the convergence rates of data-dependent, CART-like regression tree estimators. Recent studies  \cite{chi2022asymptotic, mazumder2024convergence} have obtained convergence rates for the $L_2$-error under the so-called \textit{sufficient impurity decrease} (SID) condition, that is directly linked to the behavior of the precise splitting rule of CART in the regression context. The rate of convergence obtained depends on a parameter - denoted $\lambda$ in \cite{ mazumder2024convergence} - quantifying the strength of the SID condition, and it is not \textit{a priori} easy to discuss the rate optimality. Nonetheless, it is shown in  \cite{ mazumder2024convergence} that for a univariate linear regression function, the rate obtained through the SID condition is actually optimal. A specific class of additive regression functions achieving a particular smoothness assumption called the ``locally reverse Poincaré inequality'' is provided in  \cite{ mazumder2024convergence}, satisfying the SID condition. In another direction, the recent negative results in \cite{cattaneo2022pointwise} show that CART regression can be sub-optimal, and even inconsistent, for the pointwise - and also uniform - estimation error. Such phenomenon does not occur when focusing on the $L_2$-error, but as highlighted in \cite{cattaneo2022pointwise}, pointwise convergence of decision trees is also essential for reliability of the methodologies developed in some causal inference and multi-step semi-parametric settings for instance.

In this work, we develop a theory for obtaining pointwise and uniform rates of convergence for a large class of local regression estimators, that includes previously mentioned partitioning estimators. More precisely, in a random design regression with heteroscedastic sub-Gaussian noise framework, the theory allows the localization method to be general, as it may depend on a different source of randomness (as for the purely random tree) or on the covariates sample (as for nearest neighbors) and even on the full regression sample (as in CART). 

Instead of studying the integrated $L_2$-error, our approach deals with the pointwise and uniform estimation errors recently put forward in the literature \cite{cattaneo2022pointwise}, for which we obtain a general probability upper bound (Theorem \ref{th:general}). To prove such a result, we proceed with a decomposition of the pointwise estimation error into the sum of a variance term (scaling as the inverse of the square root of the number of covariates in the partition elements) and a bias term (scaling as the diameter of the partition elements). We point out that the major advantage of focusing on the pointwise error, compared to the $L_2$-error, is that it allows the control of the variance and bias terms though the use of the Vapnik dimension of a class containing the \textit{elements} of the random partition, instead of having to control the combinatorial size of the class of the \textit{entire} partitions themselves as in \cite{lugosinobel}.

Next, we introduce the notion of \textit{shape regularity} by imposing a simple relationship between the Lebesgue volume and the diameter of the localizing set. The major interest of this simple geometric property is that it turns out to be necessary (Proposition \ref{contre_ex}) and sufficient (Theorem \ref{main_result_localizing_map}) for obtaining optimal rates - up to logarithmic factors - for the pointwise and uniform estimation errors. We then discuss how several tree constructions, including purely random trees such as uniform and Mondrian trees, satisfy - or not - the shape regularity condition, allowing to obtain  - or not - optimal rates of convergence. In addition,  the shape regularity allows to recover and slightly extend some results pertaining to the nearest neighbors literature \cite{jiang2019non,portier2021nearest} as well as to establish new guarantees for two well-known prototype-based methods: \textit{Proto-NN} \citep{gyorfi2021universal} and \textit{OptiNet} \citep{NIPS2014_8c19f571,NIPS2017_934815ad,hanneke2021universal,kerem2023error}.

Interestingly, Proto-NN and OptiNet both have universality properties in general metric spaces, while the variation proposed in \cite{xue2018achieving}, called proto-$k$-NN, fails in being universally consistent for similar reasons as standard $k$-NN would fail, as explained in \cite{cerou2006nearest}. In finite dimension, convergence rates have been obtained for OptiNet \cite{kerem2023error} and proto-$k$-NN \cite{xue2018achieving,gyorfi2021universal}. They match the minimax optimality rate for Lipschitz functions.
To the best of our knowledge, such results concerning Proto-NN have not yet been obtained and the problem seems to be nontrivial, as pointed out by \cite{gyorfi2021universal}:``\textit{Obtaining convergence rates for the universally consistent Proto-NN classifier (...) is currently an open research problem}''. Our result, a deviation error bound for Proto-NN, shows that the minimax optimal rate is also achieved for Proto-NN.

Finally, we obtain a deviation inequality on the uniform estimation error of CART-like regression trees, grown by ensuring a minimal number of covariates in the tree leaves and by following a simple rule which maintains the shape regularity of the resulting localizing sets. In the case of partitions made of hyper-rectangles, such as for CART-like algorithms, the shape-regularity condition reduces to a control of the largest side length of the localizing set by its smallest side length. Recent results obtained in \cite{cattaneo2022pointwise} indeed tend to indicate that such rules additions are likely to be unavoidable to ensure good pointwise convergence rates of CART-like regression trees. 

It is worth noting that our approach substantially differs from the use of the SID condition described earlier. The latter indeed ensures convergence rates for the $L_2$-error and is highly linked to the precise cost in the splitting rule of CART, defined through the so-called impurity gain. Moreover, the SID condition is expressed through the behavior of the unknown regression function and covariates distribution, and cannot hold for any regression function. In contrast, our shape-regularity condition does not depend on the regression function $g$, neither on the covariates distribution, and only imposes a restriction that may be effective with any cost function involved in the splitting rule. This makes our shape regularity condition easy to guarantee in practice as illustrated in Algorithm 1 (see Section \ref{s53}), where a general cost function is used to build the tree.

The outline is as follows. We state in Section \ref{s2} some necessary background and formulate the setting of local regression map estimators. Section \ref{s30} then gives a first deviation inequality for local regression map estimators. We introduce in Section \ref{s3} the shape regularity conditions and their properties. Section \ref{s5} is dedicated to pointwise and uniform convergence bounds for data-dependent regression maps, including nearest neighbors, Proto-NN, OptiNet and CART-like trees.  Finally, Section \ref{sec_PRT}  includes new positive and negative results about some classical purely random trees. All the mathematical proofs are given in the Appendix.

\section{Mathematical background}\label{s2}

\subsection{Regression set-up}\label{s21}

Let $(X,Y)$ be a random vector with probability distribution $ P$ on $\mathbb R^d \times \mathbb R$, where $d\geq 1$ is the dimension of covariates vector $X\in S_X \subset\mathbb R^d$ and $Y \in \mathbb R$ is the output variable.  The goal is to estimate the conditional expectation $x\mapsto g(x) = \mathbb E [ Y|X= x]$, $x\in S_X$. The quality of the estimation of the function $g$ by an estimator $\hat g$ will be assessed with the help of the uniform norm defined as $\sup_{x\in S_X} | \hat g(x) - g(x) | $. 
For a fixed $x\in S_X$, we also address the estimation error of the value $g(x)$ through the analysis of the deviations of the quantity $| \hat g(x) - g(x) | $. 

The following assumption on $ P$ will be key in this work and, roughly speaking, amounts to assume that the noise $\varepsilon = Y  - g(X)$ in the regression model is lightly tailed.

\begin{enumerate}[label=(E), wide=0.5em,  leftmargin=*]
\item \label{cond:epsilon} 
The random variable $\varepsilon$ is sub-Gaussian conditionally on $X$ with parameter $\sigma^2$. That is, $\mathbb E [\varepsilon |X ]= 0$ and for all $\lambda \in \mathbb R$, $$ \mathbb E [ \exp( \lambda\varepsilon  ) |X ] \leq \exp\left( \frac{ \lambda ^2}{   2\sigma^2 }\right).$$
\end{enumerate}
Note that under assumption \ref{cond:epsilon}, the noise term $\varepsilon$ is squared integrable and it is allowed to depend on the covariates $X$.  In particular, the noise is \textit{heteroscedastic}, with a uniform upper bound on its conditional variance: almost surely, we have $\mathbb{E}[\varepsilon^2\vert X] \leq \sigma^2$. A more restrictive assumption is when $\varepsilon$ is independent of $X $ and sub-Gaussian with parameter $\sigma^2$.


In this work, all the estimators will be based on the sample $\mathcal{D}_n=\left\{(X_i,Y_i) \,  : \, i=1,\ldots,n\right\}$ which satisfies the following assumption.

\begin{enumerate}[label=(D), wide=0.5em,  leftmargin=*]
\item \label{cond:D} 
The random variables $\{(X,Y), (X_i,Y_i)_ {i=1,\ldots,n} \}$ are independent and identically distributed with common distribution $  P$. 
\end{enumerate}

To conclude this section, let us introduce the notation $P^X$ as the marginal distribution of $X$. Set also $\varepsilon_ i := Y_i - g(X_i) $ for each $i=1,\ldots, n$.

\subsection{Local regression maps}\label{s22}

We consider general local regression estimators using the concept of local maps so as to include regression trees and partitioning estimators but also the nearest neighbors regression rule. Let $\mathcal B(S_X)$ denote the Borel $\sigma$-algebra on $S_X$.

\begin{definition}
    A local map for a variable $X$ is a mapping $ \mathcal V : S_X \to \mathcal B(S_X) $ such that for all $x\in S_X$, $x\in \mathcal V (x)$. 
\end{definition}

We emphasize here that this work focuses on continuous covariates, therefore all local maps will have their images to sets with positive Lebesgue measure. Let us also stress out that similar maps were introduced in \cite{nobel}, where they are however restricted to partition based estimator.
For any local map $\mathcal V$, the associated regression estimator is given by
\begin{equation*}
    \forall x \in S_X, \quad \hat g_{\mathcal V}(x) = \frac{\sum_{i=1} ^ n Y_i\mathds 1 _{ \mathcal V(x) }(X_i) }{\sum_{i=1} ^ n \mathds 1 _{ \mathcal V(x) }(X_i)  }\; , 
\end{equation*}
with the convention $0/0 = 0$, which is in force in the subsequent work.
Local maps  $\mathcal V $ depending on the sample $(X_1,Y_1),\ldots , (X_n, Y_n)$ are of particular interest. This is indeed the case for some adaptive tree constructions, as well as for the nearest neighbors algorithm.  

The local regression map framework is particularly interesting because it includes a variety of different methods, e.g., fixed partitioning,  purely random trees, nearest neighbors, and CART-like constructions, and each method induces a particular dependence structure when creating the partition.

\begin{example}[fixed hyper-rectangles partition]\label{ex1}
    The most simple case for the dependence structure of the local map is when the partition is fixed, not random. Suppose $S_X= (0,1]^d$. For each coordinate $k =1,\ldots, d$, consider the collection $0 = u_0^{(k)} < u_1^{(k)} < \ldots < u_{N_k}^{(k)} = 1$. This allows to introduce a partition of $S_X$ made of $ \prod_{k=1} ^d N_k $ elements defined as $ V_{i_1,\ldots, i_d} =  \prod_{ k = 1 }^d (u_{i_k}^{(k)}, u_{i_k + 1 }^{(k)} ]$ for each d-uplet $(i_1,\ldots, i_d)$ satisfying $i_\ell\in \left\{0,\ldots,N_\ell-1 \right\}$ for $\ell\in \left\{1,\ldots,d \right\}$. Note that each $V_{i_1,\ldots, i_d}$ has a positive Lebesgue measure $ \prod_{ k = 1 }^d (u_{i_k + 1 }^{(k)} - u_{i_k}^{(k)} )$.
\end{example} 

\begin{example}[purely random trees]
In contrast to Example \ref{ex1}, a \textit{purely random tree} construction, as described in \cite{arlot2014analysis} and initially introduced in \cite{breiman2000some},  consists in using some randomness that is independent of the observed sample. It includes centered (resp. uniform) trees, for which the split direction is uniformly distributed along the space coordinates and the split location of the selected side is at the center (resp. uniformly distributed).  It also includes Mondrian trees \citep{lakshminarayanan2014mondrian}, where the split direction is selected at random depending on the shape - i.e. side lengths - of the leaf.  
\end{example}

\begin{example}[nearest neighbors regression]\label{ex3}
    Nearest neighbors algorithm induces a Voronoi-like partition, which dependence structure is different from the one of purely random trees, since the nearest neighbors partition depends on the data through the location of the covariates in the space. The $k$-nearest neighbors ($k$-NN) estimator (see \cite{biau2015lectures} for a recent textbook) is defined, for each $x\in S_X$, as the average responses among the $k$-nearest neighbors to point $x$. As such, we have
$$ \hat g_{NN}(x) =  \frac{1}{k }  \sum_{i=1}^n  Y_i\ind_{B(x, \hat \tau_k(x))} (X_i ) ,$$
where $\hat \tau_k(x)$ is the so-called $k$-NN radius defined as the smallest radius $\tau>0 $ such that $\sum_{i=1}^n  \ind_{B(x,  \tau ) } (X_i ) \geq k$.
Note that here the local map is $\mathcal V (x)  = B(x, \hat \tau_k(x))$ and therefore  depends on $X_1,\ldots, X_n$.
\end{example}

\begin{example}[{CART-like trees}]
Regression trees are a class of partition based estimators where the partition is recursively built, and made of hyper-rectangles. Therefore, they are part of the local map framework, just as examples 1 and 2 above. Usual regression trees are grown sequentially by splitting stage-wise each (adult) leaf into two (children) leafs. In most cases, as in CART regression \citep{breiman1984classification}, each cell division results from splitting along one single variable according to a data-based criterion. This precise step is crucial as it allows to adapt the partition to the prediction problem. For instance, if one variable is not significant then it must be better not to split with respect to it. This enables to obtain a flexible regression estimator, which behaves well in many problems even when the dimension $d$ is rather large. The fact that the resulting partition depends on the full data (including the response) is however problematic for the theory since in this case, the local averaging estimator is not a sum over independent random variables, thus prohibiting a direct application of concentration inequalities for sums of independent observations. Finally, it is worth mentioning that CART regression trees are the ones that are usually combined in the standard Random Forest regression algorithm, as introduced in \cite{breiman2001random}.
 \end{example}

\section{A deviation bound for local map estimators}\label{s30}
Considering the local map estimator definition given in Section \ref{s22}, the first step in analyzing its pointwise error is standard, and consists in considering the following bias-variance decomposition,
\begin{equation*}
    \hat g_{\mathcal V}(x) - g(x) = \underbrace{\frac{\sum_{i=1}^n \varepsilon_i \mathds 1 _{ \mathcal V(x) }(X_i)}{\sum_{j=1}^n \mathds 1 _{ \mathcal V(x) }(X_j)}}_{\text{variance term}}+\underbrace{\frac{\sum_{i=1}^n \left(g(X_i) - g(x)\right) \mathds 1 _{ \mathcal V(x) }(X_i)}{\sum_{j=1}^n \mathds 1 _{ \mathcal V(x) }(X_j)}}_{\text{bias term}}.
\end{equation*}
The section is divided into two parts. We first give some preliminary concentration bound for the variance term, that is free from any restriction on the probability distribution of the covariates. Then we use it in the regression framework in order to obtain some concentration bound
on the estimation error.

\subsection{A deviation bound for the variance term}\label{3.1}

The \textit{shattering coefficient}, as introduced in Vapnik's seminal work \cite{vapnik2015uniform} and  detailed for example in \cite{wellner1996},
is key to obtain upper bounds on certain empirical sums indexed by sets or functions. Let $\mathcal A$ be a collection of subsets of a set $ S$. Given an arbitrary collection $z =( z_1,\ldots, z_n)$ of distinct points in $  S$, consider the collection of $\mathbb R^n$-points
$ \ind _ {\mathcal A} (z)$ defined as $  \{ (\ind _ A (z_1) \ldots, \ind_A (z_n )) : A\in \mathcal A \}\subset \{0,1\}^n $. 
We have that $|\ind _ {\mathcal A} (z)  | \leq 2^n$ and when $   |\ind _ {\mathcal A} (z) |  = 2^n$ we say that $z$ is shattered by $\mathcal A$. An important quantity is then
 $$\mathbb S_\mathcal A(n) : =  \sup_{z\in \mathbb R^n} | \ind _ {\mathcal A} (z) | $$
which is called the shattering coefficient. 

We now provide a VC-type inequality tailored to the analysis of the variance term for local regression estimators. Recall that, by convention, $0/0=0$.

\begin{theorem}\label{1}
Let $n\geq 1$ and $\delta \in (0,1) $. Suppose that \ref{cond:epsilon} and \ref{cond:D} are fulfilled and that $\{\mathcal V (x) \, :\, x\in \mathbb R^d\} \subset \mathcal A$, a deterministic collection of sets in $\mathbb R^d$. The following inequality holds with probability at least $1 - \delta$,
$$\sup_{x\in \mathbb R^d} \dfrac{\sum_{i=1}^n \varepsilon_i \mathds 1 _{ \mathcal V(x) }(X_i)}{\sqrt{\sum_{j=1}^n \mathds 1 _{ \mathcal V(x) }(X_j)}} \leq \sqrt{2 \sigma^2 \log\left( \frac{\mathbb S_ {  \mathcal A   } (n)}{\delta} \right)}.$$
\end{theorem}
Note that in Theorem \ref{1} above, only an upper bound is given but a lower bound is also valid, since the same holds true when each $\varepsilon _ i $ are replaced by $ -\varepsilon_i$. Moreover, combining such inequalities through a union bound gives a result for the supremum of the absolute value.

\subsection{A pointwise error bound}\label{s23}

We now state a general deviation bound on the uniform error of local regression map estimators with finite Vapnik-Chervonenkis (VC) dimension. The VC dimension is defined as 
\begin{align*}
  vc(\mathcal A)  &= \max \{ n\geq 1 \, :\, \mathbb S_\mathcal A(n) 
 = 2^n  \}. 
\end{align*}
As a consequence, the fact that all given $z_1,\cdots , z_{v+1} $ points cannot be shattered is equivalent to the fact that the VC dimension is smaller than $v$. The reason why the VC dimension is appropriate for controlling the complexity of classes of sets is perhaps explained by the Sauer's lemma (see \cite{lugosi2002pattern} for a proof) which states that
$\mathbb{S}_{\mathcal{A}}(n)  \leq  \sum_{i=0} ^{vc(\mathcal{A})} \binom{n}{i}.$ 
  An interesting consequence of Sauer's lemma is that 
$\mathbb S_\mathcal A(n) \leq (n+1)^{vc(\mathcal A) }.$

As established in \cite{wenocur1981some}, previous examples include the class of cells $(-\infty , t]\subset \mathbb R^d$, having VC dimension equal to $d$, or the class $( s , t]$, $s,t\in \RR^d$, of VC dimension equal to $2d$. In addition, the class of balls in $\mathbb R^d$ has dimension equal to $d+1$. 
\begin{definition}
 A local map $\mathcal V$ is said to be VC when there exists $\mathcal A$, a fixed VC collection of sets in $\mathbb R^d$, such that $\{\mathcal V (x) \, :\, x\in S_X\} \subset \mathcal A$.
\end{definition}
Let us further define some quantities that will be instrumental in our analysis. For any set $V$, its diameter is given by the formula
$$ \diam (V)  = \sup_{(x,y)\in V\times V} \|x-y\|_2,$$
where $\|x\|_2^2 = \sum_{k=1} ^d x_k^2$. A real function $h$ on $S_X$ is called $L$-Lipschitz as soon as $ | h(x) - h(y) |\leq L \|x-y\|_2$ for all $(x,y)\in S_X^2$.
 Define also the local Lipschitz constant $ L(V)$ of $h$ over $V\subset S_X$ as the smallest constant $L>0$ such that, for all $(x ,y)$ in $V^2$,
$$     |h(x) - h(y) | \leq L  \|  x-y\|_2.$$
For a $L$-Lipschitz function, it holds $L(V)\leq L$ for any set $V\subset S_X$. In what follows, we will consider regression functions that are Lipschitz over the domain $S_X$.

\begin{enumerate}[label=(L), wide=0.5em,  leftmargin=*]
  \item \label{cond:reg4} The function \( g: x \mapsto \mathbb E [ Y|X= x]\) is $L$-Lipschitz on \( S_X \). 
\end{enumerate}

The next probability error bound is valid for local map estimators, with a general VC local map, that may for instance depend on the sample.

\begin{theorem}\label{th:general}
Let $n\geq 1$ and $\delta \in (0,1/2) $. Under \ref{cond:epsilon}, \ref{cond:D} and \ref{cond:reg4}, suppose that the local map is VC with dimension $v$. We have, with probability at least $1 - 2 \delta $, for all $x\in S_X$,
        \begin{align*}
       |  \hat g_{\mathcal V}(x) - g(x)| \leq \sqrt{\frac{ 2 \sigma^2 \log\left( \frac{ (n+1) ^v }{\delta} \right)}{ n P_n^X(\mathcal V (x))  } } + L(\mathcal V(x)  ) \diam  (\mathcal V(x) )  .
    \end{align*}
    where for any $A\in \mathcal B(S_X)$, $n P_n^X(A)  = \sum_{i=1} ^n \ind _A(X_i) $.
\end{theorem}

An alternative approach proposed in \cite{lugosinobel,nobel} as well as in \cite{devroye96probabilistic}, see Theorem 21.2 therein, follows from a uniform control over all resulting partitions, implying consistency results for sums over all partition elements. In Theorem \ref{th:general}, our approach is substantially different, since by considering the pointwise or sup-norm error,  the complexity term comes from the elements of the partition only. In addition, Theorem \ref{th:general} above might be compared with Theorem 6.1 in \cite{devroye96probabilistic}, which is suitable to either non-random or purely random (i.e., independent of the sample) data partitioning \citep{JMLR:v9:biau08a}. While Theorem \ref{th:general} is valid for data dependent partitions, we recover almost sure consistency by imposing two conditions that are similar to those required in Theorem 6.1 of \cite{devroye96probabilistic}, namely $\diam  (\mathcal V(x) )  \to 0 $ and $ nP_n^X( \mathcal V(x) ) / \log(n )  \to 0$. Depending on whether the previous conditions hold uniformly in $x$ or for a given $x$, the consistency, uniform or pointwise, of the local map regression estimator can thus be obtained.

\section{Shape regularity}\label{s3}

We describe here the minimal mass assumption, concerning the distribution of the covariates $P^X$. We then introduce the concept of shape regularity for local maps.

\subsection{Minimal mass assumption}\label{4.1}

The next minimal mass assumption allows us to obtain an estimate for $P_n^X(\mathcal V(x)) $, which appears in the upper bound stated in Theorem \ref{th:general}.

\begin{enumerate}[label=(X), wide=0.5em,  leftmargin=*]
\item  \label{cond:density_X} 
For the local map $\mathcal V$ on $S_X$, there exists a function $\ell : x \mapsto \ell(x)> 0$ such that, almost surely, for all $x\in S_X$, 
$$ P^X(\mathcal V(x) )  \geq  \ell(x) \lambda (\mathcal V(x)  )  ,$$
where $\lambda$ stands for Lebesgue measure on $\mathbb R^d$.
\newcounter{nameOfYourChoice}
\setcounter{nameOfYourChoice}{\value{enumi}}
\end{enumerate}

Note that assumption \ref{cond:density_X} is easily satisfied when $X$ has a density $f_X$ bounded from below by a constant $b > 0$, by choosing $\ell(x) = b$ (see Sections \ref{s61}, \ref{s51}, \ref{s53} for more precise examples).
Moreover, note that the minimal mass assumption \ref{cond:density_X} is defined with respect to a specific local map $\mathcal{V}$ that we do not recall explicitely, and that will always refer in the following to the natural local map associated to the considered estimator. 

The minimal mass assumption is quite flexible, as it can be checked for the local maps naturally involved in the $k$-NN regression estimators as well as CART-like trees. Indeed, in both cases, the  minimal mass assumption can be obtained by checking a more restrictive version involving some particular class of sets: balls (for nearest neighbors) and hyper-rectangles (for CART-like trees). We refer to Sections \ref{s61} and \ref{s53} respectively for more details.  

The following definition ensures that each element of the local map contains enough points. 

\begin{definition}
A VC local map  $x\mapsto  \mathcal V(x)$ with dimension $v>0$ is called $(\delta, n)$-large whenever, for all $x\in S_X$, almost surely,
    \begin{align*}
  n \max (P_n^X (\mathcal V(x))  ,     P^X(\mathcal V(x)) )  \geq 8 \log\left(\dfrac{4 (2n+1) ^v }{\delta} \right).
\end{align*}
\end{definition}

Note that the latter inequality is easy to check in practice, as it suffices to make sure that enough data points are in each element of the local map.

\begin{theorem}
    \label{th2:general}
   Let $n\geq 1$ and $\delta \in (0,1/3) $.  Under \ref{cond:epsilon}, \ref{cond:D}, \ref{cond:reg4} and \ref{cond:density_X}, suppose that the local map is VC with dimension $v$ and is $(\delta, n)$-large, then we have with probability at least $1-3\delta$, for all $x\in S_X$,
\begin{align*}
        | \hat g_{\mathcal V}(x) - g(x)| \leq  \sqrt{\frac{ 3 \sigma^2 \log\left( \frac{(n+1)^v}{\delta} \right) }{n  \ell(x) \lambda ( \mathcal V (x) )   }} + L(\mathcal V(x) ) \diam (\mathcal V(x) ).
    \end{align*}
\end{theorem}

The previous result differs from the one of Theorem \ref{th:general}, as the bound no longer depends on the number of data points in the associated local set, but instead on its Lebesgue volume. Together with the diameter, these two quantities will appear in the definition of the $\gamma$-shape regularity, so as to minimize the latter upper bound and therefore, to attain optimal rates of convergence for the underlying regression problem.


\subsection{Shape-regular sets}\label{4.2}

A key concept is now introduced, which will help characterize the convergence rates of the local map regression estimators. As established in Theorem \ref{th2:general}, under the minimal mass assumption, the quantity $ |\hat g_{\mathcal V}(x) - g (x) |$ is bounded by
 $  \sqrt { 1 / ( n \lambda (\mathcal V (x)) )    }   +   \diam (\mathcal V (x))  $, up to constants and log terms. Theorem \ref{th2:general} thus allows us to understand that a trade-off between the volume and the diameter must be achieved to reach optimal rates.
In this regard, first note that the volume cannot be greater than the diameter to the power $d$, as we always have $ \lambda (\mathcal V(x) )  \leq \diam ( \mathcal V(x) )^d $. 
Incorporating that constraint when optimizing the previous bound 
leads to $\lambda (\mathcal V(x) ) = \diam ( \mathcal V(x) )^d = n^{-d/(d+2)}$, which is the optimal rate in our regression problem. In contrast, if $ \diam ( \mathcal V(x) )^d = \gamma_n \lambda (\mathcal V(x) )   $ with $\gamma_n \to \infty$, then the bound of Theorem \ref{th2:general} gives a slower, suboptimal convergence rate. This reasoning motivates the introduction of the following notion of shape-regularity.

\begin{definition}\label{def:gamma_regular}
For $\gamma>0 $, a set $ V $ is called $\gamma$-shape-regular ($\gamma$-SR) if
$\diam(V)^d \leq \gamma \lambda (V)$.
\end{definition}
\noindent The previous condition can be interpreted as a volume condition: the volume of $V$ should be of the same order as the volume of the smallest ball containing $V$. Roughly speaking, the shape of $V$ is not that different from that of a ball or a hypercube, i.e. $V$ is ``almost isotropic''. Moreover, it does not depend on the covariates density, making it easy to check in practice. 
 
We provide now an alternative to Definition \ref{def:gamma_regular}, specifically designed for local maps valued in the set of hyper-rectangles. 
For any hyper-rectangle $A\subset  S_X$, let $h_-(A)$ and $h_+(A)$ denote the smallest and largest side length, respectively.
\begin{definition}\label{def:beta_regular}
For $\beta>0 $, a hyper-rectangle $ A$ is called $\beta$-shape-regular ($\beta$-SR) if 
$ h_+ (A)  \leq \beta h_-(A)$.
\end{definition}
\noindent It is easily seen that when a set $V$ is an hyper-rectangle, the $\gamma$-SR property is related to $\beta$-SR. This is the subject of the following proposition. 
\begin{proposition}\label{link_beta_gamma}
     A $\beta$-SR hyper-rectangle  is $\gamma$-SR with $\gamma = \beta^d \, d^{d/2}$. Conversely, a $\gamma$-SR hyper-rectangle  is $\beta$-SR with $\beta= \gamma$.
\end{proposition}
\noindent The two definitions of shape regularity, $\gamma$ and $\beta$, are therefore equivalent in the case of hyper-rectangles. More precisely, the first implication from above will be of particular interest for us, as it will allow us to show that some regression trees are $\gamma$-shape-regular. In practice, one way to obtain a $\beta$-SR (and therefore $\gamma$-SR) tree is to allow only for $\beta$-SR splits when growing the tree, i.e., valid splits in light of Definition \ref{def:beta_regular}. This is easily imposed, as it only requires one to restrict the  optimization domain when finding the optimal split. We further develop this aspect in Section \ref{s53} below.

Note that, in dimension $d=1$, trees are necessarily shape-regular for $\gamma=\beta=1$ as $ h_- = h_+$. From this perspective, dimension $1$ plays a special role and might exhibit convergence properties that would not generalize to larger dimensions.

Let us now formalize a bit more on the idea that non-shape-regular sets would lead to suboptimal convergence rates, at least for some regression functions that are sufficiently varying in the covariates domain. Consider indeed the function $g:x\mapsto \sum_{k=1} ^d x_k $ defined on $[0,1]^d$. Set $d\geq 1$ and assume that $X \sim \mathcal U [0,1]^d$. Consider estimating $g$ at $0$ using a rectangular cell such that $\diam(\mathcal V)^d/ \lambda (\mathcal V) \geq \bar {\gamma}$ where $\bar {\gamma} >0$. Since $g$ is Lipschitz - note that each partial derivative of $g$ is actually \textit{equal} to one pointwise - optimal rates are of order $ n^{ - 1/(d+2)}$. Next we show that, under standard conditions, the optimal rate cannot be achieved when $\bar \gamma $ grows with $n$. This is important as it means that the optimal rate cannot be attained except when $\bar{\gamma}$ is bounded, meaning that trees need to be shape-regular for being optimal.

\begin{proposition}\label{contre_ex}
Let $n\geq 1$ and $d\geq 1$. Suppose that $X \sim \mathcal U [0,1]^d$ and that \ref{cond:epsilon} and \ref{cond:D} are fulfilled with $ g(x) = \sum_{k=1} ^d x_k$. Consider a local map $\mathcal V$ such that $ \mathcal V(0) = \prod_k[0,h_k]$, for some deterministic side lengths $h_k$. Let $\bar{\gamma}$ be such that $\diam(\mathcal V)^d/ \lambda (\mathcal V) \geq \bar{\gamma}$. Whenever $2^{d+4} \log(2) \leq  n \prod_{k=1}^d h_k$, there exists a constant $C_d > 0$ depending only on $d$ such that $$   \EE [ ( \hat g_{\mathcal V} (0 ) -  g(0 ))^2 ]^{1/2}  \geq  C_d \left( \dfrac{\bar{\gamma} \sigma^2}{n} \right)^{1/(d+2)}. $$
\end{proposition}





%

More generally, the latter result still holds if $g(x) - g(0) \geq \sum_{k=1}^d x_k$ on $\mathcal V (0)$. An example of such function is, for instance, $g$ differentiable, $\nabla g (0)= (1,...,1)^T$ and $g$ convex. But many non-convex functions satisfy this condition, of course. Note also that the previous result can be extended to covariates $X$ having a density uniformly bounded from above and from below. Finally, by an easy conditioning argument, Proposition \ref{contre_ex}
still holds for side lengths $h_k$ that are independent of the sample, if $\bar \gamma$ is still deterministic. We stress that for most random trees, the randomness of the construction will actually require to consider a random shape parameter $\gamma$, and to study its stochastic variability (see Section \ref{sec_PRT} where uniform, centered and Mondrian tree are considered).

\subsection{Shape regularity of local maps}\label{4.3}

Let us now introduce the following definition, which requires that all elements of the local map are $\gamma $-SR. 

\begin{definition}
\hspace{-0.2cm} A local map $x\mapsto  \mathcal V (x)$ is $\gamma$-SR  if all elements in \hbox{$\{\mathcal V (x) : x\in S_X\}$} are $\gamma$-SR. 
\end{definition}

To validate the $\gamma$-SR condition, we now provide some error rates for such $\gamma$-SR local maps, when choosing a suitable value for the volume. In the next statement, we use the notation
$ f \lesssim g $ when there exists a universal constant $ a >0$ such that 
$   f  \leq a  g .$
We write $ f \asymp g $ whenever $f\lesssim g $ and $g\lesssim f$, 

\begin{theorem}\label{main_result_localizing_map}   
Under the assumptions of Theorem \ref{th2:general}, if the local map is $\gamma$-SR and if for all $x\in S_X$,
$\lambda(\mathcal V(x) ) \asymp ({\log ( {(n+1)^v} / { \delta} ) } / { n} )  ^{d/(d+2)}$, we have, with probability at least $1 - 3\delta $, for all $x\in S_X$,
    \begin{align*}
        & | \hat g_{\mathcal V}(x) - g(x)| \lesssim c \left(\frac{\log\left( \frac{(n+1)^v}{ \delta} \right) }{ n} \right) ^{1/(d+2)}
    \end{align*}
    where $c = \sqrt{{3 \, \sigma^2}/\ell(x)} +  L(\mathcal V(x) )  \gamma^{1/d}.$ In addition, whenever $S_X$ is bounded and $ \ell(x) \geq b >0 $ for all $x\in S_X$, we have, with probability at least $1 - 3\delta $,
 \begin{align*}
        & \sup_{x\in S_X} | \hat g_{\mathcal V}(x) - g(x)| \lesssim c \left(\frac{\log\left( \frac{(n+1)^v}{ \delta} \right) }{ n} \right) ^{1/(d+2)}
    \end{align*}
    where $c = \sqrt{{3 \, \sigma^2}/b} +  L  \gamma^{1/d}.$
    \end{theorem}

 We prove in Theorem \ref{main_result_localizing_map} some probability bounds for the pointwise and sup-norm errors. Note that our pointwise probability bound is valid \textit{for all} $x$ in the domain $S_X$, but with a pre-factor depending on $x$.

The order of the volume $\lambda(\mathcal V(x)) $ in Theorem \ref{main_result_localizing_map} allows to minimize the bound in Theorem \ref{th2:general}. In most practical situations, this precise parameter cannot be directly tuned and one is only able to select another hyper-parameter that will in turn impact the value of $\lambda(\mathcal V(x))$, as observed in the examples of the next section. However, it serves to illustrate the potential rate of convergence achievable under our theorem.

Note also that our choice for the order of the volume $\lambda(\mathcal V(x))$ depends on the confidence level $\delta$, thus making the estimator $\delta$-dependent. An alternative choice, such as $\lambda(\mathcal V(x)) \asymp (\log(n) / n)^{d/(d+2)}$, has the advantage of being independent of $\delta$.
Such a choice allows us to extend our result to pointwise and uniform convergence rates in expectation of order $(\log(n)/n)^{1/(d+2)}$, which corresponds to the minimax rate in expectation for the sup-norm error (see, for instance, \citep{tsybakov2009}). 

\color{black}

\section{Data-dependent local regression maps}\label{s5}
In this section, we show that shape regularity is useful to analyse local regression maps that are data-dependent. The first example is the nearest neighbors regression estimator, the second involves recent prototype versions of the nearest neighbors algorithm, and the third considers a modified and generalized version of CART. 
\subsection{Nearest neighbors regression}\label{s61}

Nearest neighbors regression estimators are local maps estimators for which
$\mathcal V (x) = B (x, \hat \tau _{n,k}(x) ) $
where $ \hat \tau _{n,k}(x)$ has been defined in Section \ref{s2}, Example \ref{ex3}. In contrast with the general approach developed in the previous section, which relies on Assumption \ref{cond:density_X}, we here no longer consider the (possibly random) local map $\mathcal V$ but rather focus on a given class of balls (with small enough radius).

\begin{enumerate}[label=(XNN), wide=0.5em,  leftmargin=*]
\item  \label{cond:density_XNN} 
There is a positive function $\ell $ defined on $S_X$ and $T_0>0$ such that, for all $x\in S_X$ and $\tau \in (0, T_0)$,
$$ P^X( B(x, \tau) )  \geq   \ell(x)\tau ^d.$$
\end{enumerate}

  As we will see below, Assumption \ref{cond:density_XNN} is sufficient when dealing with nearest neighbors regression estimators. Moreover, Assumption \ref{cond:density_XNN} is satisfied whenever $X$ has a density $f_X$ which is bounded below by a constant $b> 0$ on $S_X$ (in which case $S_X$ must be bounded) and when $S_X$ satisfies $ \int _{S_X \cap B(x,\tau )} d\lambda \geq \kappa_0 \int _{B(x,\tau )} d\lambda$, 
  for all $\tau \in (0, T_0)$. Assumption \ref{cond:density_XNN} can also be satisfied when $S_X$ is unbounded. Several examples are given in  \cite{gadat2016classification}. 

Following an approach quite similar to the proof of Theorem \ref{th:general}, we obtain the following result.

\begin{theorem}\label{th:NN}
Let $\delta \in (0,1/3) $, $n\geq 1$, $d\geq 1$ and $k\geq  8    \log( 4 (2n+1) ^{(d+1) } / \delta ) $.  Let $\mathcal V $ be obtained from nearest neighbors algorithm, as detailed in Example \ref{ex3}. Suppose that \ref{cond:epsilon}, \ref{cond:D}, \ref{cond:reg4} and  \ref{cond:density_XNN} are fulfilled. We have, with probability at least $1-3\delta$, for all $x\in S_X$ such that $2 k \leq   T_0^d n \ell(x)$,
\begin{align*}
    &|\hat g_{\mathcal V} (x) - g(x) | \leq  \sqrt{\frac{2\sigma^2\log ( {(n+1)^{d+1}/\delta })  }{ k }} + 2 \left(\frac{ 2 k   }{  n \ell(x)  }   \right)^{1/d} L(\mathcal{V}(x)).
\end{align*}
\end{theorem}

Note that the conditions on the value of $k$ are satisfied for $n$ sufficiently large   and $k\asymp n^a $, for any $a\in (0,1)$. 
To our knowledge, the above result is new among the nearest neighbors literature, in which uniform deviation inequalities are provided, but only for densities uniformly bounded away from $0$. Such results have been investigated recently in \cite{jiang2019non} and \cite{portier2021nearest} for compactly supported covariates. In contrast, the above upper bound is valid for all $x$ in any domain $S_X$, at the price of accounting for regions with low density values, which may deteriorate the accuracy locally. We have the following corollary, in which we consider an optimal choice for $k$, as well as a uniform lower bound on the density.

\begin{corollary}\label{corknn}
In Theorem \ref{th:NN}, assuming that $ n $ is sufficiently large and choosing the integer $k \asymp n^{2/(d+2)} \log((n+1)^{d+1}/\delta)^{d/(d+2)}$, yields the following inequality, with probability at least $1-3\delta$ and for all $x\in S_X$, 
   \begin{align*}
      |\hat g_{\mathcal V}(x) - g(x) |  \lesssim c \left( \dfrac{\log((n+1)^{d+1}/\delta)}{n}\right)^{1/(d+2)},
   \end{align*}
    where $c =   \sqrt{2\sigma^2} + 2 L(\mathcal{V}(x)) \left({2}/\ell(x) \right)^{1/d}.$  Moreover, if $\ell(x) \geq b >0 $, for all $x\in S_X$,  we have, when $n$ is sufficiently large, with probability at least $1 - 3\delta $,
 \begin{align*}
        & \sup_{x\in S_X} | \hat g_{\mathcal V}(x) - g(x)| \lesssim c \left( \dfrac{\log((n+1)^{d+1}/\delta)}{n}\right)^{1/(d+2)},
    \end{align*}
   where $c =   \sqrt{2\sigma^2} + 2 L \left(2/b \right)^{1/d}$.
\end{corollary}

Note that the convergence rate is the same as in the abstract Theorem \ref{main_result_localizing_map}. However, the constant $c$ in the first above statement differs significantly from that of Theorem  \ref{main_result_localizing_map} as when $\ell(x)$ is small, the constant in Theorem \ref{main_result_localizing_map} is of order \( { \ell(x) } ^{-1/2} \), while in Corollary \ref{corknn}, it scales as \(  \ell(x) ^{-1/d} \). This is explained by the fact that, in the proofs of the respective results, the value of $\ell(x)$ has an effect on the variance term, that contributes to the bound in Theorem \ref{main_result_localizing_map}, while it appears in the bias term for Corollary \ref{corknn}. Furthermore, observe that the nearest neighbors algorithm is based on a $\gamma$-regular local map since we have the relation $$\diam(\mathcal{V}(x))^d = (2 \hat \tau _{n,k}(x))^d = \dfrac{2^d}{V_d} \lambda(\mathcal{V}(x)).$$

\subsection{Prototype nearest neighbors and OptiNet}\label{s51}

Recall that the observed sample $\mathcal D_n = \{ (X,Y) , (X_1,Y_1),\ldots, (X_n, Y_n) \}  $, as introduced in Section \ref{s2}, is independent and identically distributed with covariates $X_i\in \mathbb R^d$ and response variable $Y_i\in \mathbb R$. The goal is to build a regression map to estimate $\mathbb E [Y|X=x]$, for a given $x\in \mathbb R^d$. A prototype learning algorithm relies on two steps: construct the  prototypes sample $(Z_j,\overline {Y}_j)_{j=1,\ldots, m}$ and train a learning rule based on the prototypes. Arguably the most simple approach among the $1$-nearest neighbor prototype learning is the one studied in \cite{gyorfi2021universal}, called Proto-NN, where the prototype covariates collection $(Z_j)_{j=1,\ldots, m}$ forms an independent and identically distributed collection of random variables 
the same distribution as $X$. The labels $(\overline{Y}_j)_{j=1,\ldots, m}$ are created using the initial sample $(X_i, Y_i) _{i=1,\ldots, n}$ as follows: for each $j=1,\ldots, m$,
$$\overline{Y}_j = \frac{\sum_{i=1} ^n  Y_i \ind_{ V_j } (X_i) }{\sum_{i=1} ^n  \ind_{ V_j } (X_i) } $$ where $(V_j)_{j=1,\ldots, m}$ denotes the Voronoi cells of $(Z_j)_{j=1,\ldots, m}$. The resulting algorithm is the $1$-NN rule applied to the prototype sample $(Z_j,\overline{Y}_j)_{j=1,\ldots, m}$, as formally introduced below.

For $x\in \mathbb R^d$, let $\hat X_1(x)$ be the nearest neighbor to $x$ among the $ Z_1,\ldots, Z_m$, where tie breaking is done, for instance, by favoring larger indexes so that a unique $\hat X_1(x)$ is identified for each $x$. Let $V_k$ denote Voronoi cell of $Z_k$ defined as $ \{x\in S_X \, : \, \hat X_1(x) = Z_k \}$. The collection $V_1,\ldots V_m$ forms a partition of the domain $S_X$. Therefore, each $x$ can be given a unique element $ \mathcal V(x) := V_j$ whenever $x\in V_j $. 
The Proto-NN prediction rule then writes
$$ \hat g _{\textrm{proto}}(x) = \frac{\sum_{i=1} ^n Y_i  \ind_{ \mathcal  V (x) } (X_i )  }{\sum_{i=1} ^n    \ind_{\mathcal  V(x) } (X_i)  }\; ,$$
with, as usual, the convention that $0/0 = 0$. Note from its definition that Proto-NN is a local map estimator that results from a random partition.

Let us now introduce the assumptions required to establish our concentration bound for Proto-NN. First, we require the prototype variables $Z_i$ to satisfy the following condition.

\begin{enumerate}[label=(DZ), wide=0.5em,  leftmargin=*]
\item \label{cond:D1} 
The random variables $\{Z, (Z_i)_ {i=1,\ldots,m} \}$ are independent and identically distributed on $\mathbb R^d$ with common distribution $  P^Z  = P^X$. 
\end{enumerate}

Next, we describe the assumptions on the distribution of the covariates $X$ which also apply to the prototypes $Z$ because $P^X = P^Z$.

\begin{enumerate}[label=(XZ), wide=0.5em,  leftmargin=*]
\item \label{cond:reg2} There is $\rho_0 > 0$ such that $S_X\subset B (0, \rho_0)$. Moreover, there is $c_d>0$ and $T_0>0$ such that
\begin{align*}
&\lambda (S_X \cap  B(x, \tau ) ) \geq c_d \lambda   ( B(x, \tau )) , \qquad \forall \tau \in (0,T_0] , \, \forall x\in S_X.
\end{align*}
 By changing \( c_d \), we can assume that \( T_0 = \rho_0\). Additionally, $X$ has a density $f_X$ on $S_X$  and there exist constants $0 < b\leq M <+\infty$ such that $ b \leq f_X (x) \leq M$. 
\end{enumerate}

In Assumption \ref{cond:D1}, we require that the $X_i$'s and $Z_i$'s follow the same distribution. This is indeed classical framework for prototype algorithms. Interestingly, we note that, in fact, the distributions $P^Z$ and $ P^X$ need not be identical to preserve the rates exhibited in Theorem \ref{th:main_th_proto} and Corollary \ref{cor:main_th_proto}. More precisely, if the distributions $P^Z$ and $P^X$ are different, if $P^Z$ satisfies Assumption \ref{cond:reg2} and if $P^X$ follows Assumption \ref{cond:density_XNN} as for the classical $k$-NN estimator, then the results of Theorem \ref{th:main_th_proto} and Corollary \ref{cor:main_th_proto} would remain unchanged, up to constants. In other words, the rates are preserved under a distribution shift for $P^Z$ and $P^X$, if they respectively satisfy Assumptions \ref{cond:reg2} and \ref{cond:density_XNN}. 

Note also that, compared to Assumption \ref{cond:density_XNN} used in the previous analysis of $k$-NN regression, Assumption \ref{cond:reg2} is slightly stronger. The latter assumption is needed in our proofs to ensure that the (Lebesgue) volume of the Voronoi cell $\mathcal V(x)$ is large enough. 

We also adapt the sub-Gaussian noise assumption to account for the presence of the prototype sample.

\begin{enumerate}[label=(EZ), wide=0.5em,  leftmargin=*]
  \item  \label{cond:reg3} The random variable $\varepsilon$ is sub-Gaussian conditionally on $X$ and $Z$ with parameter $\sigma^2$.
\end{enumerate}

We are now ready to state our non-asymptotic error bound for Proto-NN.

\begin{theorem}\label{th:main_th_proto}
     Assume that \ref{cond:D}, \ref{cond:D1}, \ref{cond:reg4}, \ref{cond:reg2} and \ref{cond:reg3} are fulfilled. Let $\delta \in (0,1/5).$ If $n \geq 1,m \geq 3$ are such that $$\dfrac{n}{m} \geq \dfrac{8\psi_d\log(1/\delta)}{ \delta^d}, \quad \text{ and } \quad 32 d \log(12m / \delta )  \leq T_0^d  m b c_d V_d, $$
where $\psi_d = (18Md)^d (c_d b)^{-d}$, then for all $x \in S_X$ with probability at least $1- 5\delta$,
    $$|\hat g_{\textrm{proto}} (x) - g(x) | \leq\sqrt{\dfrac{4\sigma^2 \psi_d\log(1/\delta)}{\delta^d }} \sqrt{\dfrac{m}{n}} + 2 L(\mathcal{V}(x))  \left(\frac{ 32d \log(12m/\delta) }{ mb c_d V_d }  \right)^{1/d}.$$
\end{theorem}

By choosing $m$ appropriately as a function of $n$, the following minimax convergence rate is established.

\begin{corollary}\label{cor:main_th_proto}
In Theorem \ref{th:main_th_proto}, if $n$ is sufficiently large, then choosing the integer $m \asymp (n/c_\delta) ^{d/(d+2)} \log(12n/\delta)^{2/(d+2)}$, with $c_\delta =  \log(1/\delta) / \delta^d $, yields the following inequality for all $x\in S_X$, with probability at least $1-5\delta$,
    $$|\hat g_{\textrm{proto}} (x) - g(x) | \lesssim  C \left( \frac{c_\delta \log(12n/\delta)}{n }\right)^{1/(d+2)}$$
    where $$C = \sqrt{4\sigma^2 \psi_d} +  2 L(\mathcal{V}(x)) \left(\dfrac{ 32d }{ b c_d V_d }\right)^{1/d}  \ \ \text{   and   } \ \ \psi_d =  \left(\dfrac{18Md}{c_d b}\right)^d.$$
\end{corollary}

%
%
%
%
The previous results are, to the best of our knowledge, the first concentration bounds on the error of the Proto-NN regression estimator.  As pointed out in \cite[Section 3]{gyorfi2021universal}, ``obtaining convergence rates for the universally consistent Proto-NN classifier [...] is currently an open problem'', that the authors bypass by considering another algorithm that is simpler to analyze and that they term ``proto-$k$-NN''. 

The key step in the proof is to get a lower bound on $P^X ( \mathcal V (x)) $. This step involves the control of some order statistics of the distances between pairs of prototype variables. The analysis exhibits a quite poor scaling  - i.e. far from exponential - of the probability at which the minimax rate holds. A similar situation is observed for Mondrian trees and we believe that this cannot be much improved for these estimators. Modifying the definition of the Proto-NN estimator in order to improve the probability bound will be the subject of a forthcoming work. 

Concerning the shape regularity theory developed in previous sections, the Proto-NN algorithm is based on a \(\gamma\)-regular local map with high probability, in the sense that there exist constants \(c  > 0\) such that, with probability at least \(1 - 2\delta\),  
$$\diam(\mathcal V(x))^d \leq c \frac{\log(m/\delta)}{\delta^d} \lambda(\mathcal V(x)).$$ 
This implies that \(\gamma\)-regularity (in probability) holds with a parameter \(\gamma = c \log(m/\delta)/\delta^d\) that is polynomial in $1/\delta$, which is in line with the poor scaling of the probability rate in the concentration bound of Theorem \ref{th:main_th_proto}.

Another approach studied in \cite{NIPS2017_934815ad,hanneke2021universal} and called OptiNet, consists in creating  prototype covariates $(Z_j(\eta))_{i=1,\ldots, m(\eta)} $ as a maximal $\eta$-net  subset of $(Z_j)_{i=1,\ldots, m}$, for which the minimum spacing between the elements, $\min_{j\neq k} \|Z_j- Z_k\|$, is larger than $\eta$. The prototype labels are then created in the same way as for Proto-NN, by averaging the labels inside the Voronoi cells obtained from $(Z_j(\eta))_{i=1,\ldots, m(\eta)} $. Let $\mathcal V_\eta (x) $ be the Voronoi cell of $x$, with respect to the sample $Z(\eta)=(Z_i(\eta))_{i=1,\dots,m(\eta)} $. The OptiNet prediction rule is given by $$ \hat g _{\textrm{opt}}(x) = \frac{\sum_{i=1} ^n Y_i  \ind_{ \mathcal  V_\eta(x) } (X_i )  }{\sum_{i=1} ^n    \ind_{\mathcal  V_\eta(x) } (X_i)  }$$ with the convention that $0/0 = 0$. For the OptiNet algorithm, we obtain the following error bound.

\begin{theorem}\label{th:main_th_OptiNet}
    Suppose that \ref{cond:D}, \ref{cond:D1},  \ref{cond:reg4}, \ref{cond:reg2} and \ref{cond:reg3} hold true. If $\delta \in (0,1/4)$, $n \geq 1$, $m \geq 1$ are such that 
    \begin{equation*}
     32d\log(12m/\delta) \leq  T_0^d  m b c_d V_d, \quad \eta \leq 2T_0 \quad \text{ and } \quad n b V_d c_d \eta^d \geq 2^{d+3} \log(1/\delta), 
    \end{equation*}
then for all $x \in S_X$, with probability at least $1-4\delta$,
        $$|\hat g_{\textrm{opt}} (x) - g(x) | \leq\sqrt{\dfrac{2^{d+2} \sigma^2 \log(1/\delta)}{n b V_d c_d \eta^d} }+ 2 L(\mathcal{V}_\eta(x)) \left\{  \eta +  \left(\dfrac{32d \log(12m/\delta)}{mb c_d V_d} \right)^{1/d} \right\}.$$
\end{theorem}

Note that the OptiNet algorithm has the same rate of convergence as Proto-NN, but the above upper bound holds under a larger probability compared to the one of Proto-NN. This is a consequence of the $\eta$-net construction, which allows the control of the volume of the Voronoi cells, that is larger than $\eta^d$, in a better way than for Proto-NN. 
Optimizing in $\eta$ and $m$ the upper bound, the order of the optimal choice corresponds to $\eta = n^{-1/(d+2)}$ and $m \geq  n^{d/(d+2)}$, which yields an upper bound of order $n^{-1/(d+2)}$ up to some logarithmic terms.
Note that the previous choice of $m $ and $\eta$ automatically satisfies the condition of Theorem \ref{th:main_th_OptiNet} when $n$ large enough. 

Let us take \(1/\delta = n \log(n)^2\) and \(\eta = (\log(n)/n)^{1/(d+2)}\). Choose $m$ at least larger than \(n^{d/(d+2)}\) so that $\log(m/\delta)/{m} \leq \eta^d$ (this is ensured as soon as \(m \gtrsim n^{d/(d+2)} \log(n)^{2/(d+2)}\)). In this way, the condition on $m$ in Theorem \ref{th:main_th_OptiNet} is satisfied for large enough $n$, and by the Borel–Cantelli lemma, we obtain that for each \(x \in S_X\), almost surely  
$$|\hat g_{\textrm{opt}} (x) - g(x) | \leq C \,  \left(\frac{ \log(n)}{n}  \right)^{1/(d+2)}$$  
where \(C > 0\) is a constant depending on all the problem parameters, but independent of $n$.

The underlying local map of the OptiNet algorithm  satisfies the $\gamma$-shape regularity in probability, in the sense that there exists a constant \(c > 0\) such that, with probability at least \(1 - \delta\),  
$$\diam(\mathcal V_\eta(x))^d  \leq c \lambda(\mathcal V_\eta(x))\;,$$
whenever $\log(m/\delta)/m \leq \eta^d.$
In particular, for the above choices of $\delta, \eta$, and $m$, we have that, 
almost surely, for large enough $n$, the cell \(\mathcal V_\eta(x)\) constructed by the OptiNet algorithm is $\gamma$ shape-regular with $\gamma = c$.

\subsection{CART-like regression tree}\label{s53}

We now consider general local regression trees for which each split is selected using a general cost function. In particular, the deviation inequality obtained below is valid for local regression maps that may depend on the whole dataset $(X_i,Y_i)_{i=1,\ldots,n }$ and not only on the covariates as in nearest neighbors algorithm. We call these trees ``CART-like'', since CART algorithm is arguably the most important instance of such data-dependent regression trees, due to its wide fame and use in practice.

Let us introduce a general class of recursive data dependent trees. For simplicity, we assume that  $S_X = [0,1]^d$. 
For a given cell $V$, a split is characterized by two parameters $(p, u) \in S : = \{ 1,\ldots, d\}\times  (0,1)$. Recall that for a cell $V$, we denote by $h_k(V)$ the (Euclidian) length of its side along coordinate $k$. The resulting left and right child cells, $V(l)$  and $V(r)$, are such that for any $k \neq p$, $h_k(V(l))=h_k(V(r))  = h _k (V) $, and for $k=p$, $ h_k(V(l)) = h_k(V) u$ and $ h_k(V(r)) = h_k(V) (1-u) $. We also recall that $h_-(V) = \min_{k = 1,\ldots, d} h_k(V) $ and $h_+(V) = \max_{k = 1,\ldots, d} h_k(V) $.  With these notations, the split condition for $V$ to be $\beta$-shape regular can be expressed with the help of a restriction on the set of valid splits. Given $V$, let us define the set of $\beta$-shape regular splits as follows,
$$ S_\beta(V) : = \{  (p,u)  \in S  \  : \  h_+( V(s)) \leq \beta h_-( V(s)),\,  \forall s  \in \left\{ l,r \right\}   \}. $$
We note that when $\beta \geq 2$, the  $S_\beta(V)$ cannot be empty. Splitting the largest side in the middle is always in $ S_\beta(V) $. Another restriction on the splits is needed to ensure a sufficient number of points. It is given by
$$ S_m(V) : = \{  (p,u) \in S  \  : \  n P_n^X( V(s) ) \geq m , \, \forall s  \in \left\{ l,r \right\}   \} .$$
We do not need to fully specify the splitting criterion. When $S_m(V)\neq \emptyset$, the split in the cell $V$ is defined as a minimizer - assumed to exist - on $S_\beta(V)  \cap S_m(V)$, of a cost function $M_n$, given by
\begin{align*}
    M_n\, : \, S \times \mathcal R ([0,1]^d)  \to &\, \mathbb R \\
    ((p,u), V) \mapsto &\,  M_n ((p,u), V),
\end{align*}
where $ \mathcal R ([0,1]^d)$ is the set of hyper-rectangles in $S_X$. In the case where $S_m(V) = \emptyset$, no split is performed and the cell $V$ remains unchanged. The main strength of our analysis lies in the generality of the cost function, which can actually be any function ensuring the existence of a minimizer as required above, and that may depend or not on the sample. For instance, in CART-regression, the cost function depends on the sample and is defined as
$$ M_n ((p,u), V) = \frac{\sum_{i=1}^n (Y_i - \overline Y( V(l))) ^2 \, \ind_{V(l)}(X_i)}{n P_n^X (V(l) )  }  + \frac{\sum_{i=1}^n (Y_i - \overline Y(V(r)))^2 \, \ind_{V(r)}(X_i)}{n P_n^X(V(r) )   } $$
where $ \overline Y(V )= \sum_{i=1}^n Y_i \, \ind_{V}(X_i) / (n P_n^X (V ))$ for any cell $V$. 

\begin{algorithm}[h]
\begin{algorithmic}[h]
\Statex{\textbf{Input:} Sample  $(X_i, Y_i)_{i=1,\ldots, n} \subset [0,1]^d \times \mathbb R$, minimal number of points $m\in \{1,\ldots, n\} $, shape-regularity $\beta \geq 2$, cost function $M_n: S \times \mathcal R ([0,1]^d)  \to   \mathbb R$.
Let $V^{(0)} = \{[0,1]^d\} $ be the initial partition, made of one element (i.e. $| V^{(0)}|=1$).}
   \For{$j= 0,1,\ldots$}
\Statex\hspace{\algorithmicindent}{Let $V^{(j+1)} = \emptyset$} denote the partition at step $j+1$. The update is as follows: 
   \For{$k=  1,2,\ldots, | V^{(j)}|$} 
\Statex\hspace{\algorithmicindent}\hspace{\algorithmicindent}
(a) Whenever $S_m(V^{(j)}_k) \neq \emptyset $, define two children, $V(l)$ and $V(r)$, according to
$$\argmin_{(p,u) \in S_\beta(V^{(j)}_k)\cap S_m(V^{(j)}_k) } M_n( (p,u) , V^{(j)}_k) $$
\hspace{\algorithmicindent}\hspace{\algorithmicindent}{ If the above optimization problem has no solution, just pick $p$ as the largest side} 
\Statex\hspace{\algorithmicindent}\hspace{\algorithmicindent}{ and $u = 1/2$. Set $$V^{(j+1)} = \{ V^{(j+1)}, V (l), V (r)\}$$
}
\Statex \hspace{\algorithmicindent}\hspace{\algorithmicindent}{
(b) Whenever $S_m(V^{(j)}_k) = \emptyset $, child is same as parent. Set
  $$V^{(j+1)} = \{ V^{(j+1)}, V^{(j)}_k \}$$
}
\EndFor
\Statex\hspace{\algorithmicindent}{STOP if $ V^{(j+1)} = V^{(j)} $ (no valid split exists)}
   \EndFor
\Statex{Return the final partition elements $V^{(j+1)}$}
\end{algorithmic}
\caption{CART-like regression tree}
\label{alg:cart-like}
\end{algorithm}

By splitting on the intersection of $S_\beta$ and $S_m$,  Algorithm \ref{alg:cart-like} ensures that the two conditions are met when growing the tree. The first growing condition, which is the $\beta$-shape regularity of the cell, may not constitute a stopping criterion. Indeed, because $\beta\geq 2$, one can always split at the middle the largest side of the considered cell. The other growing condition on $m$ is easy to check in practice since it amounts to keep a cell as a leaf if and only if the number of data points belonging to that cell is greater than $m$ and strictly smaller than $2m$. As a consequence, one might modify classical algorithms, in the case precisely where the split proposed by the algorithm does not respect the $\beta$-shape-regularity condition for a prescribed value of $\beta$, or the other growing condition asking for sufficiently many points in the cells.

Similarly to what has been done in analyzing $k$-NN algorithm, we introduce here a variant of the minimal mass assumption \ref{cond:density_X} stated in Section \ref{4.1}. 

\begin{enumerate}[label=(XTREE), wide=0.5em,  leftmargin=*]
\item  \label{cond:density_XCART} The random variable $X$ admits a density function $f_X$ on $S_X = [0,1]^d$ which is bounded from below by $b >0$, i.e., $f_X(x) \geq b$ for all $x\in S_X$.
\end{enumerate}

Similarly to \ref{cond:density_XNN}, the above assumption is stronger but more practical than \ref{cond:density_X}, as it no longer involves the local map $\mathcal V$, that depends on the sample. The next theorem gives a deviation inequality on the error associated to the regression map estimator resulting from Algorithm \ref{alg:cart-like}.

\begin{theorem}\label{th_cart_like1}
Let  $\delta \in (0,1/3) $, $n\geq 1$, $d\geq 1$, $\beta\geq 2$ and $m\in \{1,\ldots,  n\}$ such  that $m\geq  4\log(4(2n+1)^{2d}/\delta )$. Suppose that \ref{cond:D}, \ref{cond:density_XCART}, \ref{cond:epsilon} and \ref{cond:reg4} are fulfilled. Let $\mathcal V$ be the local regression map obtained from a CART-like tree with input parameters $\beta $, $m$ and cost function $M_n$, then we have, with probability at least $1-3\delta$, for all $x \in S_X$,
    \begin{align*} 
            |  \hat g_\mathcal V (x) - g(x)  | \leq \sqrt { \frac{2 \sigma ^2 \log((n+1) ^{2d}/\delta) } {  m}} + L(\mathcal{V}(x)) \beta \sqrt{d} \left(\frac{5 m}{ n b}\right)^{1/d}.
    \end{align*}
\end{theorem}

Note that the conditions on the value of $m$ are satisfied whenever $n$ is sufficiently large and $m\asymp n^a $, 
 for any $a\in (0,1)$. Notice that taking $m \asymp n^{2/(d+2)}$ in the estimation bound of Theorem \ref{th_cart_like1} gives the optimal convergence rate $n^{-1/(d+2)}$, up a multiplicative logarithmic term. Moreover, such a value of $m$ allows the bound to be valid with a probability that grows to one polynomially in $n$, since the constraint $m \geq 4\log(4(2n+1)^{2d}/\delta )$ will be then satisfied.
In addition, such results remain valid for the rate of convergence in sup-norm whenever the density $f_X$ is uniformly bounded from below by a positive constant, independent of $n$. This is stated in the subsequent corollary.

\begin{corollary}\label{cor_cart_like1}
In Theorem \ref{th_cart_like1}, if the integer $m$ is chosen as $m \asymp n^{2/(d+2)} \log((n+1)^{2d}/\delta)^{d/(d+2)}$, then we have the following inequality for $n$ sufficiently large with probability at least $1-3\delta$, 
   \begin{align*}
      \sup_{x\in S_X} | \hat g_{\mathcal V}(x) - g(x)|  \lesssim c \left( \dfrac{\log((n+1)^{2d}/\delta)}{n}\right)^{1/(d+2)},
   \end{align*}
  where $c =   \sqrt{2\sigma^2} + \beta L \sqrt{d} ({5}/ b )^{1/d}$.
\end{corollary}

The previous result shows that CART-like regression trees are able to attain the optimal rate of convergence as soon as a simple constraint - restricting acceptable splits by a simple rule - is imposed during the tree construction. 

Interestingly, results presented in \cite{cattaneo2022pointwise} tend to indicate that such modifications are in general necessary for the classical CART algorithm to achieve a good pointwise - or uniform - behavior. More precisely, it is shown in \cite{cattaneo2022pointwise} that the use of CART is problematic for the estimation of a constant regression function,  measured with the sup-norm error. Indeed, its rate of convergence in dimension one is slower than any polynomial of the sample size $n$, with non-vanishing probability. In addition, the honest version of CART - i.e. when the prediction values among the cells use data that are independent of those used to construct the partition, see Definition 5.1 in \cite{cattaneo2022pointwise} -, is proved to be inconsistent with positive probability as soon as the tree depth is of order at least $\log(\log(n))$.  This is due to the fact that the splitting criterion produces leaves that are too small.  

Our results complete the picture drawn in \cite{cattaneo2022pointwise} by putting forward the fact that producing too small cells is \textit{the only problem} that can occur with the use of CART in dimension one. Indeed, any cell being $\beta$-shape-regular in dimension one, with $\beta=1$, Theorem \ref{th_cart_like1} shows that the only problem must come from the amont of data $m$ in the least populated cell. Indeed, if $m$ is of order $\log(n)$, then our deviation bound in Theorem \ref{th_cart_like1} does not converge to zero when $\delta$ is fixed and the sample size goes to infinity. This is basically what happens in \cite{cattaneo2022pointwise}. In such a case, we are indeed not able to prove the consistency of CART. 





\section{Purely random trees}\label{sec_PRT}







We consider now purely random trees (PRT), that are built by successively refining a partition of the space, in a way that is independent of the initial sample $\mathcal D_n$. Before considering uniform, centered and Mondrian trees, we start by studying a key property of Lebesgue volume invariance which will be satisfied for the trees of interest. In this section we assume, for clarity, that $S_X= [0,1] ^d$ and we always take $x\in S_X$.


\subsection{Lebesgue volume invariance}\label{subsection_VIR}



To set up notations, let us describe a PRT locally around a point $x$ using the local maps framework introduced before. The tree is generated iteratively, and at each step $i$, for the cell $\mathcal V(x)$ containing $x$, a coordinate is selected according to a random variable $D_i\in \{ 1,\ldots,d\}$ and then the side of the cell in direction $D_i$, that we write $(a,b)$, $a<b$, is split into two intervals $(a,a+(b-a)S_i)$ and $(a+(b-a)S_i,b)$, thus defining two new cells $C_1$ and $C_2$. Consequently, each step $i$ consists in splitting a cell and depends on a pair of random variables $(D_i,S_i)$, that is independent of the dataset $\mathcal D_n$. After $N$ steps, we denote $\mathcal V(x)=\mathcal V (x, (D_i,S_i)_{i=1}^N)$.

In the following proposition, we state the remarkable fact that the Lebesgue volume of the cell can be expressed independently from the successive coordinate choices. We denote $\bar{S}_i$ the length reduction of the side $D_i$ of the considered cell at step $i$, that is either equal to $S_i$ or $1-S_i$ according to the fact that the coordinate $x_{D_i}$ is smaller or greater than $a+(b-a)S_i$, respectively.

\begin{proposition}[Lebesgue Volume Invariance] \label{prop_leb_vol_inv}
With the notations above, the following formula holds
\begin{equation*} \label{eq_vol}
    \lambda(\mathcal V(x,(D_i,S_i)_{i=1}^N))=\prod_{i=1}^N \bar{S}_i.
\end{equation*}
Assume in addition that for any $i$, the distribution of $S_i$ is symmetric around $1/2$, that is, $S_i \sim 1-S_i$. Then we get the following equality in distribution,
    \begin{equation*}\label{eq_dist_inv}
        \lambda (\mathcal V(x,(D_i,S_i)_{i=1}^N)) \sim  \prod_{i=1}^N S_i.
    \end{equation*}
\end{proposition}


Note that for centered or uniform random trees, the distributions of the $S_i$'s are indeed symmetric around $1/2$. Furthermore, for centered trees, $S_i=1/2$ almost surely, the Lebesgue volume of the cell containing $x$ after $N$ steps is equal to $1/2^N$. 

It is worth also noticing that actually, the formulas of Proposition \ref{prop_leb_vol_inv} are valid even if the random variables $D_i$ and $S_i$ depend on the dataset. Thus, the Lebesgue volume of the cell containing $x$ is independent of the direction choices as soon as the random vectors $(D_i)_{i=1}^N$ and $(S_i)_{i=1}^N$ are independent of each other, but not necessarily independent of the dataset. 

\subsection{Uniform random trees}\label{subsec_URT}

Let us first provide some deviation bounds for the diameter and volume of the local map built with uniform random trees. 

\begin{proposition}\label{prop:diamURT2}
Consider  that $S_i=U_i$ are independent and uniformly distributed over $(0,1)$ and that $D_i$ are independent of each other and from the $U_i$'s and uniformly distributed over $\left\{ 1,\ldots,d\right\}$. Then, for $\mathcal V (x)= \mathcal V(x, (D_i,S_i)_{i=1}^N)$ and for any $\theta \geq 0$,
\begin{equation*}\label{eq_dev_diam_URT}
    \PP(\diam (\mathcal{V}(x))\geq \sqrt{d}e^{-N/d+N\theta}) \leq de^{-Nd\theta^2/4}.
\end{equation*} 
Moreover, for all $\theta \in (0,2/d)$ we have,
\begin{equation*}
    \PP(\diam (\mathcal{V}(x))\leq \sqrt{d}e^{-N/d - N\theta}) \leq de^{-Nd\theta^2/8}.
\end{equation*} 
\end{proposition}

\begin{proposition}\label{prop:diam_vol_URT2}
Consider  that $S_i=U_i$ are independent and uniformly distributed over $(0,1)$ and that $D_i$ are independent of each other and from the $U_i$'s and uniformly distributed over $\left\{ 1,\ldots,d\right\}$. Then, for $\mathcal V (x)= \mathcal V(x, (D_i,S_i)_{i=1}^N)$ and for any $\alpha >1$,
\begin{equation*}\label{eq_dev_vol_URT}
     \PP(\lambda(\mathcal{V}(x))\leq e^{-\alpha N}) \leq \left(\alpha e^{1-\alpha}\right)^N.
\end{equation*}
In addition, for any $\alpha \in (0,1)$,
\begin{equation*}
     \PP(\lambda(\mathcal{V}(x)) \geq e^{-\alpha N}) \leq \left(\alpha e^{1-\alpha}\right)^N.
\end{equation*}
\end{proposition}

\begin{proposition}\label{cor_diamvolURT}
When the number of splits goes to infinity, it holds that, almost surely, there exists $N_0\geq 1$ such that for all $N\geq N_0$,
$$ \sqrt d e^{-N/d-4\sqrt{N\log(N)/d}} \leq \diam(\mathcal{V}(x)) \leq \sqrt d e^{-N/d+2\sqrt{2N\log(N)/d}}$$
and
$$ e^{-N- 2\sqrt{N\log(N)}} \leq \lambda(\mathcal{V}(x)) \leq e^{-N + 2\sqrt{N\log(N)}}.$$
As a consequence, if we denote the normalized diameter $\diam^{\#}(\mathcal{V}(x)) := \diam(\mathcal{V}(x))/\sqrt{d}$, we obtain that, almost surely, for $N$ large enough, 
\begin{equation*}
  e^{-2\sqrt{N\log(N)}(1+2\sqrt{d})}   \leq \dfrac{\diam^{\#}(\mathcal{V}(x))^d}{\lambda(\mathcal{V}(x))} \leq e^{2\sqrt{N\log(N)}(1+\sqrt{2d})}. 
\end{equation*}
\end{proposition}
The previous results are valid in any dimension $d\geq 1$, but in dimension one, the (normalized) diameter of any cell is always equal to its Lebesgue volume, so we always have $\diam(\mathcal{V}(x))=\diam^{\#}(\mathcal{V}(x))=\lambda(\mathcal{V}(x))$.

We deduce  the following high probability upper bound on the pointwise error of the resulting local map regression estimator. 

\begin{theorem}\label{cor_URT}
   Let $n\geq 1$, $d\geq 1$, $x\in S_X$ and $N=d\log(n)/(d+2)$. Suppose that $\mathcal V (x)= \mathcal V(x, (D_i,S_i)_{i=1}^N)$ is obtained from a uniform random tree as described in Proposition \ref{prop:diamURT2}. Under \ref{cond:epsilon}, \ref{cond:D}, \ref{cond:density_XCART} and \ref{cond:reg4}, there exists $C>0$, that only depends on the parameters of the problem but not on $n$, such that almost surely, there exists an integer $n_0$ such that for all $ n \geq n_0$,
\[
 | \hat g_{\mathcal V}(x) - g(x)| \leq
 C n^{-1/(d+2)}\sqrt{\log(n)} \, e^{2\sqrt{\log(n)\log\log (n)}}.
\]
\end{theorem}
%

%

%

%

From Theorem \ref{cor_URT},  we see that the local regression map estimator based on the uniform random partition achieves with probability tending to one a pointwise estimation error that is close to the minimax optimal one for the error in expectation, in the sense that for any $\varepsilon>0$, the estimation error is negligible compared to $n^{-1/(d+2)+\varepsilon}$ for $n$ sufficiently large.

The following negative result establishes that uniform trees are not shape-regular, thus indicating that the optimal rate of convergence may indeed not be achieved by the local estimator based on a uniform random tree.

\begin{proposition}\label{uniform tree not regular}
Uniform trees are not $\beta$-SR, i.e., for any \( N \geq d \) and any hyper-rectangle $\mathcal V (x)= \mathcal V(x, (D_i,S_i)_{i=1}^N)$ obtained from a uniform random tree as described in Proposition \ref{prop:diamURT2}, we have, with probability at least $1/11$,
$$\dfrac{h_+(\mathcal V (x) )}{h_-(\mathcal V (x) )} \geq e^{\sqrt{N/d}}.$$
\end{proposition}

While, in the above, the value $1/11$ can certainly be improved, we stress that our result implies that shape-regularity fails to happen on an event having positive probability (independent of $n$).

\subsection{Centered random trees}\label{subsec_centered}

In the case of centered random trees, the volume of the cell $\mathcal{V}(x)$ after $N$ steps is simply $  \lambda(\mathcal{V}(x))= (1/2)^N$. The diameter of the local map behaves as follows.

\begin{proposition}\label{prop:diam_CRT}
Let $d\geq 2$ be an integer. Consider that $S_i=1/2 $ almost surely and that $D_i$ are independent of each other and uniformly distributed over $\left\{ 1,\ldots,d\right\}$. Then, for $\mathcal V (x)= \mathcal V(x, (D_i,S_i)_{i=1}^N)$ and for any $\alpha \in (0,1/d)$,
\begin{equation*}
    \PP(\diam(\mathcal{V}(x))\geq \sqrt{d}2^{-\alpha N}) \leq d\left(1-\frac{1-\theta}{d}\right)^N \theta^{-\alpha N},
\end{equation*}
where $\theta=(d-1)\alpha/(1-\alpha).$ Moreover, for any $\alpha \in (1/d,1)$, we have for the same $\theta$
\begin{equation*}
    \PP(\diam(\mathcal{V}(x))\leq \sqrt{d}2^{-\alpha N}) \leq d\left(1-\frac{1-\theta}{d}\right)^N \theta^{-\alpha N}.
\end{equation*}
\end{proposition}

We have the following corollary about the shape-regularity of centered random trees.

\begin{proposition}\label{cor22}
When the number of splits goes to infinity, it holds that, almost surely, there exists $N_0\geq 1$ such that for all $N\geq N_0$,
$$ \sqrt{d} 2^{-N/d-2\sqrt{(d-1)N\log(N)/d^2}} \leq \diam(\mathcal{V}(x)) \leq \sqrt d 2^{-N/d+2\sqrt{(d-1)N\log(N)/d^2}}.$$
In addition, if we denote the normalized diameter by $\diam^{\#}(\mathcal{V}(x)) := \diam(\mathcal{V}(x))/\sqrt{d}$, almost surely it holds, for $N$ large enough,
\begin{equation*}
  2^{\, -2\sqrt{(d-1)N\log(N)}}   \leq \dfrac{\diam^{\#}(\mathcal{V}(x))^d}{\lambda(\mathcal{V}(x))} \leq 2^{ \, 2\sqrt{(d-1)N\log(N)}}. 
\end{equation*}
\end{proposition}

In the same spirit as for uniform random trees, we obtain an upper bound on the pointwise error of the resulting local map regression estimator. 

\begin{theorem}\label{cor_UcR}
Let  $n\geq 1$, $d\geq 1$, $x\in S_X$ and $N=d\log(n)/\{(d+2)\log(2)\}$. Suppose that $\mathcal V (x)= \mathcal V(x, (D_i,S_i)_{i=1}^N)$ is obtained from a centered random tree as described in Proposition \ref{prop:diam_CRT}. Under \ref{cond:epsilon}, \ref{cond:D}, \ref{cond:density_XCART} and \ref{cond:reg4}, there exists $C>0$, that only depends on the parameters of the problem but not on $n$, such that with probability $1$, there is $n_0$ such that for all $ n \geq n_0$,
\[
 | \hat g_{\mathcal V}(x) - g(x)| \leq
 C n^{-1/(d+2)}e^{2\sqrt{\log(n)\log\log (n)}}.
\]
\end{theorem}

As for uniform random trees, the asymptotic almost sure pointwise convergence rate is close to $n^{-1/(d+2)}$, in the sense that, for any $\varepsilon>0$, the estimation error is negligible compared to $n^{-1/(d+2)+\varepsilon}$ for $n$ sufficiently large.

To conclude our analysis of centered trees, we also include the following negative result, which establishes that centered trees are not shape-regular, as suggested by the sub-optimality of the convergence rate obtained in Theorem \ref{cor_UcR}. 

\begin{proposition}\label{centered tree not regular}
Centered trees are not $\beta$-SR, i.e.,
for any \( N \geq d \) and any hyper-rectangle $\mathcal V (x)= \mathcal V(x, (D_i,S_i)_{i=1}^N)$ obtained from a centered random tree, as described in Proposition \ref{prop:diam_CRT}, we have, with probability at least $1/14$,
$$\dfrac{h_+(\mathcal V (x))}{h_-(\mathcal V (x))} \geq 2^{\sqrt{N/d}}.$$
\end{proposition}
As for Proposition \ref{uniform tree not regular} above, the precise value $1/14$ does not play any crucial role here. The important fact is that Proposition \ref{centered tree not regular} shows that shape-regularity is violated on an event having positive probability (independent of $n$).

\subsection{Mondrian trees}\label{s43}

 A Mondrian process \textit{MP} is a process that generates infinite tree partitions of $S_X = [0,1]^d $ (\cite{mondrianroy2008}). These partitions are built by iteratively splitting the different cells at random times, where both the timing and the position of the splits are determined randomly. Additionally, the probability that a cell is split depends on the length of its sides, and the probability of splitting a particular side is proportional to the length of that side. Once a side is selected, the exact position of the split is chosen uniformly along that side.
We can then define the pruned Mondrian process \textit{MP}($\Lambda$). This version introduces a pruning mechanism that removes splits occurring after a specific time $\Lambda > 0$, which is referred to as the lifetime.

Mondrian trees are studied in detail in the paper \cite{minimaxmondrian}. In particular, it is possible to give a simple description of the distribution of the cell $\mathcal V(x)$ containing $x$ and generated by a process \textit{MP}($\Lambda $). Such a property helps to demonstrate the following result, that Mondrian trees are \(\beta\)-SR in probability.

\begin{proposition}\label{mondrian}
 For any $x\in [0,1] ^d$, let  $\mathcal V (x)$ be the hyper-rectangle containing $x$ obtained from a \textit{MP}($\Lambda $) tree. For $\delta \leq 1 - (1 - e^{-1})^d$, we then have, with probability at least \(1 - 2\delta\),
\[ \frac{h_+(\mathcal V (x))}{h_-(\mathcal V (x))} \leq \frac{5d \log(\delta/d)}{{\log(1-\delta)}}. \]
\end{proposition}

The latter inequality implies that, for any small $\delta>0$, there is a constant $K_\delta >0 $ such that  the event $h_+(\mathcal V (x)) \leq K_\delta \, h_-(\mathcal V (x))$ occurs with probability at least $1-\delta$. In other words,  ${h_+(\mathcal V (x))} /{h_-(\mathcal V (x))}~ $ is a tight sequence.
As a consequence, Mondrian regression trees attain, with positive probability, the minimax rate for the pointwise error in expectation.

\begin{theorem}\label{thm:mondrian}
Let  $n\geq 1$, $d\geq 1$, $x\in S_X$, $\Lambda  \asymp n^{1/d+2}$, $\delta\in (0,1/5)$ and define $c_{\delta,d} = \log(1/\delta) (d/ \log(1/(1-\delta)) )^d$. Let  $\mathcal V (x)$ be the hyper-rectangle  containing $x$ obtained from a \textit{MP}($\Lambda $) tree.
 Under \ref{cond:epsilon}, \ref{cond:D}, \ref{cond:density_XCART} and \ref{cond:reg4}, 
 if $n^{2/(d+2)} b \geq 8 c_{\delta,d} $, then it holds, with probability at least $1-5\delta$, 
\[
 | \hat g_{\mathcal V}(x) - g(x)| \lesssim
 C \, n^{-1/(d+2)}\; ,
\]
where $C = \sqrt{ {4\sigma^2  c_{\delta,d}} / {b }} +  5 \sqrt{d} \, L(\mathcal{ V } (x))\log(d/\delta).$
\end{theorem}

While the convergence rate above matches the minimax rate for pointwise error in expectation, it holds with a probability that scales poorly, far from exponential decay. For instance, this rate cannot be extended to an almost sure convergence guarantee. Interestingly, an almost identical result, minimax rates under poor probability scaling, has been obtained for Proto-NN in Corollary \ref{cor:main_th_proto}. This similarity likely stems from the use of an additional source of randomness in the partition construction of both Proto-NN and Mondrian tree. We believe that in both cases, the (random) construction process may lack sufficient stability, with bad events occuring with too large probability, such as the formation of excessively small cells.

\section*{Mathematical proofs} 
Let $\PP$ be the probability measure on the underlying probability space $(\Omega, \mathcal F)$ on which are defined all introduced random variables.

\subsection*{Proof of Theorem \ref{1}}

Let \( \mathbb P_{X_{1:n}} \) denote the conditional probability given \( X_1, \dots, X_n \). Let $\mathcal V = \{ \mathcal V (x) \, : \, x\in \mathbb R^d\}$ and define
$$ \mathcal G = \{ (\mathds 1 _{ A }(X_1), \dots,\mathds 1 _{ A}(X_n)) \, :\,  A \in \mathcal{A}  \}.$$
With this notation we have 
$$\sup_{x\in \mathbb R^d} \frac{\sum_{i=1}^n \varepsilon_i \mathds 1 _{ \mathcal V(x) }(X_i)}{\sqrt{\sum_{j=1}^n \mathds 1 _{ \mathcal V(x) }(X_j)}} \leq  \sup_{(g_1,\ldots, g_n) \in \mathcal G } \frac{\sum_{i=1}^n \varepsilon_i g_i }{\sqrt{ \sum_{j=1}^n g_j}}.$$
Consequently, for all $t > 0$,
\begin{align*}
   \pr_{X_{1:n}}\left(\sup_{x\in \mathbb R^d} \dfrac{\sum_{i=1}^n \varepsilon_i \mathds 1 _{ \mathcal V(x) }(X_i)}{\sqrt{\sum_{j=1}^n \mathds 1 _{ \mathcal V(x) }(X_j)}} > t  \right) 
   &\leq \pr_{X_{1:n}}\left( \bigcup_{(g_1,\ldots, g_n) \in \mathcal G }  \left\{  \frac{\sum_{i=1}^n \varepsilon_i g_i }{\sqrt{ \sum_{j=1}^n g_j}} > t \right\}  \right) \\ 
   &\leq \sum_{(g_1,\ldots, g_n) \in \mathcal G } \pr_{X_{1:n}}\left(\frac{\sum_{i=1}^n \varepsilon_i g_i}{\sqrt{ \sum_{j=1}^n g_j}} > t  \right).
\end{align*}
Moreover, since the conditional distribution of $\varepsilon_i$ given $X_1,\ldots,X_n$ is sub-Gaussian with parameter $\sigma^2$, then  $\varepsilon_i g_i$ is sub-Gaussian under $\pr_{X_{1:n}}$, with parameter $\sigma^2 g_i^2$. Hence, ${\sum_{i=1}^n \varepsilon_i g_i}/{\sqrt{ \sum_{j=1}^n g_j}}$ is sub-Gaussian with parameter $\sigma^2 \sum_{i=1}^n g_i^2 / \sum_{j=1}^n g_j$ by independence. Moreover, $\sum_{i=1}^n g_i^2 = \sum_{i=1}^n g_i$ because $g_i\in \{ 0,1\} $. Hence, ${\sum_{i=1}^n \varepsilon_i g_i}/{\sqrt{ \sum_{j=1}^n g_j}}$ is sub-Gaussian with parameter $\sigma^2$ under $\pr_{X_{1:n}}$. Therefore, we obtain
\begin{align*}
    \pr_{X_{1:n}}\left(\sup_{x\in \mathbb R^d} \dfrac{\sum_{i=1}^n \varepsilon_i \mathds 1 _{ \mathcal V(x) }(X_i)}{\sqrt{\sum_{j=1}^n \mathds 1 _{ \mathcal V(x) }(X_j)}} > t  \right)  &\leq \sum_{(g_1,\ldots, g_n) \in \mathcal G } \exp\left( \dfrac{-t^2}{2\sigma^2}\right)  \leq  \mathbb S_\mathcal V(n) \exp\left( \dfrac{-t^2}{2\sigma^2}\right).
\end{align*}
   If we set $\delta = \mathbb S_\mathcal A(n) \exp\left( {-t^2} /{(2\sigma^2)}\right)$, we have $t = \sqrt{2\sigma^2 \log\left({\mathbb S_\mathcal A(n)}/{\delta}\right)}$. Finally, with probability $\pr_{X_{1:n}}$ at least equal to $1 - \delta$, we get
$$\sup_{x\in \mathbb R^d} \dfrac{\sum_{i=1}^n \varepsilon_i \mathds 1 _{ \mathcal V(x) }(X_i)}{\sqrt{\sum_{j=1}^n \mathds 1 _{ \mathcal V(x) }(X_j)}}  \leq \sqrt{2 \sigma^2 \log\left( \frac{\mathbb S_{\mathcal A}(n)}{\delta} \right)}.$$
   Since $\delta$ is independent of $(X_1,\ldots,X_n)$, we obtain the result by integrating with respect to $(X_1,\ldots,X_n)$.
\qed 

\subsection*{Proof of Theorem \ref{th:general}}

    Let \(x \in S_X\). We write the bias-variance decomposition 
$\hat g_{\mathcal V}(x) - g(x) = V + B$, where \begin{align*}
    V := \dfrac{\sum_{i=1}^n \varepsilon_i \mathds 1 _{ \mathcal V(x) }(X_i)}{\sum_{j=1}^n \mathds 1 _{ \mathcal V(x) }(X_j)} \qquad \text{and}\qquad
       B := \dfrac{\sum_{i=1}^n \left(g(X_i) - g(x)\right) \mathds 1 _{ \mathcal V(x) }(X_i)}{\sum_{j=1}^n \mathds 1 _{ \mathcal V(x) }(X_j)}.
\end{align*}
The inequality from Theorem \ref{1} gives, with probability at least $1 - 2\delta$, for all $x \in S_X $,
\begin{align*}
    |V| &\leq \left({\sqrt{\sum_{j=1}^n \mathds 1 _{ \mathcal V(x) }(X_j)}} \right)^{-1} \sup_{x\in \mathbb R^d} \left| \dfrac{\sum_{i=1}^n \varepsilon_i \mathds 1 _{ \mathcal V(x) }(X_i)}{\sqrt{\sum_{j=1}^n \mathds 1 _{ \mathcal V(x) }(X_j)}} \right| \\ &\leq \dfrac{1}{\sqrt{n P_n^X(\mathcal{V}(x))}} \ \sqrt{2 \sigma^2 \log\left( \frac{\mathbb S_\mathcal A(n)}{\delta} \right)} .
\end{align*}
Using the inequality $\mathbb S_\mathcal A(n) \leq (n+1)^v$ we recover the first term of the stated bound. Furthermore, using the triangle inequality, we obtain that
\begin{eqnarray*}
    |B| &\leq& \dfrac{\sum_{i=1}^n \left|g(X_i) - g(x)\right| \mathds 1 _{ \mathcal V(x) }(X_i)}{\sum_{j=1}^n \mathds 1 _{ \mathcal V(x) }(X_j)} \\ &\leq& \dfrac{\sum_{i=1}^n \sup_{y \in \mathcal V(x)} |g(y) - g(x)| \mathds 1 _{ \mathcal V(x) }(X_i)}{\sum_{j=1}^n \mathds 1 _{ \mathcal V(x) }(X_j)} = \sup_{y \in \mathcal V(x)} |g(y) - g(x)|.
\end{eqnarray*}
Moreover, using the Lipschitz assumption, it follows that
$$|g(y) - g(x)| \leq  L(\mathcal V(x)) \|x-y\|_2  \leq L(\mathcal V(x)) \diam  (\mathcal V(x) ) , $$
which concludes the proof.
\qed

\subsection*{Proof of Theorem \ref{th2:general}}

Assume that the maximum of $P_n^X(\mathcal V(x))$ and $P^X(\mathcal V(x))$ is $P^X(\mathcal V(x)).$
We have, by assumption, for all $x\in S_X$,
$$ \frac{ n     P^X(\mathcal V(x)) }{2}  \geq 4 \log\left(\dfrac{4 (2n+1) ^v }{\delta} \right). $$
 Using that $1-1/\sqrt 2 \geq  2/3$, we deduce that
$$ \frac 2 3 \leq 1 -   \sqrt{ \frac{4 \log\left(\dfrac{4 (2n+1) ^v }{\delta} \right) }{n     P^X(\mathcal V(x))} }.$$
Hence, using Theorem \ref{th:vapnik_normalized}, we obtain that with probability $1-\delta$, for all $x\in S_X$,
$$ P_n^X(\mathcal V(x)) \geq P^X(\mathcal V(x)) \left(1-\sqrt { \frac{4 \log( 4 (2n+1) ^v / \delta) }{ n P^X (\mathcal V(x))} } \right)\geq \frac 2 3P^X (\mathcal V(x))  \geq \frac 2 3\ell(x) \lambda (\mathcal V(x) ).$$
Now, if the maximum of $P_n^X(\mathcal V(x))$ and $P^X(\mathcal V(x))$ is $P_n^X(\mathcal V(x))$, then we have
$$ P_n^X(\mathcal V(x) ) \geq P^X(\mathcal V(x) )  \geq \ell(x) \lambda (\mathcal V(x) ).$$ 
Using Theorem \ref{th:general} and the previous inequality on $ P_n^X(\mathcal V(x) )$ yields the result.
\qed

\subsection*{Proof of Proposition \ref{link_beta_gamma}}

Let $A$ be a hyper-rectangle. We use the shortcut $h_- $ and $h_+$ for $h_- (A) $ and $ h_+(A)$, respectively. The first statement is a consequence of 
$ \diam (A) \leq \sqrt{d} h_+ $ 
and 
$ \lambda (A) \geq h_- ^d $, as using $\beta $-shape regularity, we obtain
$$ \diam (A)\leq \sqrt{d} \beta h_- \leq \sqrt{d} \beta \lambda (A) ^{1/d}.$$
    The second statement can be obtained as follows. Since $\diam (A) \geq h_ +$ and $ \lambda (A)^{1/d} \leq h_+ ^{1 -1/d}  h_- ^{1/d}$ we find
 \begin{align*}
      \gamma^{1/d} \geq \frac{ \diam (A)}{ \lambda (A)^{1/d} } \geq \frac{ h_ + }{  h_+ ^{1 -1/d}  h_- ^{1/d} } = \left(\frac{  h_ +}{   h_-  } \right)^{1/d}.
\end{align*}\qed

\subsection*{Proof of Proposition \ref{contre_ex}}

Let $V_0 =  \mathcal V(0) $. Define
$$ W =  \frac{\sum_{i=1} ^n  (Y_i - g(X_i) ) \ind_{ V_0} (X_i)}{\sum_{i=1} ^n    \ind_{V_0 } (X_i)}  $$
and 
 $$B = \frac{\sum_{i=1} ^n  g(X_i)   \ind_{V_0 } (X_i)}{\sum_{i=1} ^n    \ind_{V_0} (X_i)}. $$
We have, since $g(0) = 0$, 
$$\hat g(0) - g(0) =  W + B , $$
and by conditional independence
$$ \EE [ (\hat g_{\mathcal V} (0 ) -  g(0 ))^2 ]  = \EE[W^2] + \EE [B^2]. $$
The lower bound for $W$ can be obtained relying on \ref{cond:epsilon}. We have
$$ \EE [W^2 |X_1,\ldots, X_n] = \frac{\sigma^2}{\sum_{i=1} ^ n \ind_{ V_0 } (X_i) } , $$
and then, taking the expectation and using Jensen's inequality, we get
$$  \EE [W^2 ] \geq \sigma^2  \EE \left[ \sum_{i=1} ^ n \ind_{V_0 } (X_i) \right]^{-1} = \sigma^2 ( n \lambda (V_0) ) ^{-1}.$$
 Let  $V_1 = \prod_{k=1}^d [h_k/2 , h_k] \subset V_0$. We have, $a_0: = \sum_{i=1} ^n \ind_{V_0} (X_i)\geq \sum_{i=1} ^n \ind_{{ V_1}} (X_i) := a_1$. It follows that
\begin{align*}
   B \geq \frac{\sum_{i=1} ^n g(X_i) \ind_{V_1} (X_i)}{\sum_{i=1} ^n \ind_{V_0} (X_i)}
    &\geq \frac{1}{2} ( h_1+ \dots + h_d ) \frac{a_1}{a_0} \\ &\geq \frac{c}{2}  \diam( V_0) \ind_{a_1\geq c a_0 } ,
\end{align*}
where the previous inequality is valid for any $c>0$. This implies that, for any $c>0$,
$$ \EE[B^2 ] \geq \frac{c^2}{4}  \diam _1 ( V_0) ^2 \,  \PP(  {a_1\geq ca_0 } ).$$
Let us now look for a suitable choice of constant \( c > 0\). From Theorem \ref{lemma=chernoff}, one has that with probability at least $1 - 2\delta = 1/2$,
$$ \frac{a_1}{a_0} \geq \frac{P^X({ V_1})}{P^X( V_0)} \frac{\left( 1 - \sqrt {2\log(4 ) / (n P^X( V_1)) }  \right )}{\left( 1 + \sqrt {3\log(4 ) / (n  P^X({  V_0 })})  \right )}.$$ Furthermore, note that $\lambda({ V_1}) = \prod_{k=1}^d h_k/2 = 2^{-d} \prod_{k=1}^d h_k = 2^{-d} \lambda( V_0)$ and $P^X({  V_k}) = \lambda({  V_k})$ for each  $k \in \{0,1\}$. Note also that we necessarily have $n \prod_{k=1}^d h_k \geq  2^{d+3} \log(4) \geq 3 \times 2^{d+1} \log(4) \geq 3 \times 4 \log(4).$ This ensures also that the numerator is positive. As a consequence, we find that, with probability at least $1/2$,
\begin{align*}
    \frac{a_1}{a_0} &\geq 2^{-d} \, \dfrac{1 - \sqrt{\dfrac{2^{d+1} \log(4)}{n \prod_{k=1}^d h_k}}}{1+ \sqrt{\dfrac{3 \log(4)}{n \prod_{k=1}^d h_k}}} \geq 2^{-d} \dfrac{1 - 1/2}{1 + 1/2} = \dfrac{2^{-d}}{3} := c.
\end{align*}
Thus, we have obtained that
\begin{align*}
    \EE [ (\hat g_{\mathcal V} (0 ) -  g(0 ))^2 ]  &= \EE[W^2] + \EE [B^2] \\ 
    &\geq \dfrac{\sigma^2}{n \lambda (  V_0)} + \dfrac{ c ^2}{4} 
     \diam ( V_0)^2 \times \dfrac 1 2  \\
    &\geq \dfrac{\sigma^2}{n \lambda (  V_0)} + \dfrac{(c \gamma)^2}{8} \lambda ( V_ 0)^{2/d}. 
\end{align*}
where $\gamma = \overline \gamma ^{1/d} $. Let \( a_1 \) and \( a_2 \) be positive real numbers. By studying the function \( \psi : x \mapsto a_1x^{-d} + a_2 x^2 \) on \(\mathbb R_+^*  \), we notice that \( \psi \) has global minimum achieved at \( x_m = (a_1d/(2a_2))^{1/(d+2)} \). This implies that \begin{align*}
\min_{x > 0} \psi (x) &\geq x_m^2 a_2 \left( \frac{ a_1 } {a_2x_m^{d +2 } }  + 1\right) 
\\ &= \left(\dfrac{a_1d}{2a_2}\right)^{2/(d+2)} a_2 \left( \dfrac{2}{d} +1  \right) 
\\ &= \left(\dfrac{a_2^{d/2}  a_1  d}{2}\right)^{2/(d+2)}  \left( \dfrac{2}{d} +1 \right)
\end{align*}
Now, setting \( a_1 = \sigma^2 n^{-1} \), \( a_2 = (c \gamma)^2/8 \), we find \begin{align*}
    \EE [ (\hat g_{\mathcal V} (0 ) -  g(0 ))^2 ] \geq \psi ( \lambda (\mathcal V)^{1/d}) \geq
    \left(\dfrac{\sigma^{2} d \, (c \gamma)^{d}}{2 (2\sqrt{2})^d \, n}\right)^{2/(d+2)} \left(1 + \dfrac{2}{d}\right)  = C_d ^2 \left(\dfrac{\bar\gamma \sigma^2}{n}\right)^{2/(d+2)}
\end{align*} where $C_d =  \sqrt{{2}/{d} + 1} \, \left({ d}/{2}\right)^{1/(d+2)} \left( {2^d \sqrt{72}}\right)^{-d/(d+2)}.$
\qed

\subsection*{Proof of Theorem \ref{main_result_localizing_map}}

By assumption, there is $(a_-,a_+) $ such that $ 0< a _- \leq 1\leq   a _+ < +\infty $ and for all $x \in S_X$,
$$ \lambda(\mathcal V(x))  a_-  \leq \left( \frac{\log \left({(n+1)^v}/{\delta} \right)}{n} \right)^{d/(d+2)}\leq a_ + \lambda(\mathcal V(x)).$$
According to Theorem \ref{th2:general}, the $\gamma$-SR assumption, we obtain with probability at least $1 - 3\delta $, for all $x\in S_X$
\begin{align*}
        | \hat g_{\mathcal V}(x) - g(x)| &\leq  \sqrt{\frac{ 3 \sigma^2 \log\left( \frac{(n+1)^v}{\delta} \right) }{n  \ell(x) \lambda ( \mathcal V (x) )   }} + L(\mathcal V(x) ) \diam (\mathcal V(x) ) \\ 
        &\leq  \sqrt{\frac{ 3 \sigma^2 \log\left( \frac{(n+1)^v}{\delta} \right) }{n  \ell(x) \lambda ( \mathcal V (x) )   }} + L(\mathcal V(x) )\gamma^{1/d}  \lambda ( \mathcal V (x) )^{1/d} \\ &\leq  \sqrt{\frac{ 3 \sigma^2   \lambda ( \mathcal V (x) )  ^{(d+2)/d}a_+ ^{(d+2)/d}}{ \ell(x) \lambda ( \mathcal V (x) )   }} + L(\mathcal V(x) )\gamma^{1/d}  \lambda ( \mathcal V (x) )^{1/d} \\ &\leq  \left(  \sqrt{\frac{ 3 \sigma^2 a_+ ^{(d+2)/d} }{ \ell(x)}} + L(\mathcal V(x) )\gamma^{1/d} \right)  \lambda ( \mathcal V (x) )^{1/d} \\ &\leq \left( \sqrt{\frac{ 3 \sigma^2 a_+ ^{(d+2)/d}}{ \ell(x)}} + L(\mathcal V(x) )\gamma^{1/d} \right) 
 \left(\frac{\log\left( \frac{(n+1)^v}{ \delta} \right) }{ n} \right) ^{1/(d+2)} a_- ^{-1/d}.
\end{align*}
The result follows by taking care that $a_+ ^{(d+2)/d} \leq a_+^3 $  and $ a_- ^{-1/d}\leq a_-^{-1} $ which means that the universal constant in the upper bound can be taken as $ a_+^{3/2} / a_- $. \qed

\subsection*{Proof of Theorem \ref{th:NN}}

For any $x\in S_X$, define $ \tau(x)^d  = 2 k /  (n \ell(x) ) $  and check that  $\tau(x)^d \leq T_0^d$. Using \ref{cond:density_XNN} we obtain 
   \begin{align*}
    \forall x\in S_X,\qquad   n P^X(B (x, \tau(x)) ) \geq n \ell(x) \tau(x)^d   = 2  k .
   \end{align*}
   Next from Theorem \ref{th:vapnik_normalized}, and using that the set of all balls in $\mathbb R^d$, denoted by $\mathcal{A}$, has Vapnik dimension $d+1$ so that $\mathbb S_\mathcal A(2n)\leq (2n+1)^{(d+1)} $,  we deduce that with probability at least $1-\delta$, for all $x \in S_X,$
   $$n P_n^X (B (x, \tau(x)) ) \geq n P^X(B (x, \tau(x)) ) - \sqrt { n P^X(B (x, \tau(x)) )  4\log( 4 (2n+1)^{(d+1)}  / \delta)   } . $$
   Note that $x\mapsto x - \sqrt{ x\ell} $ is increasing whenever $ x\geq \ell /4$. Since, by assumption on $k$, 
   $$\forall x\in S_X,\quad  n P^X(B (x, \tau(x)) ) \geq  2k \geq 16\log( 4 (2n+1)^{(d+1)}  / \delta) \geq \log( 4 (2n+1)^{(d+1)}  / \delta).   $$
   We obtain that, with probability at least $1-\delta$,
   $$ \forall x\in S_X,\qquad    n P_n^X (B (x, \tau(x)) ) \geq 2k - \sqrt {8k  \log( 4  (2n+1)^{(d+1)}  / \delta)   }  .$$
   Now using again that $k \geq  {  8 \log( 4 (2n+1)^{(d+1)} / \delta)   }  $, we find that with probability at least $1-\delta$
   $$ \forall x\in S_X,\qquad  nP_n^X(B (x, \tau(x)))\geq k .$$
   However, for each $x\in S_X$, $\hat \tau _{n,k}(x) $ is defined as the smallest such value of $\tau $. Therefore, we obtain that for all $x \in S_X$, $ \hat \tau_{n,k} (x)  \leq \tau (x)$. As a consequence, we have shown that, with probability at least $1-\delta$, 
   $$ \forall x\in S_X,\qquad  \hat \tau_{n,k}(x) ^d \leq \frac{ 2 k  }{ n \ell(x)  }.$$
The result then follows from applying Theorem \ref{th:general}. The variance term is obtained just noting that $n P_n^X (\mathcal V (x) ) = k$ and $v = d+1 $ because the local map is valued in the collection of balls which VC dimension is given in \cite{wenocur1981some}. For the bias we use the Lipschitz condition and the inequality above since the \(\ell^2\)-diameter is twice the radius \(\hat \tau_{n,k}(x)\), which gives the upper bound with probability at least \(1 - 3\delta\).

\subsection*{Proof of Corollary \ref{corknn}}

By assumption, there is $(a_-,a_+) $ such that $ 0< a _- \leq 1\leq   a _+ < +\infty $ and $$  k \, a_-  \leq n^{2/(d+2)} \log((n+1)^{d+1}/\delta)^{d/(d+2)} \leq a_ + \, k.$$
When $n$ is large enough, $k$ satisfies $ 8 \log( 4 (2n+1)^{(d+1)} / \delta) \leq k \leq T_0^d n \ell(x) / 2.$
According to Theorem \ref{th:NN}, we have the following inequalities with probability at least $1-3\delta$, for all $x\in S_X$, \begin{align*} | \hat g_{\mathcal V}(x) - g(x)| \quad &\leq \quad \sqrt{\frac{2\sigma^2\log ( {(n+1)^{d+1}/\delta })  a_+ }{ n^{2/(d+2)} \log((n+1)^{d+1}/\delta)^{d/(d+2)} }} \\ &\qquad \quad + \quad  2 L(\mathcal{V}(x)) \left( \frac{2 n^{2/(d+2)} \log((n+1)^{d+1}/\delta)^{d/(d+2)} }{n  \ell(x) a_- } \right)^{1/d} \\ &\leq \quad c \left( \dfrac{\log((n+1)^{d+1}/\delta)}{n}\right)^{1/(d+2)}  \dfrac{\sqrt{a_+}}{a_-}\end{align*} where $c = \sqrt{2\sigma^2} + 2 L(\mathcal{V}(x)) \left(2/\ell(x) \right)^{1/d}$.   Moreover, if $\ell$ is bounded below uniformly on $S_X$ by $b > 0$, we have $c\leq \sqrt{2\sigma^2} + 2 L \left(2/b\right)^{1/d}$. \qed

\subsection*{Proof of Theorem \ref{th:main_th_proto}}

    Let \(x \in S_X\). We write the bias-variance decomposition 
$\hat g_{\mathcal V}(x) - g(x) = V + B$, where \begin{align*}
    V := \dfrac{\sum_{i=1}^n \varepsilon_i \ind_{ \mathcal V(x) }(X_i)}{\sum_{j=1}^n \ind_{ \mathcal V(x) }(X_j)} \qquad \text{and}\qquad
       B := \dfrac{\sum_{i=1}^n \left(g(X_i) - g(x)\right) \ind_{ \mathcal V(x) }(X_i)}{\sum_{j=1}^n \ind_{ \mathcal V(x) }(X_j)}.
\end{align*}
Each of the above terms will be treated in two independent propositions.

\begin{proposition}\label{prop_lemme}
    Let $x\in S_X$ and assume that \ref{cond:D}, \ref{cond:D1}, \ref{cond:reg2} and \ref{cond:reg3} are fulfilled.
Let $\delta \in (0,1/4)$. If $n \geq 1,m \geq 3$ and $$ \dfrac{n}{m} \geq \dfrac{ 8 \psi_d \log(1/\delta)}{ \delta^d},$$
with $\psi_d = (18Md)^d (c_d b)^{-d}$,
    then with probability at least $1- 4\delta$,
    $$\left| \dfrac{\sum_{i=1}^n \varepsilon_i \ind _{ \mathcal V(x) }(X_i)}{{\sum_{j=1}^n \ind _{ \mathcal V(x) }(X_j)}} \right| \leq\sqrt{\dfrac{ 4 \psi_d  \sigma^2   \log(1/\delta)}{\delta^d }} \sqrt{\dfrac{m}{n}}.$$   
\end{proposition}

\begin{proof}
The proof is in two steps. As a first step we show that with probability at least $1-\delta$,
\begin{align*}
n  P^X (\mathcal{V}(x) ) \geq \dfrac{n}{m} \psi_d^{-1} \, \delta^d   ,
\end{align*}
with \( \psi_d \) defined in the statement.
As a second step, we rely on Lemma \ref{sousgauss} to obtain the stated upper bound. 

\noindent \textit{Step 1:} Invoking \ref{cond:reg2}, we apply Lemma \ref{Z vers X}, (a) and (b),  to $Z \sim X$ and $W := \|Z - x\|$
 to obtain that $f_W(\rho) \leq c_2 \rho^{d-1} $ whenever $\rho\leq \rho_0$ and $f_W (\rho ) \leq c_2 \rho_0^ {d-1} $ for all $\rho \geq 0$, with $c_2 = Md V_d$. Similarly, we apply Lemma \ref{Z vers X} (c) in light of assumption \ref{cond:reg2} to get that $ F_W(\rho) \geq c_1 \rho ^d$, for all $\rho \leq \rho_0 = T_0$, with $c_1 = c_d b V_d$. This allows to  apply Lemma \ref{condition F vers kappa} with \( c_1 = c_d b V_d\), \(c_2 = Md V_d\), \(U = c_2 \rho_0^{d-1} \), to obtain the inequality, for all $t\geq 0$,
\[
f_W(t) \leq c_0 F_W(t)^\kappa
\]  
where \( \kappa = 1 - 1/d \) and $ c_0 =  c_1^{-\kappa} \max(c_2,  c_2 {(\rho_0/T_0)^{d-1}}) = c_1^{-\kappa} c_2 $. For all $i=1,\ldots, m$, let $ W_i =\| Z_i - x\| $ and $W_{(i)} $ the ordered statistics $ W_{(1)} \leq W_{(2)} \leq  \ldots \leq   W_{(m)}.$ 
We are now in position to apply Lemma \ref{kappa} with \( \kappa = 1 - 1/d \) and  $ c_0 = c_1^{-\kappa} c_2 $ as defined above, to obtain that, with probability at least \(1 - \delta\), 
$$W_{(2)} -W_{(1)}  \geq    C^{-1} \delta  m^{ -1/d } \geq \overline{C}^{-1} \delta  m^{ -1/d } ,$$ where $C = c_0 \Gamma(2-1/d) 3^{2-1/d}$ and   $\overline{C}= 9 c_0 = 9Md V_d^{1/d} {(c_d b)^{-1+1/d}}  $,  satisfy $\overline{C} \geq C$ since \( \Gamma(2-1/d)  \leq 1 \) and $3^{2-1/d}\leq 9 $. Moreover, since \( f_X \) has compact support included in $B (0, \rho_0)$, we have for all $(i,j) \in \{1, \dots, m\}^2$ 
\[  W_{(2)} -  W_{(1)}  \leq \| Z_i - Z_j \| \leq \|Z_i\| + \| Z_j\| \leq 2 \rho_0 = 2T_0.\]  
Recall that we have shown that $ c_1 \rho ^d \leq  F_W(\rho) = P^Z(B(x,\rho)) = P^X(B(x,\rho))$, for all $\rho \leq \rho_0 = T_0$. Thus, we can apply Lemma \ref{PV} with $c_3 = c_1$, \( T_1 = T_0 \) and with the distribution $P = P^X$, which yields, with probability at least $1-\delta$,
\[
 P^X(\mathcal{V}(x) )\geq \frac{c_1}{2^d}   (W_{(2)} -  W_{(1)}) ^d \geq \frac{c_1}{2^d \overline{C}^{d}}  \delta^d  m^{ -1 }= \psi_d^{-1} \delta^d  m^{ -1 },
\]  
where 
\( \psi_d =  c_1^{-1} (2\overline{C})^{d} \). 
Note in particular that, as soon as \( n\geq 8\psi_dm\log(1/\delta)/ \delta^d \), we get the inequality  
\[
n P^X(\mathcal{V}(x) ) \geq \dfrac{n}{m} \psi_d^{-1} \, \delta^d \geq 8\log(1/\delta),
\]  
that is valid with probability at least $1-\delta$. 

\noindent \textit{Step 2:} Let $E_1$ be the event from previous equation.  Let $E_2 $ be the event such that
$$\left| \dfrac{\sum_{i=1}^n \varepsilon_i \ind _{ \mathcal V(x) }(X_i)}{{\sum_{j=1}^n \ind _{ \mathcal V(x) }(X_j)}}  \right| \leq \sqrt{\dfrac{4 \sigma^2 \log(1/\delta)}{n P^X(\mathcal{V}(x))} } .$$
It is easy to see that $E_1$ and $E_2 $ imply that 
\[
\left| \dfrac{\sum_{i=1}^n \varepsilon_i \ind _{ \mathcal V(x) }(X_i)}{{\sum_{j=1}^n \ind _{ \mathcal V(x) }(X_j)}}  \right| \leq \sqrt{\dfrac{4\sigma^2 \log(1/\delta)}{n P^X(\mathcal{V}(x))}} \leq \sqrt{\dfrac{4\psi_d \sigma^2 \log(1/\delta)}{\delta^d}} \sqrt{\dfrac{m}{n}}.
\]
It remains to check that $\mathbb P ( E_1\cap E_2) \geq 1-4\delta$. Note that 
$A\cup B = A \cup  (A^c\cap B)  $
which, when applied to $ A = E_1^c$ and $B = E_2^c $, gives $\mathbb P ( E_1^c\cup E_2^c) = \mathbb P (  E_1^c) + \mathbb P ( E_1\cap E_2^c)   $. The first term $ \mathbb P (  E_1^c)$ is smaller than $\delta$ as shown before in Step 1.
 According to assumption \ref{cond:reg3} and Lemma \ref{sousgauss}, we have $\mathbb P ( E_1\cap E_2^c | Z_{1}, \dots, Z_m) $ is smaller than $3\delta$. Integrating with respect to $Z_{1} , \ldots, Z_{m}$, we obtain $\mathbb P ( E_1\cap E_2^c )\leq 3\delta $. 
\end{proof}

\begin{proposition}
Suppose that \ref{cond:D}, \ref{cond:D1}, \ref{cond:reg2} and \ref{cond:reg4} hold true. Then, for all $m\geq 1$, $\delta \in (0,1)$ such that $ 32 d \log(12m / \delta )  \leq T_0^d  m b c_d V_d $, it holds, with probability at least $1-\delta$,
$$ \left|  \dfrac{\sum_{i=1}^n \left(g(X_i) - g(x)\right) \ind_{ \mathcal V(x) }(X_i)}{\sum_{j=1}^n \ind_{ \mathcal V(x) }(X_j)} \right| \leq  2 L(\mathcal{V}(x))  \left(\frac{ 32d \log(12m/\delta) }{ mb c_d V_d }  \right)^{1/d} .$$
\end{proposition}

\begin{proof}
Using triangle inequality we obtain
    \begin{align*}
     &\left|  \dfrac{\sum_{i=1}^n \left(g(X_i) - g(x)\right) \ind_{ \mathcal V(x) }(X_i)}{\sum_{j=1}^n \ind_{ \mathcal V(x) }(X_j)} \right|\\ 
     &\leq \dfrac{\sum_{i=1}^n \left|g(X_i) - g(x)\right| \ind_{ \mathcal V(x) }(X_i)}{\sum_{j=1}^n \ind_{ \mathcal V(x) }(X_j)} \\ 
     &\leq \dfrac{\sum_{i=1}^n \sup_{y \in \mathcal V(x)} |g(y) - g(x)| \ind_{ \mathcal V(x) }(X_i)}{\sum_{j=1}^n \ind_{ \mathcal V(x) }(X_j)} = \sup_{y \in \mathcal V(x)} |g(y) - g(x)|.
\end{align*}
Moreover, using the Lipschitz assumption, it follows that
$$|g(y) - g(x)| \leq  L(\mathcal V(x)) \|x-y\|_2  \leq L(\mathcal V(x)) \diam  (\mathcal V(x) ).$$
Thus, we obtain
\begin{equation}\label{d1}
     \left|  \dfrac{\sum_{i=1}^n \left(g(X_i) - g(x)\right) \ind_{ \mathcal V(x) }(X_i)}{\sum_{j=1}^n \ind_{ \mathcal V(x) }(X_j)} \right| \leq L(\mathcal{V}(x)) \diam (\mathcal V (x) ).
\end{equation}
Suppose that $x$ and $y$ belong to the Voronoi cell of $Z_i$. That is $ i = \argmin _{j=1,\ldots, m} \| x- Z_j\| = \argmin_{j=1,\ldots, m} \| y- Z_j\| $. Hence 
$$\|x- y\| \leq  \| x- Z_i \| + \| y-Z_i\|=
\min_{j=1,\ldots, m}  \| x- Z_j \| + \min_{j=1,\ldots, m} \| y-Z_j\|.$$
Therefore 
$$  \diam (\mathcal V (x) ) \leq 2 \sup_{x\in S_Z} \hat \tau_1 (x) \leq 2 \sup_{x\in S_Z} \hat \tau_{k} (x) $$
with $k=16d \log(12m/\delta)$ and $\hat \tau_{k} (x) $ is $k$-NN radius 
\begin{align*}
\hat \tau_{k}(x) =  \inf   \{ \tau\geq 0 \, :\,  \sum_{i=1}^m \ind _{ B(x,\tau) }(Z_i) \geq k \} .
\end{align*}
Using Lemma 3 in \cite{portier2021nearest} whenever $2k \leq  T_0^d  m b c_d V_d$, we obtain that, with probability at least $1-\delta$,
\begin{equation}\label{d2}
   \diam (\mathcal V (x) ) \leq  2 \left(\frac{ 32d \log(12m/\delta) }{ mb c_d V_d }  \right)^{1/d}. 
\end{equation}
Finally, the combination of equations \eqref{d1} and \eqref{d2} yields the desired result.   
\end{proof}

Getting back to the proof of Theorem \ref{th:main_th_proto}, the upper bounds on \( V \) and \( B \) from the previous propositions yield the stated result.
\qed

\subsection*{Proof of Corollary \ref{cor:main_th_proto}}

It suffices to write the inequality \(\log(12m/\delta) \leq \log(12n/\delta)\) and then observing that the identity \(\sqrt{\log(1/\delta)/\delta^d} \sqrt{m/n} = (\log(n/\delta)/m)^{1/d}\) gives the correct order for $m$. Finally, using this choice of $m$ into the bound yields the stated result.
\qed

\subsection*{Proof of Theorem \ref{th:main_th_OptiNet}}

We start by establishing two facts that are related to the $\eta$-net construction. They will be useful to deal with the bias term (Fact 1) and the variance term (Fact 2). For \( A \subset X \) and \( \eta > 0 \), a \( \eta \)-net of \( A \) is any subset \( B \subset A \) such that the distance between any two distinct points in \( B \) is at least \( \eta \), i.e., \( \forall x, y \in B, x \neq y \Rightarrow \|x - y\| > \eta \),  and such that \( B \) is maximal with respect to this property (i.e., no point from \( A \) can be added to \( B \) without violating the previous condition). 
Let $Z(\eta) = \{ Z_i(\eta), i \in [\![1,m(\eta)]\!] \}$ be an $\eta$-net of the set $(Z_i)_{i=1,\dots,m}$. 

\noindent \textbf{Fact 1.} 
We have $\forall i \in [\![1,m]\!], \, d(Z_i, Z(\eta)) \leq \eta$. Indeed, if there exists $i \in [\![1,m]\! ]$ such that $Z_i \in Z(\eta)$ then $d(Z_i, Z(\eta)) = 0 \leq \eta$. Otherwise if $Z_i \notin Z(\eta)$ and $d(Z_i, Z(\eta)) > \eta$ this denies the fact that $Z(\eta)$ is maximal and therefore contradicts the $\eta$-net construction.

\noindent \textbf{Fact 2.} We have that $ B ( Z_k (\eta), \eta /2)  \subset V_k(\eta)$, where $V_k(\eta)$ is the Voronoi cell containing $Z_k(\eta)$ and relative to the $\eta$-net $Z(\eta)$. The result indeed simply follows from noting that
 $ B ( Z_k (\eta) , \Delta_k (\eta) /2) \subset  V_k (\eta)      $ where $\Delta _k(\eta)  = \min _{i\neq k} \|Z_i(\eta) - Z_k(\eta) \| $ is larger than $\eta $ by construction. 

Let \(x \in S_X\). We write the bias-variance decomposition 
$\hat g_{\mathcal V_\eta}(x) - g(x) = V + B$, where \begin{align*}
    V := \dfrac{\sum_{i=1}^n \varepsilon_i \ind_{ \mathcal V_\eta(x) }(X_i)}{\sum_{j=1}^n \ind_{ \mathcal V_\eta(x) }(X_j)} \qquad \text{and}\qquad
       B := \dfrac{\sum_{i=1}^n \left(g(X_i) - g(x)\right) \ind_{ \mathcal V_\eta(x) }(X_i)}{\sum_{j=1}^n \ind_{ \mathcal V_\eta(x) }(X_j)}.
\end{align*}
   We start by considering the bias term $B$. Using triangle inequality we obtain
    \begin{align*}
     \left| B\right|
     &\leq \dfrac{\sum_{i=1}^n \left|g(X_i) - g(x)\right| \ind_{ \mathcal V_\eta(x) }(X_i)}{\sum_{j=1}^n \ind_{ \mathcal V_\eta(x) }(X_j)} \\ 
     &\leq \dfrac{\sum_{i=1}^n \sup_{y \in \mathcal V_\eta(x)} |g(y) - g(x)| \ind_{ \mathcal V_\eta(x) }(X_i)}{\sum_{j=1}^n \ind_{ \mathcal V_\eta(x) }(X_j)} = \sup_{y \in \mathcal V_\eta(x)} |g(y) - g(x)|.
\end{align*}
Moreover, using the Lipschitz assumption, it follows that
$$|g(y) - g(x)| \leq  L(\mathcal V_\eta(x)) \|x-y\|_2  \leq L(\mathcal V_\eta(x)) \diam  (\mathcal V_\eta(x) ).$$
Thus, we obtain
\begin{equation*}
    |B| \leq L(\mathcal{V}_\eta(x)) \diam (\mathcal V _\eta (x) ),
\end{equation*}
we can now provide an upper bound on the diameter of $\mathcal V_\eta (x) $. Let $Z_\eta (x) $ (resp. $Z(x)$) denote the closest point to $x$ among $Z(\eta) $ (resp. $Z_1,\ldots, Z_m$). We have
$$(x,z) \in \mathcal V_\eta (x) ^2 \implies Z_\eta(x) = Z_\eta(z) $$ then $$d(x,z) \leq d(x, Z_\eta(x) ) + d(Z_\eta(x) ,z) = d(x, Z_\eta(x) ) + d(Z_\eta(z) ,z). $$
For the first term we write $d(x, Z_\eta(x) ) = d(x, Z_\eta )$ and by triangle inequality
$$d(x, Z_\eta ) = d(x, Z(x)) + d(Z(x) , Z_\eta ) \leq \sup_{x \in\mathcal V_\eta (x)  } d(x, Z(x) ) + \eta $$
using \textbf{Fact 1}. It follows that the diameter is such that 
$$ \diam (\mathcal V _\eta (x) ) \leq 2\sup_{x \in \mathcal V_\eta (x)  } d(x, Z(x) ) + 2\eta. $$
The first above term is bounded by $2(32d\log(12m/\delta) / [mb c_d V_d])^{1/d}$ with probability at least $1-\delta$, from Lemma 3 in \cite{portier2021nearest} whenever $32d\log(12m/\delta) \leq  T_0^d  m b c_d V_d$. As a consequence, we have shown that, with probability at least $1-\delta$,
$$ |B|\leq 2L(\mathcal V_\eta(x))  \left( \left( \frac{32d\log(12m/\delta)}{  mb c_d V_d}\right)^{1/d}+ \eta \right).$$

Let us now deal with the variance term $V$. As soon as $nP^X(\mathcal{V}_\eta(x)) \geq 8 \log(1/\delta)$ we can apply Lemma \ref{sousgauss} to obtain the following inequality which holds with probability at least $1 - 3\delta$
\begin{equation}\label{var_eq}
    |V| = \left| \dfrac{\sum_{i=1}^n \varepsilon_i \ind _{\mathcal{V}_\eta(x) }(X_i)}{{\sum_{j=1}^n \ind _{\mathcal{V}_\eta(x)}(X_j)}} \right|     \leq \sqrt{\dfrac{4\sigma^2 \log(1/\delta)}{n P^X(\mathcal{V}_\eta(x))} }.
\end{equation} 
   We can therefore conclude by obtaining a lower bound on $ P^X (\mathcal V_\eta (x) ) $. Let $V_k(\eta)$ denote the Voronoi collection of $X(\eta)$.  We have  using \textbf{Fact 2},
\begin{align*}
    P^X (\mathcal V_\eta (x) )  
     = \sum _{k=1} ^ m \ind_{ V_k(\eta) }(x) \, P^X[ \mathcal{V}_\eta(x) ]  &= \sum _{k=1} ^ m \ind_{V_k(\eta) }(x) \, P^X[  V_k (\eta) ] \\ &\geq \sum _{k=1} ^ m \ind_{V_k(\eta) }(x) \, P^X[ B ( Z_k(\eta) , \eta /2) ].
\end{align*}
 Moreover if $\eta\leq 2T_0$, using \ref{cond:reg2} to obtain that for all $z \in S_Z$, $$P^X(B ( z , \eta /2)) \geq b \lambda(S_Z \cap B ( z , \eta /2)) \geq b c_d  \lambda(B ( z , \eta /2)) = b c_d V_d \eta^d/2^d,$$ and therefore $P^X (\mathcal{V}_\eta(x) )  \geq b {c_d} V_d \eta^d / {2^d}.$ We then find that, using \eqref{var_eq}, the following inequality holds with probability at least $1-3\delta$, 
 $$|V| \leq \sqrt{\dfrac{2^{d+2} \sigma^2 \log(1/\delta)}{n b c_d \eta^d} }.$$
 Combining the obtained bound on $|B|$ and $|V|$ yields the stated result.

\qed

\subsection*{Proof of Theorem \ref{th_cart_like1}}

The proof follows from a straightforward application of the next result, which is stated for general local regression maps. 

\begin{theorem}\label{th_cart_like}
Let $S_X=[0,1]^d$, $\delta \in (0,1/3) $, $n\geq 1$, $d\geq 1$, and $m\geq 4\log(4(2n+1)^{2d}/\delta )$. Suppose that \ref{cond:D}, \ref{cond:density_XCART}, \ref{cond:epsilon} and \ref{cond:reg4} are fulfilled. Let $\beta \geq 2$ and suppose that $\mathcal V$ is a local regression map valued in the set of hyper-rectangles contained in $S_X$, for all $V\in \left\{ \mathcal V(x) \, :\, x\in \mathbb R^d \right\}$,
    \begin{align*}
        h_+(V) \leq \beta h_- (V)
\quad \text{and}\quad 
       n P_n^X( V) \geq  m,
    \end{align*}
    then we have, with probability at least $1-3\delta$, for all $x \in S_X$,
    \begin{align*} 
            |  \hat g_\mathcal V (x) - g(x)  | \leq \sqrt { \frac{2 \sigma ^2 \log((n+1) ^{2d}/\delta) } {  m}} + L(\mathcal{V}(x)) \beta \sqrt{d} \left(\frac{5 m}{ n b}\right)^{1/d}.
    \end{align*}
    
\end{theorem}

    Note that,  when growing the tree, the constraint $h_+(V) \leq \beta h_- (V) $ can never be a stopping criterion because one can always select the largest side and split it in the middle.
    When the tree is fully grown according to the prescribed rules,  acceptable splits are no longer possible. Therefore any $V $ satisfies 
    $$ 2 m \geq   nP_n^X(V) \geq  m.$$
Since the Vapnik dimension of hyper-rectangles is $ v = 2d$, using Assumption \ref{cond:density_XCART} and Theorem \ref{th:vapnik_normalized}, then for all $\delta \in (0,1)$ and $m\geq 4\log(4(2n+1)^{2d}/\delta )$, we obtain with probability at least $1 - \delta$,
$$b h_-^d \leq P^X(V) \leq \dfrac{4}{n} \log\left(\dfrac{4(2n+1)^{2d}}{\delta} \right) + 2 P_n^X(V) \leq \frac{m}{n} + \frac{4m}{n} = \frac{5m}{n}.$$
 In addition,
$$\diam(V) \leq \sqrt{d} h_+ \leq \sqrt{d} \beta h_- \leq  \sqrt{d} \beta \left(\dfrac{5 m}{n b}\right)^{1/d}.$$
It remains to apply Theorem \ref{th:general} and to use that $n P_n^X(V) \geq m$ for the variance term to get the stated result.
\qed


\subsection*{Proof of Corollary \ref{cor_cart_like1}}
We apply Theorem \ref{th_cart_like1} to the stipulated choice of $m$. As the proof follows the same steps as that of Corollary \ref{corknn}, the details are left to the reader.
\qed

\subsection*{Proof of Proposition \ref{prop_leb_vol_inv}}

Recall that $S_X= [0,1] ^d$, so each side length of the initial cell is equal to one. For any $k\in \{1,\ldots,d\}$, we have the following formula for the length $h_k$ of the side $k$ of the cell $\mathcal V(x,(D_i,S_i)_{i=1}^N)$,
    \[
    h_k= \prod_{i=1}^n  \tilde S_i^{B_i^{(k)}},
    \]
    where $B_i^{(k)}=\mathbb{I}_{D_i=k}$. Taking the product over $k$ gives
    \begin{align*}
         \prod_{k=1}^d h_k 
         &= \prod_{k=1}^d \prod_{i=1}^n \tilde S_i^{B_i^{(k)}} 
         =  \prod_{i=1}^n \tilde S_i^{\sum_{k=1}^d B_i^{(k)}}.
    \end{align*}
    The result follows by noting that $\lambda(\mathcal V(x,(D_i,S_i)_{i=1}^N))=\prod_{k=1}^d h_k$ and that for any $i$, $\sum_{k=1}^d B_i^{(k)} = 1$, the latter identity simply corresponding to the fact that exactly one side of the cell is split at each step.
\qed

\subsection*{Proof of Proposition \ref{prop:diamURT2}}

First notice that, by a union bound and symmetry in the directions, we have \[ \pr(\diam (\mathcal{V}(x))\geq t) \leq d\pr \left(h_1 \geq \frac{t}{\sqrt{d}}\right). \]
Furthermore, by denoting $B_i^{(1)}=\ind_{D_i=1}$ as in the proof of Proposition \ref{prop_leb_vol_inv}, we get for any $r\in (0,1)$ and $\lambda>0$,
\begin{align*}
    \pr(h_1 \geq r^N) & = \pr\left(\prod_{i=1}^N U_i^{B_i^{(1)}} \geq r^N\right)\\
    & \leq \mathbb E\left[\left(\frac{\prod_{i=1}^N U_i^{B_i^{(1)}}}{r^N}\right)^\lambda \right]=\left(\frac{\mathbb E\left[{U_1^{\lambda B_1^{(1)}}}\right]}{r^{\lambda}}\right)^N. 
\end{align*}
It holds
\[
 \mathbb E\left[{U_1^{\lambda B_1^{(1)}}}\right]= \frac{1}{d(1+\lambda)}+1-\frac{1}{d}.
\]
Hence,
\[
 \pr(h_1 \geq r^N) \leq \left(\frac{1}{d(1+\lambda)}+1-\frac{1}{d}\right)^N r^{-\lambda N}.
\]
First note that it suffices to optimize the bound for $N=1$. Let us denote 
\[
Q(\lambda)= \frac{1}{d(1+\lambda)}+1-\frac{1}{d}
\]
and 
\[
h(\lambda)= Q(\lambda) r^{-\lambda}.
\]
Denote $r=(1/e)^{1/d-\theta}$ for $\theta>0$, then
\begin{align*}
    h(\lambda)
    & =   \exp\left(\lambda\left(\frac{1}{d}-\theta\right)+\log\left(1-\frac{\lambda}{d(1+\lambda)}\right)\right)\\
    & \leq  \exp\left(\lambda\left(\frac{1}{d}-\theta\right)-\frac{\lambda}{d(1+\lambda)}\right) \\
    & \leq  \exp\left(\lambda\left(\frac{1}{d}-\theta\right)-\frac{\lambda(1-\lambda)}{d}\right) \\
    & =   \exp\left(-\lambda\left(\theta-\frac{\lambda}{d}\right)\right).
\end{align*}
By taking $\lambda= d\theta/2$, we get
\[
 \pr(h_1 \geq e^{N(\theta-1/d)}) \leq e^{-d\theta^{2}N/4}.
\]
Then for all $\theta > 0$, we obtain
\[
 \pr(\diam (\mathcal{V}(x))\geq \sqrt{d}e^{N(\theta-1/d)}) \leq d\pr \left(h_1 \geq e^{N(\theta-1/d)}\right) \leq d e^{-d\theta^{2}N/4}.
\]
We now consider the lower bound. We proceed in the same way as before. By a union bound and symmetry in the directions, we have
\[
 \pr(\diam (\mathcal{V}(x))\leq t) \leq d\pr \left(h_1 \leq \frac{t}{\sqrt{d}}\right).
\]
Furthermore, for any $(r, \lambda)\in (0,1)^2$,
\begin{align*}
    \pr(h_1 \leq r^N) & = \pr\left(\prod_{i=1}^N U_i^{B_i^{(1)}} \leq r^N\right) = \pr\left(\prod_{i=1}^N U_i^{-\lambda B_i^{(1)}} \geq r^{-\lambda N} \right) \\
    & \leq \mathbb E\left[\left(\frac{\prod_{i=1}^N U_i^{B_i^{(1)}}}{r^N}\right)^{-\lambda} \right]=\left(\frac{\mathbb E\left[{U_1^{-\lambda B_1^{(1)}}}\right]}{r^{-\lambda}}\right)^N. 
\end{align*}
It holds
\[
 \mathbb E\left[{U_1^{-\lambda B_1^{(1)}}}\right]= \frac{1}{d(1-\lambda)}+1-\frac{1}{d}.
\]
Hence,
\[
 \pr(h_1 \leq r^N) \leq \left(\frac{1}{d(1-\lambda)}+1-\frac{1}{d}\right)^N r^{\lambda N}.
\]
Without loss of generality, we can optimize the bound for $N=1$. Define 
$$Q(\lambda)= \frac{1}{d(1-\lambda)}+1-\frac{1}{d} = 1 + \dfrac{\lambda}{d(1-\lambda)},$$
and
$$ h(\lambda)= Q(\lambda) r^{\lambda}.$$
Set $r=(1/e)^{1/d+\theta}$ for $\theta>0$, then for all $\lambda \in (0, 1/2)$,
\begin{align*}
    h(\lambda)
    & =   \exp\left(-\lambda\left(\frac{1}{d}+\theta\right)+\log\left(1 + \dfrac{\lambda}{d(1-\lambda)}\right)\right)\\
    & \leq  \exp\left(-\lambda\left(\frac{1}{d}+\theta\right)+ \dfrac{\lambda}{d(1-\lambda)}\right) \\
    & \leq  \exp\left(-\lambda\left(\frac{1}{d}+\theta\right)+ \frac{\lambda(1+2\lambda)}{d}\right) \\
    & =   \exp\left(-\lambda\left(\theta-\frac{2\lambda}{d}\right)\right),
\end{align*}
where in the second inequality we used the fact that $(1-\lambda)^{-1} \leq 1+2\lambda$ for $\lambda \in (0,1/2)$. By taking $\lambda= d\theta/4 \in (0, 1/2)$, we get
\[
 \pr(h_1 \leq e^{-N(\theta+1/d)}) \leq e^{-d\theta^{2}N/8}.
\]
Then for all $\theta \in (0, 2/d)$,
\[
 \pr(\diam(\mathcal{V}(x))\leq \sqrt{d}e^{-N(\theta+1/d)}) \leq d\pr \left(h_1 \leq e^{-N(\theta+1/d)}\right) \leq d e^{-d\theta^{2}N/8}.
\] \qed

\subsection*{Proof of Proposition \ref{prop:diam_vol_URT2}}

As in the proof of Proposition \ref{prop:diamURT2}, we optimize along some polynomial moments controlling the deviation probability of interest. We have $\lambda(\mathcal V(x))=\prod_{i=1}^N U_i$, which gives, for any $\alpha>1$, $\lambda \in (0,1)$,
\[
\PP(\lambda(\mathcal V(x))^{-1}\geq e^{N\alpha})\leq \EE\left[\prod_{i=1}^N U_i^{-\lambda}\right]e^{-N\alpha \lambda} = \left(\frac{e^{-\alpha \lambda}}{1-\lambda}\right)^N.
\]
By taking $\lambda =1 -1/\alpha$, we get
\[
\PP(\lambda(\mathcal V(x))^{-1}\geq e^{N\alpha})\leq (\alpha e^{1-\alpha})^N.
\]
Moreover for any $\lambda>0$, $\alpha \in (0,1)$,
\[
\PP(\lambda(\mathcal V(x)) \geq e^{-N\alpha})\leq \EE\left[\prod_{i=1}^N U_i^{\lambda}\right]e^{N\alpha \lambda} = \left(\frac{e^{\alpha \lambda}}{1+\lambda}\right)^N.
\]
By taking $\lambda =1/\alpha - 1$, this gives
\[
\PP(\lambda(\mathcal V(x))\geq e^{-N\alpha})\leq (\alpha e^{1-\alpha})^N.
\]
\qed

\subsection*{Proof of Proposition \ref{cor_diamvolURT}}

We will use the Borel Cantelli lemma together with the inequalities obtained in theorems \ref{prop:diamURT2} and \ref{prop:diam_vol_URT2}. To prove the upper bound on the diameter, we provide values $\theta_N$ leading to small enough probabilities. More precisely, by taking $\theta_N=2\sqrt{2\log(N)/(dN)}$, we get $e^{-Nd\theta_N^2/4}=N^{-2}$.   Then  
    \begin{equation*}
    \PP\left(\diam (\mathcal{V}(x))\geq \sqrt{d}e^{N(-1/d+\theta_N)}\right) \leq \dfrac{d}{N^2}.
\end{equation*} 
The Borel-Cantelli lemma then gives $$\PP\left(\liminf_{N \to + \infty} \left\{ \diam (\mathcal{V}(x)) \leq \sqrt{d} e^{N(-1/d + \theta_N)} \right\}\right) = 1.$$
This means that almost surely, beyond a certain rank, we have $$\diam (\mathcal{V}(x)) \leq \sqrt{d} e^{N(-1/d + \theta_N)}.$$
For the lower bound on the diameter, we proceed in the same way with the choice $\tilde \theta_N=4\sqrt{\log(N)/(dN)}$, or equivalently $e^{-Nd \tilde \theta_N^2/8}=N^{-2}$. We deduce that, almost surely, beyond a certain rank, $$\diam (\mathcal{V}(x)) \geq \sqrt{d} e^{-N(1/d + \tilde \theta_N)}.$$
Now, regarding the volume, we set $\alpha_N=1+2\sqrt{\log(N)/N}$ and we obtain \[
    \left(\alpha_N e^{1-\alpha_N}\right)^N = \exp\left( {-2\sqrt{N\log(N)}+N\log\left(1+2\sqrt{\log(N)/N}\right)}\right).\]
As $\log\left(1+2\sqrt{\log(N)/N}\right)=2\sqrt{\log(N)/N} - 2\log(N)/N+ O\left((\log(N)/N)^{3/2}\right)$, we get $$\left(\alpha_N e^{1-\alpha_N}\right)^N = \exp \left( -2\log(N) + O\left(\log(N)^{3/2} / \sqrt{N} \right) \right) \sim 1/N^2.$$
The Borel-Cantelli lemma gives us that almost surely, beyond a certain rank $N_0$, we have
$$\forall N \geq N_0, \qquad \frac{ \lambda(\mathcal{V}(x)) }{ e^{-N- 2\sqrt{N\log(N)}}} \geq 1.$$
For the upper bound on the volume, we set for $N \geq 9$, $\tilde \alpha_N=1 - 2\sqrt{\log(N)/N} \in (0,1)$ and we obtain \[
    \left(\tilde \alpha_N e^{1-\tilde \alpha_N}\right)^N = \exp\left( {2\sqrt{N\log(N)}+N\log\left(1-2\sqrt{\log(N)/N}\right)}\right).\]
As $\log\left(1-2\sqrt{\log(N)/N}\right)=-2\sqrt{\log(N)/N} - 2\log(N)/N+ O\left((\log(N)/N)^{3/2}\right)$, we get $$\left(\tilde \alpha_N e^{1-\tilde \alpha_N}\right)^N = \exp \left( -2\log(N) + O\left(\log(N)^{3/2} / \sqrt{N} \right) \right) \sim 1/N^2.$$
Again, the Borel-Cantelli lemma implies that almost surely, beyond a certain rank $N_0$, we have
$$\forall N \geq N_0, \qquad \frac{ \lambda(\mathcal{V}(x)) }{ e^{-N + 2\sqrt{N\log(N)}}} \leq 1.$$
Finally, the last inequality stated in Proposition \ref{cor_diamvolURT} comes readily by using the two previous inequalities on the diameter and the volume.
\qed

\subsection*{Proof of Theorem \ref{cor_URT}}

Firstly, since the local regression map is obtained from a tree construction, each element in $\mathcal V ( x) : = \mathcal V ( x, (D_i,S_i)_{i=1}^N )$ is a hyper-rectangle. Hence, in light of \cite{wenocur1981some}, it holds that the image of the resulting local map is included in the set of hyper-rectangles, that has VC dimension $v = 2d$. Hence, the local map is indeed VC. 

Secondly, let us note that assumption \ref{cond:density_XCART} implies assumption \ref{cond:density_X} with $\ell(x) = b$. We can therefore apply Theorem \ref{th2:general} pointwise for $x \in S_X$ and, with $\delta = (n+1)^{-2}$, we find that, whenever $n P^X (\mathcal{V}(x) ) / \log(n) \to \infty$, it holds that
$$ \sum_{n\geq 1} \mathbb P ( |\hat g_{\mathcal V}(x) - g(x)| > v_n ) < \infty,$$
where $$v_n = \sqrt{ 3 \sigma^2 \log( {(n+1)^{v+2}}  ) / (n  b \lambda ( \mathcal V(x)) )}  + L(\mathcal V(x)) \diam (\mathcal V(x)).$$
Applying the Borel-Cantelli Lemma, we get that with probability $1$, for $n$ large enough,
$$|\hat g_{\mathcal V}(x) - g(x)| \leq  \sqrt{\frac{ 3 \sigma^2 \log\left( {(n+1)^{v+2}} \right)}{n  b \lambda ( \mathcal V(x)) }} + L(\mathcal V(x)) \diam (\mathcal V(x)).$$
Thirdly, from Proposition \ref{cor_diamvolURT} and using that $ N = d\log(n)/(d+2)$, with probability $1$, for a sufficiently large \(n\), we have
\begin{align*}
& \lambda(\mathcal{V}(x)) \geq e^{-N - 2 \sqrt{N \log(N)}} \geq n^{-d/(d+2)} e^{- 2 \sqrt{\log(n) \log(\log(n)) }  },
\end{align*}
where we have used that \(N \leq \log(n)\). Using \ref{cond:density_XCART}, it follows that
$$ P^X (\mathcal{V}(x) ) \geq b \lambda(\mathcal{V}(x))  \geq b n^{-d/(d+2)} e^{- 2 \sqrt{\log(n) \log(\log(n)) }  }.$$
As a consequence,
$$ n P^X (\mathcal{V}(x) ) \geq b n ^{ 2 /(d+2) } e^{- 2 \sqrt{\log(n) \log(\log(n)) }  }.$$
Hence, we get that with probability $1$, $ n P^X (\mathcal{V}(x) ) / \log(n) \to \infty$. This ensures the $(\delta, n)$-large hypothesis, in order to apply Theorem \ref{th2:general}.

Fourthly, by putting together the second and third point from above, we have the following inequality, with probability $1$, for $n $ large enough and $\delta = (n+1)^{-2}$,
$$|\hat g_{\mathcal V}(x) - g(x)| \leq \sqrt{\frac{ 3 \sigma^2 \log\left( (n+1)^{v+2} \right)}{n  b \lambda ( \mathcal V(x)) }} + L(\mathcal V(x)) \diam (\mathcal V(x)).$$
This gives in virtue of Proposition \ref{cor_diamvolURT}
$$| \hat g_{\mathcal V}(x) - g(x)| \leq \sqrt{\frac{ 3 \sigma^2 \log\left( (n+1)^{v+2}\right)}{n b e^{-N - 2\sqrt{N \log(N)}}}} + L(\mathcal V(x)) \sqrt d e^{-N/d + 2\sqrt{2N \log(N)/d}}.$$
Recalling that \( N = d\log(n)/(d+2) \), we obtain
\begin{align*}
    &| \hat g_{\mathcal V}(x) - g(x)| \\ & \leq n^{-1/(d+2)} e^{\sqrt{N \log(N)}} \sqrt{\frac{ 3 \sigma^2 \log\left( (n+1)^{v+2} \right)}{b}} + n^{-1/(d+2)} L(\mathcal V(x)) \sqrt d e^{2\sqrt{2N \log(N)/d}}.
\end{align*}
Then
\[
| \hat g_{\mathcal V}(x) - g(x)| \leq n^{-1/(d+2)} e^{C_d \sqrt{N \log(N)}} \left( \sqrt{\frac{ 3 \sigma^2 \log\left( (n+1)^{v+2} \right)}{b}} + L(\mathcal V(x)) \sqrt d \right),
\]
where \(C_d = \max(1, \sqrt{8/d})\). But since $\log(n+1) \leq 2 \log(n)$ for $n\geq2$, we have
\[
\sqrt{\frac{ 3 \sigma^2 \log\left( (n+1)^{v+2}\right)}{b}} = \sqrt{\frac{ 3 \sigma^2 (v+2) \log(n+1)}{b}} \leq \sqrt{\frac{ 6 \sigma^2 (v+2) \log(n)}{b}}.
\]
Then, we set $$C = \sqrt{\frac{ 6 \sigma^2 (v+2)}{b}} + L(\mathcal V(x)) \sqrt d = \sqrt{\frac{ 6 \sigma^2 (d+1)}{b}}+ L(\mathcal V(x)) \sqrt d.$$
Additionally, since \(N \leq \log(n)\), we have 
\[
N \log(N) \leq \frac{d}{d+2} \log(n) \log(\log(n)),
\]
so we set \(c_d = \sqrt{\frac{d}{d+2}} C_d = \max\left(\sqrt{\frac{d}{d+2}}, \sqrt{\frac{8}{d+2}}\right) \leq 2 \) to obtain the desired inequality.
\qed

\subsection*{Proof of Proposition \ref{uniform tree not regular}}

At each stage, for each terminal leaf, draw uniformly $D_i $ in $\{1,\ldots, d\}$ as well as a uniform random variable $U_i$. Then we divide the cell according to coordinate $k = D_i$. The corresponding length $h_k(\mathcal V (x) )$ is then updated into $h_k(\mathcal V (x) ) U_i$ and $h_k(\mathcal V (x) ) (1-U_i)$. Note that $1-U_i$ is still uniformly distributed. As a consequence, for a given leaf, after $N$ stages,  the $k$-th length has the following representation 
$$ h _ k(\mathcal V (x) ) = U_1^{B_1^{(k)}}\times \ldots \times U_N^{B_N^{(k)}} = \exp\left( \sum_{i=1} ^ N B_i^{(k)} \log(U_i) \right)  $$
where $B_i^{(k)}  = \ind_{D_i = k }$. It follows that
\begin{align*}
 &h_+(\mathcal V (x) ) = \exp\left( \max_{k=1,\ldots, d}  \sum_{i=1} ^ N B_i^{(k)} \log(U_i) \right), \\
 &h_-(\mathcal V (x) ) =  \exp\left(\min_{k=1,\ldots, d}  \sum_{i=1} ^ N B_i^{(k)} \log(U_i) \right), 
\end{align*}
and the expression of the ratio is
\begin{align*}
 h_+(\mathcal V (x) )/h_-(\mathcal V (x) )  &= \exp \left(   \max_{1 \leq k,j \leq d} \sum_{i=1} ^N (B_i^{(k)} - B_i^{(j)} ) E_i \right) \end{align*}
 where $E_i = - \log(U_i)$ follows an exponential distribution with parameter 1.

By denoting $V_i^{k,j} = B_i^{(k)} - B_i^{(j)}$, we get
\[V_i^{k,j} = \begin{cases} 
1 & \text{with probability } 1/d \\
0 & \text{with probability } 1 - 2/d \\
-1 & \text{with probability } 1/d 
\end{cases}.\]
Note that the variables \((V_i^{k,j})_{i=1}^N\) are mutually independent because the \((D_i)_{i=1}^N\) are independent. Furthermore, since the \(U_i\)'s are independent of the \(V_i\)'s, the \(V_i^{k,j}\)'s are independent of the \(E_i\)'s.
Let \(Z_{k,j} = \sum_{i=1}^{N} V_i^{k,j} E_i\) such that
\[
\frac{h_+(\mathcal V (x) )}{h_-(\mathcal V (x) )} = \exp \left( \max_{1 \leq k,j \leq d} Z_{k,j} \right).
\]
Note that \(Z_{k,j} = - Z_{j,k}\) and \(Z_{k,k} = 0\), which gives
\[
\max_{1 \leq k,j \leq d} Z_{k,j} = \max_{1 \leq k < j \leq d} | Z_{k,j} |
\]
and thus the formula 
\[
\frac{h_+(\mathcal V (x) )}{h_-(\mathcal V (x) )} = \exp \left( \max_{1 \leq k < j \leq d} | Z_{k,j} | \right).
\]
By using the Paley-Zygmund inequality to $Z_{k,j}^2$, we get for all $\theta \in (0,1)$,
\[
\PP\left(|Z_{k,j}| \geq \sqrt{\theta} \sqrt{\mathbb{E}(Z_{k,j}^2)} \right) \geq (1-\theta)^2 \frac{\mathbb{E}(Z_{k,j}^2)^2}{\mathbb{E}(Z_{k,j}^4)}.
\]
We therefore seek to calculate the 2nd and 4th moments of \(Z_{k,j}\). 
\\ Since \(Z_{k,j}^2 = \sum_{i \neq \ell} V_i^{k,j} V_\ell^{k,j} E_i E_\ell + \sum_{i=1}^N V_i^{k,j \, 2} E_i^2 \), \(\mathbb{E}(V_i^{k,j}) = 0\) and by independence along the subscripts, we obtain 
$$\mathbb E(Z_{k,j}^2) = \sum_{i=1}^N \mathbb{E}((V_i^{k,j})^2) \mathbb{E}(E_i^2) = N \times \frac{2}{d} \times 2 = \frac{4N}{d}.$$
Moreover, according to Lemma \ref{lemme moment} applied to \( M_i := V_i^{k,j} E_i \), we obtain 
\begin{eqnarray*}
    \mathbb{E}(Z_{k,j}^4) &=& N \mathbb{E}(M^4) + 3N (N-1) \mathbb{E}(M^2)^2 \\ &=& N \mathbb{E}((V_i^{k,j})^4) \mathbb{E}(E_i^4) + 3N (N-1) \mathbb{E}((V_i^{k,j})^2)^2 \mathbb{E}(E_i^2)^2.
\end{eqnarray*}
Indeed, it is easily checked that the variables \( (M_i)_{i=1}^N \) are centered and independent, due to the independence between the elements of the collections \( (V_i^{k,j})_{i=1}^N \) and \( (E_i)_{i=1}^N \) and the fact that the $V_i^{k,j}$ are centered. Basic calculations then give
\begin{eqnarray*}
    \mathbb{E}(Z_{k,j}^4) &=& N \mathbb{E}({V_i^{k,j \, 4}}) \mathbb{E}(E_i^4) + 3N (N-1) \mathbb{E}({V_i^{k,j \, 2}})^2 \mathbb{E}(E_i^2)^2 \\ &=& N \times \dfrac{2}{d} \times  4! + 3N(N-1) \left( \dfrac{2}{d}\right)^2 \times 2^2 \\ &=& \dfrac{48N}{d^2} (d+N-1).
\end{eqnarray*}
Consequently, we get
\[
\frac{\mathbb{E}(Z_{k,j}^2)^2}{\mathbb{E}(Z_{k,j}^4)} = \frac{16N^2}{d^2} \times \frac{d^2}{48N (d + N -1)} = \frac{N}{3(d + N -1)}
\]
and thus, for all \(\theta \in (0,1)\),
\[
\PP\left(|Z_{k,j}| \geq \sqrt{\theta} \sqrt{\frac{4N}{d}}\right) \geq (1-\theta)^2  \frac{N}{3(d + N -1)}.
\]
In particular, for \(N \geq d\), we have $3(d + N - 1) \leq 6N$, which gives
\[
\PP\left(|Z_{k,j}| \geq \sqrt{\theta} \sqrt{\frac{4N}{d}}\right) \geq (1-\theta)^2 / 6.
\]
With the choice \(\theta = 1/4\), it holds
\[
\PP\left(|Z_{k,j}| \geq \sqrt{\frac{N}{d}}\right) \geq \frac{9}{16} \times \frac{1}{6} = \frac{3}{32} \geq \frac{1}{11}.
\]
Finally, by the following lower bound,
\[
\frac{h_+(\mathcal V (x) )}{h_-(\mathcal V (x) )} = \exp \left( \max_{1\leq k < j \leq d } |Z_{k,j}| \right) \geq \exp (|Z_{1,2}|),
\]
we get, for any $N \geq d$, 
\begin{align*}
\PP\left(\frac{h_+(\mathcal V (x) )}{h_-(\mathcal V (x) )} 
\geq \exp\left(\sqrt{\frac{N}{d}}\right)\right) &\geq \PP\left(\exp(|Z_{1,2}|) \geq \exp\left(\sqrt{\frac{N}{d}}\right)\right) \\ 
&= \PP\left(|Z_{1,2}| \geq \sqrt{\frac{N}{d}}\right) \geq \frac{1}{11}.    
\end{align*} \qed

\subsection*{Proof of Proposition \ref{prop:diam_CRT}}

As in the proof of Proposition \ref{prop:diamURT2}, notice that \[ \pr(  \diam (\mathcal{V}(x))\geq t) \leq d\pr \left(h_1 \geq \frac{t}{\sqrt{d}}\right). \]
Then, for any $r\in (0,1)$ and $\lambda>0$,
\begin{align*}
    \pr(h_1 \geq r^N) & = \pr\left(\prod_{i=1}^N 2^{-B_i^{(1)}} \geq r^N\right)\\
    & \leq \mathbb E\left[\left(\frac{\prod_{i=1}^N 2^{-B_i^{(1)}}}{r^N}\right)^\lambda \right]=\left(\frac{\mathbb E\left[{2^{-\lambda B_1^{(1)}}}\right]}{r^{\lambda}}\right)^N. 
\end{align*}
It holds
\[
 \mathbb E\left[{2^{-\lambda B_1^{(1)}}}\right]= \frac{1}{d2^{\lambda}}+1-\frac{1}{d}.
\]
Hence,
\[
 \pr(h_1 \geq r^N) \leq \left(\frac{1}{d2^{\lambda}}+1-\frac{1}{d}\right)^N r^{-\lambda N}.
\]
Let us set $r=2^{-\alpha}$ and define 
\[
h(\lambda)= Q(\lambda) 2^{\lambda \alpha}
\]
with
\[
Q(\lambda)= \frac{1}{d2^{\lambda}}+1-\frac{1}{d}.
\]
By differentiating in $\lambda$, we get
\[
h^{\prime}(\lambda)=\log(2) 2^{\lambda \alpha}\left(\alpha Q(\lambda)-\frac{1}{d2^\lambda}\right).
\]
Hence, $h^{\prime}(\lambda_0)=0$ for $\lambda_0$ such that $2^{-\lambda_0}=\theta=\alpha (d-1)/(1-\alpha)$ and $\alpha \in (0,1/d)$. With this choice of $\lambda$, 
\[
 \PP(\diam(\mathcal{V}(x))\geq \sqrt{d}2^{-\alpha N}) \leq d\left(1-\frac{1-\theta}{d}\right)^N \theta^{-\alpha N}.
\]
We proceed in the same way as before for the diameter upper bound. By a union bound and symmetry in the directions, we have
\[
 \pr(  \diam (\mathcal{V}(x))\leq t) \leq d\pr \left(h_1 \leq \frac{t}{\sqrt{d}}\right).
\]
Then, for any $r\in (0,1)$ and $\lambda>0$,
\begin{align*}
    \pr(h_1 \leq r^N) & = \pr\left(\prod_{i=1}^N 2^{B_i^{(1)}} \geq r^{-N}\right)\\
    & \leq \mathbb E\left[\left(\frac{\prod_{i=1}^N 2^{B_i^{(1)}}}{r^{-N}}\right)^\lambda \right]=\left(\frac{\mathbb E\left[{2^{\lambda B_1^{(1)}}}\right]}{r^{-\lambda}}\right)^N. 
\end{align*}
It holds
\[
 \mathbb E\left[{2^{\lambda B_1^{(1)}}}\right]= \frac{2^{\lambda}}{d}+1-\frac{1}{d}.
\]
Hence,
\[
 \pr(h_1 \leq r^N) \leq \left(\frac{2^{\lambda}}{d}+1-\frac{1}{d}\right)^N r^{\lambda N}.
\]
Let us set $r=2^{-\alpha}$ and denote 
\[
h(\lambda)= Q(\lambda) 2^{-\lambda \alpha}
\]
with
\[
Q(\lambda)= \frac{2^{\lambda}}{d}+1-\frac{1}{d}.
\]
By differentiating in $\lambda$, we get
\[
h^{\prime}(\lambda)=\log(2) 2^{-\lambda \alpha}\left(\frac{2^\lambda}{d}-\alpha Q(\lambda)\right).
\]
Hence, $h^{\prime}(\lambda_0)=0$ for $\lambda_0$ such that $2^{\lambda_0}=\theta=\alpha (d-1)/(1-\alpha)$ and $\alpha \in (1/d , 1)$. With this choice of $\lambda$, 
\[
 \PP(\diam(\mathcal{V}(x))\leq \sqrt{d}2^{-\alpha N}) \leq d\left(1-\frac{1-\theta}{d}\right)^N \theta^{-\alpha N}.
\]
\qed

\subsection*{Proof of Proposition \ref{cor22}}

According to Proposition \ref{prop:diam_CRT}, for any $\alpha \in (1/d,1)$, we have, for $\theta = \alpha (d-1)/(1-\alpha)$,
\begin{equation*}
    \PP(\diam(\mathcal{V}(x))\leq \sqrt{d}2^{-\alpha N}) \leq d\left(1-\frac{1-\theta}{d}\right)^N \theta^{-\alpha N}.
\end{equation*}
Take now $\alpha=\alpha_N=1/d+\omega_N$, with $\omega_N \rightarrow_{N\rightarrow +\infty} 0$. In this case,
\[
\theta=\theta_N=\frac{\alpha_N(d-1)}{1-\alpha_N}=(d-1)\frac{1+d\omega_N}{d-1-d\omega_N}=1+a_d \omega_N+ b_d \omega_N^2+O(\omega_N^3),
\]
where $a_d=d^2/(d-1)$ and $b_d=d^3/(d-1)^2$. This gives
\[
\log\left(1-\frac{1-\theta}{d}\right)= \frac{a_d}{d} \omega_N+ \frac{b_d}{d}\omega_N^2 - \frac{a^2_d}{2d^2}\omega_N^2 + O(\omega_N^3).
\]
Futhermore $$\log(\theta) = - a_d \omega_N + b_d \omega_N^2 -  \frac{a_d^2}{2} \omega_N^2 + O(\omega_N^3)$$
so $$\alpha \log(\theta) = \frac{a_d}{d} \omega_N + \frac{b_d}{d} \omega_N^2 -  \frac{a_d^2}{2d} \omega_N^2 + a_d \omega_N^2 + O(\omega_N^3).$$
Then 
\begin{eqnarray*}
    \log\left(1-\frac{1-\theta}{d}\right) - \alpha \log(\theta) &=& -\frac{a^2_d}{2d^2}\omega_N^2 + \frac{a_d^2}{2d} \omega_N^2 - a_d \omega_N^2 + O(\omega_N^3) \\ &=& - a_d \omega_N^2 \left( 1 + \frac{a_d}{2d^2} - \frac{a_d}{2d} \right) + O(\omega_N^3).
\end{eqnarray*}
Moreover 
$$1 + \frac{a_d}{2d^2} - \frac{a_d}{2d}  = 1 + \frac{1}{2(d-1)} - \frac{d}{2(d-1)} = 1 - \frac{1}{2} = \frac{1}{2}.$$
Finally
\begin{eqnarray*}
  \left(1-\frac{1-\theta}{d}\right)^N \theta^{-\alpha N} &=& \exp \left( N \log\left(1-\frac{1-\theta}{d}\right) - N \alpha \log(\theta)    \right) \\ &=& \exp \left( - \frac{a_d}{2} N \omega_N^2 + O(N \omega_N^3)  \right).  
\end{eqnarray*}
Choosing $\omega_N=2 \sqrt{\log(N)/(a_d N)} \in (0 , 1 - 1/d)$ for $N$ large enough, gives
$$\left(1-\frac{1-\theta}{d}\right)^N \theta^{-\alpha N} = \exp \left( - 2 \log(N) + O \left(\log(N)^{3/2} / \sqrt{N}   \right) \right) \underset{ {N \to + \infty}}\sim N^{-2}$$
and concludes the proof via the Borel-Cantelli lemma. Furthermore, for the upper bound of the diameter, we also use Proposition \ref{prop:diam_CRT}. For any $\alpha \in (0,1/d)$, we have for $\theta = \alpha (d-1)/(1-\alpha)$,
\[
 \PP(\diam(\mathcal{V}(x))\geq \sqrt{d}2^{-\alpha N}) \leq d\left(1-\frac{1-\theta}{d}\right)^N \theta^{-\alpha N}.
\]
Let us take here $\alpha=\alpha_N=1/d-\omega_N$, with $\omega_N \rightarrow_{N\rightarrow +\infty} 0$. In ths case,
\[
\theta=\theta_N=\frac{\alpha_N(d-1)}{1-\alpha_N}=(d-1)\frac{1-d\omega_N}{d-1+d\omega_N}=1-a_d \omega_N+ b_d \omega_N^2+O(\omega_N^3),
\]
where $a_d=d^2/(d-1)$ and $b_d=d^3/(d-1)^2$. This gives
\[
\log\left(1-\frac{1-\theta}{d}\right)= - \frac{a_d}{d} \omega_N+ \frac{b_d}{d}\omega_N^2 - \frac{a^2_d}{2d^2}\omega_N^2 + O(\omega_N^3).
\]
In addition, $$\log(\theta) = - a_d \omega_N + b_d \omega_N^2 -  \frac{a_d^2}{2} \omega_N^2 + O(\omega_N^3)\;,$$
so $$\alpha \log(\theta) = - \frac{a_d}{d} \omega_N + \frac{b_d}{d} \omega_N^2 -  \frac{a_d^2}{2d} \omega_N^2 + a_d \omega_N^2 + O(\omega_N^3)\;.$$
Now,
\begin{eqnarray*}
    \log\left(1-\frac{1-\theta}{d}\right) - \alpha \log(\theta) &=& -\frac{a^2_d}{2d^2}\omega_N^2 + \frac{a_d^2}{2d} \omega_N^2 - a_d \omega_N^2 + O(\omega_N^3) \\ &=& - a_d \omega_N^2 \left( 1 + \frac{a_d}{2d^2} - \frac{a_d}{2d} \right) + O(\omega_N^3).
\end{eqnarray*}
Moreover,
$$1 + \frac{a_d}{2d^2} - \frac{a_d}{2d}  = 1 + \frac{1}{2(d-1)} - \frac{d}{2(d-1)} = 1 - \frac{1}{2} = \frac{1}{2}.$$
Finally,
\begin{eqnarray*}
  \left(1-\frac{1-\theta}{d}\right)^N \theta^{-\alpha N} &=& \exp \left( N \log\left(1-\frac{1-\theta}{d}\right) - N \alpha \log(\theta)    \right) \\ &=& \exp \left( - \frac{a_d}{2} N \omega_N^2 + O(N \omega_N^3)  \right).  
\end{eqnarray*}
Choosing $\omega_N=2 \sqrt{\log(N)/(a_d N)}$ gives
$$\left(1-\frac{1-\theta}{d}\right)^N \theta^{-\alpha N} = \exp \left( - 2 \log(N) + O \left(\log(N)^{3/2} / \sqrt{N}   \right) \right) \underset{ {N \to + \infty}}\sim N^{-2}$$
and concludes the proof via the Borel-Cantelli lemma.

The last inequality follows directly by invoking the two previous inequalities on diameter and volume.
\qed

\subsection*{Proof of Theorem \ref{cor_UcR}}

The proof follows the same steps as the one of Theorem \ref{cor_URT}. First, since the local regression map results from a tree construction, each element in $\mathcal V ( x, (D_i,S_i)_{i=1}^N )$ is a hyper-rectangle. Hence, in light of \cite{wenocur1981some}, it holds that the resulting local map has dimension $v = 2d$.
Second, observe that assumption \ref{cond:density_XCART} implies assumption \ref{cond:density_X} with $\ell(x) = b$, so that Theorem \ref{th2:general}, applied pointwise for $x \in S_X$, yields that whenever $n P^X(\mathcal{V}(x)) / \log(n) \to \infty$, it holds that
$ \sum_{n\geq 1} \mathbb P ( |\hat g_{\mathcal V}(x) - g(x)| > v_n ) < \infty$
where \[v_n = \sqrt{ 3 \sigma^2 \log( {(n+1)^{v+2}}  ) / (n  b \lambda ( \mathcal V(x)) )}  + L(\mathcal V(x)) \diam (\mathcal V(x)).\] Note that we have set $\delta = (n+1)^{-2}$.  Making use of Borel Cantelli Lemma, it implies that with probability $1$, for $n$ large enough,
$$|\hat g_{\mathcal V}(x) - g(x)| \leq  \sqrt{\frac{ 3 \sigma^2 \log\left( {(n+1)^{v+2}} \right)}{n  b \lambda ( \mathcal V(x)) }} + L(\mathcal V(x)) \diam (\mathcal V(x)).$$

Third, using \ref{cond:density_XCART}, it follows that
$$ P^X (\mathcal{V}(x) ) \geq b \lambda(\mathcal{V}(x)) = b 2^{-N} = n^{-d/(d+2)} b $$
then $$ n P^X (\mathcal{V}(x) ) \geq n^{2/(d+2)} b.$$
Hence, we get that $ n P^X (\mathcal{V}(x) ) / \log(n) \to \infty$.

Fourth, by putting together the second and third point from above, we have the following inequality, with probability $1$, for $n $ large enough,
$$|\hat g_{\mathcal V}(x) - g(x)| \leq  \sqrt{\frac{ 3 \sigma^2 \log\left( (n+1)^{v+2} \right)}{n  b \lambda ( \mathcal V(x)) }} + L(\mathcal V(x)) \diam (\mathcal V(x)).$$
Now, from Proposition \ref{cor22}, for a sufficiently large, we have
\begin{align*}
& \diam (\mathcal{V}(x)) \leq \sqrt d 2^{-N/d+2\sqrt{(d-1)N\log(N)/d^2}}, \\
& \lambda(\mathcal{V}(x)) = 2^{-N}.
\end{align*}
Hence, we get, with probability $1$, for $n$ large enough,
$$| \hat g_{\mathcal V}(x) - g(x)| \leq \sqrt{\frac{ 3 \sigma^2 \log\left( {(n+1)^{v+2}} \right)}{n b 2^{-N}}} + L(\mathcal V(x)) \sqrt d 2^{-N/d+2\sqrt{(d-1)N\log(N)/d^2}}.$$
Because \( N = d\log(n)/(\log(2)(d+2)) \), we obtain
\begin{align*}
 &   | \hat g_{\mathcal V}(x) - g(x)| \\&\leq n^{-1/(d+2)} \sqrt{\frac{ 3 \sigma^2 \log\left( {(n+1)^{v+2}} \right)}{b}} + n^{-1/(d+2)} L(\mathcal V(x)) \sqrt d e^{2\sqrt{(d-1)N\log(N)/d^2}}\\
& \leq n^{-1/(d+2)} \sqrt{\frac{ 6 \sigma^2 (v+2)\log\left( n \right)}{b}} + n^{-1/(d+2)} L(\mathcal V(x)) \sqrt d e^{2\sqrt{(d-1)N\log(N)/d^2}}
\\
& \leq n^{-1/(d+2)} \sqrt{\frac{ 12 \sigma^2 (d+1)\log\left( n \right)}{b}} + n^{-1/(d+2)} L(\mathcal V(x)) \sqrt d e^{\sqrt{N\log(N)}}
\end{align*}
where we use $v = 2d$ and the inequality $2 \sqrt{(d-1)/d^2} \leq 1$ since $(d-2)^2 \geq 0$.
For $n$ large enough, we have \( N = d\log(n)/(\log(2)(d+2)) \geq 8.\) Thus, this implies that $\log(n) = N \log(2) (d+2)/ d \leq 3 \log(2) N = \log(8) N \leq \log(N) N$. 
Moreover $N \leq 2 \log(n)$, and for $n$ large enough $N \leq \log(n)^2$ then $\log(N) \leq 2 \log \log(n).$ Finally $\log(n) \leq \log(N) N \leq  4 \log(n) \log(\log(n)).$ We conclude by using the inequality $\sqrt{x} \leq e^{\sqrt{x}}$ for $x= \log(n)$ and setting $C = \sqrt{{12 \sigma^2 (d+1)\log\left( n \right)}/b} +  L(\mathcal V(x)) \sqrt d$.
\qed

\subsection*{Proof of Proposition \ref{centered tree not regular}}

We follow the proof of Proposition \ref{uniform tree not regular}, with similar notation, but this time the variable $E_i := -\log(U_i)$ is replaced by $E_i := \log(2)$. By performing the calculations again, we find the moments with lemma \ref{lemme moment},
\begin{eqnarray*}
    \mathbb{E}(Z_{k,j}^2) &=& 2N \log(2)^2/d. \\
    \mathbb{E}(Z_{k,j}^4) &=& N \mathbb{E}({V_i^{k,j \, 4}}) \mathbb{E}(E_i^4) + 3N (N-1) \mathbb{E}({V_i^{k,j \, 2}})^2 \mathbb{E}(E_i^2)^2 \\ &=& N \times \dfrac{2}{d} \times  \log(2)^4 + 3N(N-1) \left( \dfrac{2}{d}\right)^2 \times \log(2)^4 \\ &=& \dfrac{2N}{d^2} \log(2)^4 (6N - 6 + d).
\end{eqnarray*}
The Paley-Zygmund bound becomes
\[
\frac{\mathbb{E}(Z_{k,j}^2)^2}{\mathbb{E}(Z_{k,j}^4)} = \frac{4 \log(2)^4 N^2}{d^2} \times \frac{d^2}{2N \log(2)^4 (6N - 6 + d) }= \frac{2N}{6N - 6 + d}.
\]
Thus, for all \(\theta \in (0,1)\) and $N \geq d$,
\[
\PP\left(|Z_{k,j}| \geq \sqrt{\theta} \sqrt{\frac{2N \log(2)^2}{d}}\right) \geq (1-\theta)^2 \frac{2N}{6N - 6 + d} \geq \dfrac{2(1-\theta)^2}{7}.
\]
Let us choose \(\theta = 1/2\) to obtain
\[
\PP\left(|Z_{k,j}| \geq \log(2) \sqrt{\frac{N}{d}}\right) \geq \dfrac{1}{14}
\]
and thus for $N \geq d$,
\[
\PP\left(\dfrac{h_+(\mathcal V (x))}{h_-(\mathcal V (x))} \geq 2^{\sqrt{{N}/{d}}} \right) \geq \frac{1}{14}.
\]
Thus, with probability at least $1/14$, the ratio ${h_+(\mathcal V (x))}/{h_-(\mathcal V (x))} $ is bounded below by a quantity that grows exponentially towards infinity. This means that uniform trees are not regular.
\qed

\subsection*{Proof of Proposition \ref{mondrian}}

According to \cite[Proposition 1]{minimaxmondrian}, we know the distribution of the largest and the smallest side. In fact, we have \( h_-(\mathcal V (x)) \sim \min(X_1, \dots, X_d) \) and \( h_+(\mathcal V (x)) \sim \max(X_1, \dots, X_d) \), where the \( X_i \) are i.i.d. and follow the Gamma distribution \( X \sim \Gamma(2, \Lambda) \). We have, for all $u \geq 0$,
\[ \mathbb{P}(X \geq u) = \int_u^{+\infty} \Lambda^2 t e^{-\Lambda t} \, dt = e^{-\Lambda u} (1 + u \Lambda) \geq e^{-\Lambda u}. \]
Moreover,
\[ \mathbb{P}(h_-(\mathcal V (x)) \geq u) = \mathbb{P}(X \geq u)^d \geq e^{-\Lambda u d} = 1 - \delta\; , \]
for \( u = -{\log(1-\delta)}/{(\Lambda d)}. \)
Then, with probability at least $1 - \delta$,
\begin{equation}\label{h-}
h_-(\mathcal V (x)) \geq - \dfrac{\log(1-\delta)}{\Lambda d}.    
\end{equation}
We focus now on \( h_+(\mathcal V (x)) \). We have, for all $t \geq 0$,
\[ \mathbb{P}(h_+(\mathcal V (x)) \leq t) = \mathbb{P}(X \leq t)^d. \]
Let \( Y := X - \EE(X) \). Since \( X \) follows a Gamma distribution, \( Y \) is Sub-Gamma. According to \cite[p. 29]{boucheron2013concentration},
\[ \forall t > 0, \quad \mathbb{P}(\Lambda Y \geq 2\sqrt{t} + t) \leq e^{-t}. \]
Thus,
\begin{eqnarray*}
    \mathbb{P}(\Lambda h_+ \leq 2 \sqrt{t} + t + \Lambda \EE(X)) &=& \mathbb{P}(\Lambda Y \leq 2 \sqrt{t} + t)^d = \left(1 - \mathbb{P}(\Lambda Y > 2 \sqrt{t} + t)\right)^d \\ &\geq& (1 - e^{-t})^d = 1 - \delta\; ,
\end{eqnarray*}
with \( t = - \log(1-(1-\delta)^{1/d}). \)
Therefore, with probability at least \(1 - \delta\),
$$h_+(\mathcal V (x)) \leq \frac{2 + 2\sqrt{- \log(1-(1-\delta)^{1/d})} - \log(1-(1-\delta)^{1/d})}{\Lambda}.$$
In particular, for \( \delta \leq 1 - (1 - e^{-1})^d \),
\begin{equation}\label{h+} 
h_+(\mathcal V (x)) \leq \frac{-5 \log(1-(1-\delta)^{1/d})}{\Lambda} \leq \dfrac{-5\log(\delta/d)}{\Lambda} 
\end{equation}
where the last inequality comes from the inequality \( \delta/d \leq 1 - (1 - \delta)^{1/d} \).
Hence, with probability at least \(1 - 2\delta\) for $\delta \leq 1 - (1 - e^{-1})^d$
\[ \frac{h_+(\mathcal V (x))}{h_-(\mathcal V (x))} \leq \frac{5d \log(\delta/d)}{{\log(1-\delta)}}. \]
\qed

\subsection*{Proof of Theorem \ref{thm:mondrian}}
Let \(x \in S_X\). We write the bias-variance decomposition $\hat g_{\mathcal V}(x) - g(x) = V + B$, where \begin{align*}
    V := \dfrac{\sum_{i=1}^n \varepsilon_i \ind_{ \mathcal V(x) }(X_i)}{\sum_{j=1}^n \ind_{ \mathcal V(x) }(X_j)} \qquad \text{and}\qquad
       B := \dfrac{\sum_{i=1}^n \left(g(X_i) - g(x)\right) \ind_{ \mathcal V(x) }(X_i)}{\sum_{j=1}^n \ind_{ \mathcal V(x) }(X_j)}.
\end{align*}
Let us recall Inequality \eqref{h-} obtained in the previous proof, with probability at least \(1 - \delta\),  $$h_-(\mathcal V (x)) \geq - \frac{\log(1-\delta)}{\Lambda d}.$$ We thus have, whenever $n b\geq 8  c_{\delta,d} \Lambda^d$, that the inequality 
$$n P^X(\mathcal{V}(x)) \geq n b h_-^d \geq \frac{nb \log(1/(1-\delta))^d} { (\Lambda d)^d} = \frac{nb \log(1/\delta) }{\Lambda^d c_{\delta,d} } \geq 8 \log(1/\delta)$$
holds with probability at least $1-\delta$. Let $E_1$ be the event from previous equation. Let $E_2 $ be the event such that
$$\left| \dfrac{\sum_{i=1}^n \varepsilon_i \ind _{ \mathcal V(x) }(X_i)}{{\sum_{j=1}^n \ind _{ \mathcal V(x) }(X_j)}}  \right| \leq \sqrt{\dfrac{4 \sigma^2 \log(1/\delta)}{n P^X(\mathcal{V}(x))} } .$$
On $E_1\cap E_2 $, it holds
\[
\left| \dfrac{\sum_{i=1}^n \varepsilon_i \ind _{ \mathcal V(x) }(X_i)}{{\sum_{j=1}^n \ind _{ \mathcal V(x) }(X_j)}}  \right| \leq \sqrt{\dfrac{4\sigma^2 \log(1/\delta)}{n P^X(\mathcal{V}(x))}} \leq \sqrt{\dfrac{4\sigma^2 c_{\delta,d} \Lambda^d  }{n b}}\;.
\]
It remains to check that $\mathbb P ( E_1\cap E_2) \geq 1-4\delta$. Note that 
$A\cup B = A \cup  (A^c\cap B)  $
which, when applied to $ A = E_1^c$ and $B = E_2^c $, gives $\mathbb P ( E_1^c\cup E_2^c) = \mathbb P (  E_1^c) + \mathbb P ( E_1\cap E_2^c)   $. The first term $ \mathbb P (  E_1^c)$ is smaller than $\delta$, as shown before.
 According to Lemma \ref{sousgauss}, we have $\mathbb P ( E_1\cap E_2^c |\mathcal{V}(x)) $ is smaller than $3\delta$. Integrating with respect to $\mathcal{V}(x)$, we obtain $\mathbb P ( E_1\cap E_2^c )\leq 3\delta $. 
As for the bias term, for $\delta \leq 1 - (1 - e^{-1})^d$, it was shown in \eqref{h+} (see previous proof) that, with probability at least $1-\delta$,  \[ h_+(\mathcal V (x)) \leq \dfrac{-5\log(\delta/d)}{\Lambda}\;.\]
Observing that $ 1 - (1 - e^{-1})^d \geq 1/5$, it follows that, with probability at least $1-\delta$, 
$$ |B|\leq L (\mathcal V(x)) \diam(\mathcal{ V } (x)) \leq L (\mathcal V(x))  \sqrt{d} \, h_+(\mathcal{ V } (x)) \leq L (\mathcal V(x))   \sqrt{d}\, 5 \, \frac{\log(d/\delta)}{\Lambda}\,.$$ 
Thus, putting together the obtained bounds on $|V|$ and $|B|$, we find, with probability at least $1-5\delta$,
  \[
 | \hat g_{\mathcal V}(x) - g(x)| \leq \sqrt{\dfrac{4\sigma^2 c_{\delta,d} \Lambda^d }{n b }} +  5 \sqrt{d} \, L(\mathcal{ V } (x)) \frac{\log(d/\delta)}{\Lambda}\;.
\]
In addition, choosing $\Lambda \asymp n^{1/d+2}$ yields
  \[
 | \hat g_{\mathcal V}(x) - g(x)| \lesssim \dfrac{C}{n^{1/d+2}}\;,
\]
with $C = \sqrt{{4\sigma^2 c_{\delta, d}   }/ {b }} +  5 \sqrt{d} \, L(\mathcal{ V } (x))\log(d/\delta).$ \qed

\section{Auxiliary results and technical lemmas}

Let us state the following Vapnik-type inequality \cite{vapnik2015uniform}, which involves some standard-error normalization. The first inequality in the next theorem is Theorem 2.1 in  \cite{anthony1993result} (see also Theorem 1.11 in  \cite{lugosi2002pattern}). The second inequality can be obtained from the first one.

\begin{theorem}[normalized Vapnik inequality] \label{th:vapnik_normalized}
Let $(Z, Z_1, \ldots , Z_n) $ is a collection of random variables independent and identically distributed with common distribution $P^Z$ on $(S,\mathcal S)$. For any $A \in \mathcal{S},$ let denote $nP_n^Z(A) = \sum_{i=1}^n \ind_{A}(Z_i)$. For any class $\mathcal A\subset \mathcal S $, $\delta > 0$ and $n\geq 1$, it holds with probability at least $1 -\delta $, for all $A\in \mathcal A$,
$$ P_n^Z(A) \geq P^Z(A) \left(1-\sqrt { \frac{4 \log( 4\mathbb S_\mathcal A(2n) / \delta) }{ n P^Z(A)} } \right).$$ 
In particular, with probability at least $1-\delta$ we have, for all $A\in \mathcal A$,
$$   P^Z(A) \leq \dfrac{4}{n} \log\left(\dfrac{4\, \mathbb S_\mathcal A(2n)}{\delta} \right) + 2 P_n^Z(A).$$ 
\end{theorem}

\begin{proof}
    The first statement is proved in \cite{anthony1993result}.   Let us  prove the second statement. According to the first point, with probability at least $1 - \delta$, we have for all $A \in \mathcal{A}$ $$n P_n^Z(A) - n P^Z(A) \geq     - \sqrt{  4nP^Z(A)    \log(4\mathbb S_{\mathcal A}(2n) /\delta )   },$$   equivalently,  
    $$nP^Z(A) - \sqrt{  4nP^Z(A)    \log(4\mathbb S_{\mathcal A}(2n) /\delta )   } -  n P_n^Z(A)\leq 0.$$
    Setting $x = \sqrt{nP^Z(A)}$, $\alpha = \sqrt{  4 \log(4\mathbb S_{\mathcal A}(2n) /\delta )}$ and $\beta = nP_n^Z(A)$, we have that $x^2 - \alpha x - \beta \leq 0.$ Solving the inequality, we find  $$( \alpha - \sqrt{\alpha^2 + 4\beta}) / 2 \leq x \leq ( \alpha + \sqrt{\alpha^2 + 4\beta} ) / 2.$$ 
    By using the fact that \(x\) is positive and squaring both sides, it follows that $x^2 \leq ( \alpha + \sqrt{\alpha^2 + 4\beta})^2 / 4$. And by using the inequality \((a+b)^2 \leq 2(a^2 + b^2)\), we obtain $nP^Z(A) = x^2 \leq \alpha^2 + 2\beta = 4 \log(4\mathbb S_{\mathcal A}(2n) /\delta ) + 2nP_n^Z(A)$
    which is the desired result by dividing each side of the inequality by \(n\).   
\end{proof}

For more details, one can also refer to the book by \cite{boucheron2013concentration}, especially chapters 12 and 13, as well as \cite{devroye96probabilistic}. 



The following result is standard and known as the multiplicative Chernoff bound for empirical processes. The following version can be found in \cite{hagerup1990guided}.

\begin{theorem}\label{lemma=chernoff}
Let $(Z, Z_1, \ldots , Z_n) $ is a collection of random variables independent and identically distributed with common distribution $P^Z$ on $(S,\mathcal S)$. Let $A$ be a set in $\mathbb R^d$ and let denote $nP_n^Z(A) = \sum_{i=1}^n \ind_{A}(Z_i)$. For any $\delta \in (0,1)$ and all $n\geq 1$, we have with probability at least $1-\delta$
\begin{align*}
P_n^Z(A) \geq \left(1- \sqrt{ \frac{2 \log(1/\delta)  }{ nP^Z(A) } } \right) P^Z(A)  .
\end{align*}
 In addition, for any $\delta \in (0,1)$ and $n\geq 1$, we have with probability at least $1-\delta$
\begin{align*}
P_n^Z(A)  \leq \left(1 +  \sqrt{ \frac{3 \log(1/\delta)   }{ n P^Z(A)} }  \right) P^Z(A).
\end{align*}
\end{theorem}

The following lemmas will be useful to prove Proposition \ref{prop_lemme}.

\begin{lemma}\label{kappa}
     Let $W,W_1,\ldots, W_n$ be independent and identically distributed random variables on $\mathbb R_{+}$ with density $f_W$ and cumulative distribution function $F_W$. Suppose that there exists $c_0>0$ and $\kappa \geq 0$, such that  $ f_W (t) \leq c_0 F_W(t) ^{\kappa} $ for all $t\in  \mathbb R_{+} $. 
     Let $\delta \in (0,1) $. It holds, with probability at least $1-\delta $, 
   $$ W_{(2)} -W_{(1)} >    C^{-1} \delta  m^{ \kappa-1 }  ,$$ 
   where $C=  c_0 \Gamma(\kappa+1) 3^{\kappa+1}.$
\end{lemma}

\begin{proof}
    Note that 
\begin{align*}
 \PP (W_{(2)} -W_{(1)}   > t ) &= \sum_{k=1} ^ m \mathbb P (  W_j > W_k + t \, : \, \forall  j \neq k) \\
 &= \sum_{k=1} ^ m \mathbb E [  (1- F_W ( W_k + t ))^{m-1} ]\\
 &= m \mathbb E [   (1- F_W ( W + t ))^{m-1} ].
  \end{align*}
Let us define, 
for any $t\in \mathbb{R}_+$,
\begin{equation*}
    D(t):= \frac{1}{m}\mathbb{P}(W_{(2)}-W_{(1)}>t) =  \mathbb E [   (1- F_W ( W + t ))^{m-1}  ].
\end{equation*}
It holds $D(0)=1/m$, $D(+\infty)=0$ and by its definition through the deviation probability, $D$ is a non-increasing function on $\mathbb{R}_+$. 
By using Fubini-Tonelli and integrating first with respect to $t$, we find
\begin{align*}
    &(m-1)\int_0^t \mathbb{E}[f_W(W+u)(1- F_W ( W + u ))^{m-2}] \, du \\ 
    &= \mathbb E\left[ (m-1)\int_0^t f_W(W+u)(1- F_W ( W + u ))^{m-2} \, du \right] \\ &= \mathbb E [   (1- F_W ( W ))^{m-1}  ] - \mathbb E [   (1- F_W ( W + t ))^{m-1}  ] \\ &= D(0)-D(t).
\end{align*}
We also have
\begin{align*}
      & \mathbb{E}[f_W(W+u)(1- F_W ( W + u ))^{m-2}] \\ &=   \int_0^{+\infty} f_W(r+u)f_W(r)(1- F_W ( r + u ))^{m-2} \, \ind_{f_W(r+u) > 0} \, dr \\
    & \leq  c_0 \int_0^{+\infty} f_W(r+u)  F_W(r)^\kappa (1- F_W ( r + u ))^{m-2} \, \ind_{f_W(r+u) > 0} \, dr.
\end{align*}
Let us now apply a change of variable, which is justified because it is differentiable and bijective when $f_W$ is positive. Note also that $ F _W ( F_W^{-1} (u)) = u$ because $F_W$ is continuous. By setting $v=F_W(r+u)$, we get $dv=f_W(r+u)dr$, which gives, for any $\kappa \geq 0$,
\begin{align*}
    &\int_0^{+\infty} f_W(r+u)  F_W(r)^\kappa (1- F_W ( r + u ))^{m-2}\, \ind_{f_W(r+t) > 0} \, dr \\ &= \int_{F_W(u)}^{1}  F_W(F_W^{-1}(v)-u)^\kappa (1- v)^{m-2} \, \ind_{f_W(F^{-1}_W(v)) > 0} \, dv \\ &\leq  \int_{0}^{1} v^\kappa (1- v)^{m-2}dv 
\end{align*}
because $\ind_{f_W(F^{-1}_W(v)) > 0} \leq 1$ and $ F_W(F_W^{-1}(v)-u) \leq F_W(F_W^{-1}(v)) = v.$
Moreover,
\begin{align*}
\int_{0}^{1} v^\kappa (1- v)^{m-2}dv &\leq \int_{0}^{1} v^\kappa \exp(-v(m-2)) dv = \dfrac{1}{(m-2)^{\kappa+1}}\int_{0}^{m-2} s^\kappa e^{-s} ds \\ &\leq \dfrac{\Gamma(\kappa+1)}{(m-2)^{\kappa+1}}.
\end{align*}
Putting things together, we have shown that
$$D(0) - D(t) \leq c_0 \int_{0}^t (m-1)  \dfrac{\Gamma(\kappa+1)}{(m-2)^{\kappa+1}} du \leq m c_0 \Gamma(\kappa+1) \dfrac{3^{\kappa+1}}{m^{\kappa+1}} t = C \dfrac{t}{m^\kappa}.$$
In the latter upper-bound, we used $m-2 \geq m/3$, and $C = c_0 \Gamma(\kappa+1) 3^{\kappa+1}$. As a consequence, $$\PP\left(W_{(2)} -W_{(1)}   > t \right) = m D(t) \geq m D(0) - m \dfrac{C}{m^\kappa} t =   1 - \dfrac{C}{m^{\kappa-1}} t.$$
Choosing $t = { \delta  \, m^{\kappa-1}} / {C}$, leads to the statement.
\end{proof}



The purpose of this lemma is to establish conditions ensuring the assumption of the previous lemma. 
\begin{lemma}\label{condition F vers kappa}  
     Let $W$ be a random variable on $\mathbb R_{+}$ with density $f_W$ and cumulative distribution function $F_W$. Suppose that there exists \( T_0 > 0 \) such that for all \( t \in (0,T_0) \), we have
\[
F_W(t) \geq c_1 t^d \quad \text{and} \quad f_W(t) \leq c_2 t^{d-1}.
\]  
Additionally, suppose that there exists $U>0$ such that, for all \( t \geq 0 \), we have \( f_W(t) \leq U \). Then,   for all \( t \geq 0 \), the following inequality holds
\[
f_W(t) \leq c F_W(t)^\kappa,
\]  
where $\kappa = 1 - 1/d$ and \( c = {c_1^{-\kappa}} \max\left\{c_2, {U} / {T_0^{d-1}}\right\} \).  
\end{lemma}  

\begin{proof}  
For all \( t \in (0, T_0) \), we have  
\[
f_W(t) \leq c_2 t^{d-1} \leq c_2 \left(\frac{F_W(t)}{c_1}\right)^{1-1/d} = \dfrac{c_2}{c_1^\kappa} F_W(t)^\kappa,
\]  
where \( \kappa = 1 - 1/d \).  For \( t \geq T_0 \), we have
\[
f_W(t) \leq \frac{U}{(c_1T_0^d)^\kappa} (c_1T_0^d)^\kappa \leq  \frac{U}{(c_1T_0^d)^\kappa} F_W(T_0)^\kappa \leq \frac{U}{(c_1T_0^d)^\kappa} F_W(t)^\kappa.
\]  
Setting  
\[
c = \max\left(\frac{c_2}{c_1^{\kappa}}, \frac{U}{(c_1T_0^d)^{\kappa}}\right) = \frac{1}{c_1^{\kappa}} \max\left(c_2, \frac{U}{T_0^{d-1}}\right),
\]  
we obtain, for all \( t \geq 0 \), 
$ f_W(t) \leq c F_W(t)^\kappa$, as desired. 
\end{proof}  

In the following lemma, we provide assumptions on $Z$ to obtain results on $W$ that will be useful for applying Lemma \ref{condition F vers kappa}.

\begin{lemma}\label{Z vers X}
 Let $ Z $ be a random variable in $\mathbb R^d$ with density $f_Z$. Let $x\in \mathbb R^d$ and $W =\| Z - x\| $. The following holds:
\begin{enumerate}[label={(\alph*)}]
    \item  If $f_{Z} $ is bounded  by $M > 0$ then, for almost all $\rho>0$, $f_W(\rho) \leq Md V_d \rho^{d-1}$. 
    \item If $f_{Z} $ is bounded  by $M>0$ with compact support included in $B (0, \rho_0)$, then $f_W $ is bounded from above almost everywhere by $M d V_d \rho_0^{d-1}.$
    \item If $f_{Z}$ is bounded from below by $b > 0$ with support $S_{Z}$ and if there exists $c_d>0$ and $T_0>0$ such that $\lambda (S_{Z} \cap B(x, \tau ) ) \geq c_d \lambda   (B(x, \tau ))$ for all $\tau \in (0,T_0]$ and $x\in S_{Z}$, then for all $\rho \leq T_0, \, F_W(\rho) \geq c_d b V_d \rho^d$.
\end{enumerate}  
\end{lemma}  

\begin{proof}
Let us start by showing (a). Note that for any function $h: \mathbb R_{+} \to \mathbb R_{+}$, we have
\begin{equation*}
\mathbb   E [  h(\|{Z-x}\|)  ] = \int h(\|t-x\| )  f_{Z}(t)  dt \leq M  \int h(\|t-x\| ) dt = Md V_d \int h(\rho)\rho^{d-1} d\rho .     
\end{equation*}
Then almost everywhere $f_W(\rho) \leq Md V_d \rho^{d-1}$. Now we consider (b). 
From the first point, we have almost everywhere for $\rho > 0$, $f_W(\rho) \leq Md V_d \rho^{d-1}.$ Moreover, when $f_W$ has compact support included in $B(0, \rho_0)$, we then have almost everywhere $f_W(\rho) \leq \sup_{\rho \leq \rho_0} Md V_d \rho^{d-1} = Md V_d \rho_0^{d-1}.$ 
To prove (c), note that we have, for all $\rho \leq T_0$,
\begin{align*}
 F_W(\rho) &= \PP({Z} \in B(x, \rho)) = \int_{S_{Z}\cap B (x,\rho)} f_{Z}(u) du \\ &\geq b  \lambda (S_{Z}\cap B (x,\rho) ) \geq b c_d \lambda (B(x,\rho) ) = bc_d V_d \rho^d.    
\end{align*} 
\end{proof}


%

\begin{lemma}\label{PV}
 Let $(Z_i)_{i=1,\dots,m} \subset  \mathbb R^d$. Let $\hat Z_i(z)$ be the $i$-th nearest neighbor of $z\in \mathbb R^d$ (breaking ties in favor of larger index). Let $ x\in \mathbb R^d$ and define $\mathcal V(x) = \{z\in\mathbb R^d\,:\, \hat Z_1(z) = \hat Z_1(x)\}$ and $ W_{(i)} = \| \hat Z_i(x) - x\| $, for $i=1,\ldots, m$. Let $P$ be a probability measure such that $ P ( B (x,t) )  \geq c_3 t^d $ for all $t\in (0,T_1)$ and $x \in S_{Z}$ and for some $T_1>0$. Then, whenever $ (W_{(2)} -  W_{(1)}) \leq 2T_1$, we have
   $$ P(\mathcal{V}(x) ) \geq \frac{c_3}{2^d}   (W_{(2)} -  W_{(1)}) ^d.$$
\end{lemma}

\begin{proof}
We have
\begin{align}\label{first_ineq}
     P  (\mathcal{V}(x) ) 
    & = \sum _{k=1} ^ m \ind_{V_k }(x) \,  P [ \mathcal{V}(x) ]  = \sum _{k=1} ^ m \ind_{V_k }(x) \,  P [  V_k ]
\end{align}
Remark that 
$$   B ( Z_k , \Delta_k /2) \subset  V_k      $$
where $\Delta _k  = \min _{i\neq k} \|Z_i - Z_k\|$. 
Hence
$$  P [  V_k ] \geq  P [ B ( Z_k , \Delta_k /2) ].$$
We have (second triangular inequality)
$$ \|Z_i- Z_k\| \geq \left|  \|Z_i- x\| - \|x- Z_k\| \right|, $$
but when $x\in V_k$, $\|x-Z_k\| \leq  \min_{i\neq k} \|Z_i- x\| $. Therefore
$$ \|Z_i- Z_k\| \geq   \|Z_i- x\| - \|x- Z_k\| \geq 0. $$
Hence, whenever $x\in V_k$,
$$ \Delta _k = \min_{i\neq k} \|Z_i- Z_k\| \geq   \min_{i\neq k} \|Z_i- x\| - \|x- Z_k\|.$$
but since $ W_i =\| Z_i - x\| $, for $i=1,\ldots, n$, and $W_{(i)} $ are the increasing ordered statistics, we have
$$\Delta _k \geq W_{(2)} - W_{(1)} := \Delta.$$
It follows that
$$  P  (V_k) \geq   P [ B ( Z_k , \Delta /2) ].   $$
Suppose that $(W_{(2)} -  W_{(1)})\leq 2T_1$,  using the assumption $ P ( B(x,t) )  \geq c_3 t^d $ for $t = \Delta/2 \in (0, T_1)$, we find 
$$ P  (V_k) \geq \frac{c_3}{2^d}  \Delta^d .$$ 
From \eqref{first_ineq}, it finally follows that,
$$  P (\mathcal{V}(x) ) \geq  \frac{c_3}{2^d} \Delta ^d \sum _{k=1} ^ m \ind_{V_k } (x) = \frac{c_3}{2^d} \Delta ^d.$$
\end{proof}
\begin{lemma} \label{sousgauss}
Let $n\geq 1$, $\delta\in (0,1)$.  Let $(X,\varepsilon), (X_1,\varepsilon_1),\ldots, (X_n, \varepsilon_n)$ be an independent and identically distributed collection of random variables. Assume that the random variable $\varepsilon$ is sub-Gaussian conditionally on $X$ with parameter $\sigma^2$. Let $ V  $ be a measurable set such that $ nP^X(V )  \geq 8 \log(1/\delta) $. We have with probability at least $1-3\delta$,
    $$\left| \dfrac{\sum_{i=1}^n \varepsilon_i \ind _{V }(X_i)}{\sum_{j=1}^n \ind _{V}(X_j)} \right| \leq \sqrt{\dfrac{4\sigma^2 \log(1/\delta)}{n P^X(V)} }.$$  
\end{lemma}

\begin{proof}
Let us revisit the idea of the proof of Theorem \ref{1}, here, however we are not dealing with a uniform version. For all $i \in \{1, \dots, n\}$, let us denote $g_i = \ind_V(X_i)$ and \( \mathbb{P}_{X_{1:n}} \) the probability \( \mathbb{P} \) conditional on \( {X}_1, \dots, {X_n} \). Since the conditional distribution of $\varepsilon$ given $X$ is sub-Gaussian with parameter $\sigma^2$, then  $\varepsilon_i g_i$ is sub-Gaussian under $\mathbb{P}_{X_{1:n}}$ with parameter $\sigma^2 g_i^2$. Hence, ${\sum_{i=1}^n \varepsilon_i g_i}/\sqrt{{ \sum_{j=1}^n g_j}}$ is sub-Gaussian with parameter $\sigma^2 \sum_{i=1}^n g_i^2 / \sum_{j=1}^n g_j$ by independence. Moreover, $\sum_{i=1}^n g_i^2 = \sum_{i=1}^n g_i$ because $g_i\in \{ 0,1\} $. Hence, ${\sum_{i=1}^n \varepsilon_i g_i}/\sqrt{{ \sum_{j=1}^n g_j}}$ is sub-Gaussian with parameter $\sigma^2$ under $\mathbb{P}_{X_{1:n}}$. It follows that
\begin{align*}
    \mathbb{P}_{X_{1:n}}\left(\dfrac{\sum_{i=1}^n \varepsilon_i \ind_{ V }(X_i)}{\sqrt{\sum_{j=1}^n \ind_{ V }(X_j)}} > t  \right)  \leq \exp\left( \dfrac{-t^2}{2\sigma^2}\right) = \delta,
\end{align*}
   with  $t = \sqrt{2\sigma^2 \log(1/\delta)}$.
 Integrating with respect to $X_1,\dots, X_n$, we obtain the same inequality with $\mathbb P $ instead of $\mathbb{P}_{X_{1:n}}$. By symmetry, we obtain the result with absolute values with probability at least \( 1 - 2\delta \). 
 We have shown that with probability at least $1-2\delta$,
\begin{align}\label{eq1}
    \left| \dfrac{ \sum_{i=1}^n \varepsilon_i \ind_{V}(X_i) } {\sqrt{\sum_{j=1}^n \ind_{ V }(X_j)}}  \right| \leq  \sqrt{2\sigma^2 \log(1/\delta)}.
\end{align}
It now remains to show that, with probability at least $1-\delta$,
\begin{align}\label{eq2}
   \sum_{j=1}^n \ind_{ V }(X_j) = n 
P_n^X(V)  \geq  \dfrac{n P^X(V)}{2}.
\end{align}
Indeed, it can easily be seen that \eqref{eq1} and \eqref{eq2} imply the stated inequality and these inequalities hold together with probability at least $1-3\delta$.

Define $W_i = \ind_{ V } (X_i) $. Note that $W_1,\ldots, W_n$ is an independent and identically distributed collection of Bernoulli variables with parameter $\mu = P^X (V)  $.
 We have the following inequality for any \( \theta \in (0,1) \)
\[
\PP \left(\sum_{i=1}^n W_i \leq (1 - \theta) n \mu  \right) \leq e^{-\theta^2 n \mu /2}.
\]  
Furthermore, for any \( \delta \in (0,1) \), we have  
\[
\PP  \left( \frac{1}{n} \sum_{i=1}^n W_i \leq \left[1 - \sqrt{\frac{2\log(1/\delta)}{ n  \mu }} \right]\mu  \right) \leq \delta.
\]  
Since $n\mu  \geq 8 \log(1/\delta)$, we obtain with probability at least $1-\delta$,
$ \sum_{i=1}^n W_i >  n \mu / 2$ which is \eqref{eq2}.
\end{proof}

This lemma is useful for proving Propositions \ref{uniform tree not regular} and \ref{centered tree not regular}.
\begin{lemma}\label{lemme moment}
   Let \( M, (M_i)_{i=1,\ldots, N} \) be a collection of independent and identically distributed random variables such that $\mathbb E[M^4] <\infty$. It holds \[\mathbb{E}\left[ \left( \sum_{i=1}^{N} M_i \right)^4\right] = N \mathbb{E}(M^4) + 3N (N-1) \mathbb{E}(M^2)^2. \]
\end{lemma}

\begin{proof}
We have \[   \left( \sum_{i=1}^{N} M_i \right)^4 = \sum_{i,p,q,r=1}^{N} M_i M_p M_q M_r.\]
Since the \( M_i \) are independent and centered, the expectation of each product \( M_i M_p M_q M_r \) will be zero if at least one of the indices is distinct. This restricts the analysis to cases where all indices are identical or two pairs of indices are identical. If all indices are identical, i.e. \( i = p = q = r \), then the expectation of \( M_i^4 \) contributes to the sum: $ \sum_{i=1}^{N} \mathbb{E}(M_i^4) = N \, \mathbb{E}(M^4)$. When two indices are identical and the other two are also identical, i.e. \( i = p \neq q = r \), we get a product of the form \( M_i^2 M_q^2 \). We have 3 choices either $i$ is equal to \( p \), \( q \), or \( r \)). The remaining two indices must necessarily be equal. This yields: $ 3 \sum_{i \neq q} \mathbb{E}(M_i^2) \, \mathbb{E}(M_q^2) = 3N (N - 1) \, \mathbb{E}(M^2)^2.$ Combining the two terms, this proves the result.
\end{proof}

\bibliography{biblio}

\begin{thebibliography}{GKKW06}

\bibitem[AG14]{arlot2014analysis}
Sylvain Arlot and Robin Genuer.
\newblock Analysis of purely random forests bias.
\newblock {\em arXiv preprint arXiv:1407.3939}, 2014.

\bibitem[And66]{anderson}
T.~W. Anderson.
\newblock {\em Some nonparametric multivariate procedures based on statistically equivalent blocks}.
\newblock Multivariate Analysis (P. R. Krishnaiah, ed), 5-27, Academic Press, New York., 1966.

\bibitem[AST93]{anthony1993result}
Martin Anthony and John Shawe-Taylor.
\newblock A result of {V}apnik with applications.
\newblock {\em Discrete Appl. Math.}, 47(3):207--217, 1993.

\bibitem[BD15]{biau2015lectures}
G\'{e}rard Biau and Luc Devroye.
\newblock {\em Lectures on the nearest neighbor method}.
\newblock Springer Series in the Data Sciences. Springer, Cham, 2015.

\bibitem[BDL08]{JMLR:v9:biau08a}
G{{\'e}}rard Biau, Luc Devroye, and G{{\'a}}bor Lugosi.
\newblock Consistency of random forests and other averaging classifiers.
\newblock {\em Journal of Machine Learning Research}, 9(66):2015--2033, 2008.

\bibitem[BFSO84]{breiman1984classification}
Leo Breiman, Jerome Friedman, Charles~J Stone, and RA~Olshen.
\newblock {\em Classification and Regression Trees}.
\newblock CRC Press, 1984.

\bibitem[Bia12]{biau2012analysis}
G{\'e}rard Biau.
\newblock Analysis of a random forests model.
\newblock {\em The Journal of Machine Learning Research}, 13(1):1063--1095, 2012.

\bibitem[BLM13]{boucheron2013concentration}
St\'{e}phane Boucheron, G\'{a}bor Lugosi, and Pascal Massart.
\newblock {\em Concentration inequalities. A nonasymptotic theory of independence}.
\newblock Oxford University Press, Oxford, 2013.

\bibitem[Bre00]{breiman2000some}
Leo Breiman.
\newblock Some infinity theory for predictor ensembles.
\newblock Technical report, Citeseer, 2000.

\bibitem[Bre01]{breiman2001random}
Leo Breiman.
\newblock Random forests.
\newblock {\em Machine learning}, 45:5--32, 2001.

\bibitem[BS16]{biau2016random}
G{\'e}rard Biau and Erwan Scornet.
\newblock A random forest guided tour.
\newblock {\em Test}, 25:197--227, 2016.

\bibitem[CG06]{cerou2006nearest}
Fr{\'e}d{\'e}ric C{\'e}rou and Arnaud Guyader.
\newblock Nearest neighbor classificationin infinite dimension.
\newblock {\em ESAIM: Probability and Statistics}, 10:340--355, 2006.

\bibitem[CKT22]{cattaneo2022pointwise}
Matias~D Cattaneo, Jason~M Klusowski, and Peter~M Tian.
\newblock On the pointwise behavior of recursive partitioning and its implications for heterogeneous causal effect estimation.
\newblock {\em arXiv preprint arXiv:2211.10805}, 2022.

\bibitem[Cov68]{cover1968estimation}
T~Cover.
\newblock Estimation by the nearest neighbor rule.
\newblock {\em IEEE Trans. Inform. Theory}, 14(1):50--55, 1968.

\bibitem[CVFL22]{chi2022asymptotic}
Chien-Ming Chi, Patrick Vossler, Yingying Fan, and Jinchi Lv.
\newblock Asymptotic properties of high-dimensional random forests.
\newblock {\em The Annals of Statistics}, 50(6):3415--3438, 2022.

\bibitem[DGL96]{devroye96probabilistic}
Luc Devroye, L{\'a}szl{\'o} Gy{\"o}rfi, and G{\'a}bor Lugosi.
\newblock {\em A probabilistic theory of pattern recognition}, volume~31.
\newblock Springer Science \& Business Media, 1996.

\bibitem[EM00]{einmahl2000}
Uwe Einmahl and David~M. Mason.
\newblock An empirical process approach to the uniform consistency of kernel-type function estimators.
\newblock {\em J. Funct. Anal.}, 13(1):1--37, 2000.

\bibitem[FH51]{fix1989discriminatory}
Evelyn Fix and Joseph~Lawson Hodges.
\newblock Discriminatory analysis. nonparametric discrimination: Consistency properties.
\newblock {\em Int. Stat. Rev.}, 57(3):238--247, 1951.

\bibitem[GG02]{gine+g:02}
Evarist Gin{\'e} and Armelle Guillou.
\newblock Rates of strong uniform consistency for multivariate kernel density estimators.
\newblock {\em Ann. Inst. Henri Poincar\'{e} Probab. Stat.}, 38(6):907--921, 2002.
\newblock En l'honneur de J. Bretagnolle, D. Dacunha-Castelle, I. Ibragimov.

\bibitem[GKKW06]{gyorfi2006distribution}
L{\'a}szl{\'o} Gy{\"o}rfi, Michael Kohler, Adam Krzyzak, and Harro Walk.
\newblock {\em A distribution-free theory of nonparametric regression}.
\newblock Springer Science \& Business Media, 2006.

\bibitem[GKM16]{gadat2016classification}
S\'{e}bastien Gadat, Thierry Klein, and Cl\'{e}ment Marteau.
\newblock Classification in general finite dimensional spaces with the {$k$}-nearest neighbor rule.
\newblock {\em Ann. Statist.}, 44(3):982--1009, 2016.

\bibitem[GKN14]{NIPS2014_8c19f571}
Lee-Ad Gottlieb, Aryeh Kontorovich, and Pinhas Nisnevitch.
\newblock Near-optimal sample compression for nearest neighbors.
\newblock In Z.~Ghahramani, M.~Welling, C.~Cortes, N.~Lawrence, and K.Q. Weinberger, editors, {\em Advances in Neural Information Processing Systems}, volume~27. Curran Associates, Inc., 2014.

\bibitem[GO80]{gordon1980consistent}
Louis Gordon and Richard~A Olshen.
\newblock Consistent nonparametric regression from recursive partitioning schemes.
\newblock {\em Journal of Multivariate Analysis}, 10(4):611--627, 1980.

\bibitem[GW21]{gyorfi2021universal}
L{\'a}szl{\'o} Gy{\"o}rfi and Roi Weiss.
\newblock Universal consistency and rates of convergence of multiclass prototype algorithms in metric spaces.
\newblock {\em Journal of Machine Learning Research}, 22(151):1--25, 2021.

\bibitem[HKSW21]{hanneke2021universal}
Steve Hanneke, Aryeh Kontorovich, Sivan Sabato, and Roi Weiss.
\newblock Universal bayes consistency in metric spaces.
\newblock {\em The Annals of Statistics}, 49(4):2129--2150, 2021.

\bibitem[HR90]{hagerup1990guided}
Torben Hagerup and Christine R{\"u}b.
\newblock A guided tour of chernoff bounds.
\newblock {\em Information processing letters}, 33(6):305--308, 1990.

\bibitem[Jia19]{jiang2019non}
Heinrich Jiang.
\newblock Non-asymptotic uniform rates of consistency for $k$-{NN} regression.
\newblock In {\em AAAI proceedings}, volume~33, pages 3999--4006, 2019.

\bibitem[Klu21]{klusowski2021sharp}
Jason Klusowski.
\newblock Sharp analysis of a simple model for random forests.
\newblock In {\em International Conference on Artificial Intelligence and Statistics}, pages 757--765. PMLR, 2021.

\bibitem[KSW17]{NIPS2017_934815ad}
Aryeh Kontorovich, Sivan Sabato, and Roi Weiss.
\newblock Nearest-neighbor sample compression: Efficiency, consistency, infinite dimensions.
\newblock In I.~Guyon, U.~Von Luxburg, S.~Bengio, H.~Wallach, R.~Fergus, S.~Vishwanathan, and R.~Garnett, editors, {\em Advances in Neural Information Processing Systems}, volume~30. Curran Associates, Inc., 2017.

\bibitem[KW23]{kerem2023error}
Omer Kerem and Roi Weiss.
\newblock On error and compression rates for prototype rules.
\newblock In {\em Proceedings of the AAAI Conference on Artificial Intelligence}, volume~37, pages 8228--8236, 2023.

\bibitem[LN96]{lugosinobel}
Gabor Lugosi and Andrew Nobel.
\newblock Consistency of data of data-driven histogram methods for density estimation and classification.
\newblock {\em The Annals of Statistics}, 1996.

\bibitem[LRT14]{lakshminarayanan2014mondrian}
Balaji Lakshminarayanan, Daniel~M Roy, and Yee~Whye Teh.
\newblock Mondrian forests: Efficient online random forests.
\newblock {\em Advances in neural information processing systems}, 27, 2014.

\bibitem[Lug02]{lugosi2002pattern}
G{\'a}bor Lugosi.
\newblock Pattern classification and learning theory.
\newblock In {\em Principles of nonparametric learning}, pages 1--56. Springer, 2002.

\bibitem[MGS17]{mondrian}
Jaouad Mourtada, Stéphane Gaïffas, and Erwan Scornet.
\newblock Universal consistency and minimax rates for online mondrian forests.
\newblock {\em Advances in Neural Information Processing Systems 30}, 2017.

\bibitem[MGS19]{minimaxmondrian}
Jaouad Mourtada, Stéphane Gaïffas, and Erwan Scornet.
\newblock Minimax optimal rates for mondrian trees and forests.
\newblock {\em Annals of Statistics}, 2019.

\bibitem[MW24]{mazumder2024convergence}
Rahul Mazumder and Haoyue Wang.
\newblock On the convergence of {CART} under sufficient impurity decrease condition.
\newblock {\em Advances in Neural Information Processing Systems}, 36, 2024.

\bibitem[Nad64]{nadaraya1964estimating}
Elizbar~A Nadaraya.
\newblock On estimating regression.
\newblock {\em Theory Probab. Appl.}, 9(1):141--142, 1964.

\bibitem[Nob96]{nobel}
Andrew Nobel.
\newblock Histogram regression estimation using data-dependent partitions.
\newblock {\em The Annals of Statistics}, 1996.

\bibitem[Por21]{portier2021nearest}
Fran{\c{c}}ois Portier.
\newblock Nearest neighbor process: weak convergence and non-asymptotic bound.
\newblock {\em arXiv preprint arXiv:2110.15083}, 2021.

\bibitem[RT08]{mondrianroy2008}
Daniel~M Roy and Yee~W Teh.
\newblock The {M}ondrian process.
\newblock In {\em Advances in {N}eural {I}nformation {P}rocessing {S}ystems}, pages 1377--1384, 2008.

\bibitem[SBV15]{scornet2015consistency}
Erwan Scornet, G{\'e}rard Biau, and Jean-Philippe Vert.
\newblock Consistency of random forests.
\newblock {\em Annals of Statistics}, 43(4):1716--1741, 2015.

\bibitem[Sto77]{stone1977consistent}
Charles~J Stone.
\newblock Consistent nonparametric regression.
\newblock {\em Ann. Statist.}, 5(1):595--620, 1977.

\bibitem[Sto82]{stone1982optimal}
Charles~J Stone.
\newblock Optimal global rates of convergence for nonparametric regression.
\newblock {\em The annals of statistics}, pages 1040--1053, 1982.

\bibitem[Tsy08]{tsy_08}
Alexandre~B. Tsybakov.
\newblock {\em Introduction to Nonparametric Estimation}.
\newblock Springer Publishing Company, Incorporated, 1st edition, 2008.

\bibitem[Tsy09]{tsybakov2009}
Alexandre~B. Tsybakov.
\newblock {\em Introduction to nonparametric estimation}.
\newblock Springer Series in Statistics. Springer, 2009.

\bibitem[VC15]{vapnik2015uniform}
V.~N. Vapnik and A.~Ya. Chervonenkis.
\newblock On the uniform convergence of relative frequencies of events to their probabilities.
\newblock In {\em Measures of complexity}, pages 11--30. Springer, Cham., 2015.
\newblock Reprint of Theor. Probability Appl. {16} (1971), 264--280.

\bibitem[VDVW96]{wellner1996}
Aad~W. Van Der~Vaart and Jon~A. Wellner.
\newblock {\em Weak Convergence and Empirical Processes. With Applications to Statistics}.
\newblock Springer Series in Statistics. Springer-Verlag, New York, 1996.

\bibitem[WD81]{wenocur1981some}
Roberta~S Wenocur and Richard~M Dudley.
\newblock Some special {V}apnik-{C}hervonenkis classes.
\newblock {\em Discrete Math.}, 33(3):313--318, 1981.

\bibitem[XK18]{xue2018achieving}
Lirong Xue and Samory Kpotufe.
\newblock Achieving the time of 1-nn, but the accuracy of k-nn.
\newblock In {\em International Conference on Artificial Intelligence and Statistics}, pages 1628--1636. PMLR, 2018.

\end{thebibliography}

\end{document}